\documentclass[12pt, a4paper]{amsart}
\usepackage[english]{babel}
\usepackage[utf8]{inputenc}
\usepackage{amsmath}
\usepackage{graphicx}
\usepackage[colorinlistoftodos]{todonotes}
\usepackage{amssymb,amscd,latexsym,epsfig,xy}
\usepackage[mathscr]{euscript}
\usepackage{tikz}
\usetikzlibrary{matrix,arrows,decorations.pathmorphing}
\usepackage{hyperref}
\usepackage{enumerate}
\usepackage{float}
\usepackage{tikz-cd}
\usepackage{relsize}
\usepackage{mathtools}
\usepackage[numbers]{natbib}

\newtheorem{thm}[equation]{Theorem}
\newtheorem{lem}[equation]{Lemma}
\newtheorem{cor}[equation]{Corollary}

\newtheorem{prop}[equation]{Proposition}

\newcommand\A{\mathbb A}
\newcommand\C{\mathbb C}
\newcommand\E{\mathbb E}

\newcommand\NN{\mathbb N}

\newcommand\PP{\mathbb P}
\newcommand\R{\mathbb R}
\newcommand\Q{\mathbb Q}
\newcommand\Z{\mathbb Z}
\newcommand\cA{\mathcal A}

\newcommand\cC{\mathcal C}

\newcommand\cF{\mathcal F}

\newcommand\cM{\mathcal M}

\newcommand\cO{\mathcal O}
\newcommand\cP{\mathcal P}

\newcommand\cX{\mathcal X}

\newcommand\Spec{\operatorname{Spec}\,}

\newcommand\Hom{\operatorname{Hom}}

\newcommand\vdim{\operatorname{vdim}}
\newcommand\ev{\operatorname{ev}}

\newcommand\virt{\mathrm{virt}}

\begin{document}

\title[Tropical refined curve counting from higher genera]{Tropical refined curve counting from higher genera and lambda classes
}

\author{Pierrick Bousseau}

\maketitle

\begin{abstract}
Block and G\"ottsche have defined a $q$-number refinement of counts of tropical curves in 
$\R^2$.
Under the change of variables 
$q=e^{iu}$, we show that the result 
is a generating series of higher genus
log Gromov-Witten invariants with insertion of a lambda class.
This gives a geometric interpretation of the 
Block-G\"ottsche invariants 
and makes their deformation 
invariance manifest.
\end{abstract}

\setcounter{tocdepth}{1}

\tableofcontents

\section{Introduction}
Tropical geometry gives a combinatorial way to approach
problems in complex and real 
algebraic geometry. 
An early success of this approach is Mikhalkin's correspondence theorem 
\cite{MR2137980}, proved differently
and generalized by Nishinou and Siebert
\cite{MR2259922}, between counts of 
complex algebraic curves in complex
toric surfaces and counts
with multiplicity of tropical curves in $\R^2$. 
Our main result,
Theorem \ref{main_thm1},
is an extension 
to a correspondence between some generating series of higher genus log Gromov-Witten invariants 
of toric surfaces and counts with $q$-multiplicity of tropical curves in $\R^2$.

Counts of tropical curves in $\R^2$ with $q$-multiplicity were introduced by 
Block and G\"ottsche 
\cite{MR3453390}. 
The usual multiplicity of a tropical curve is defined as a 
product of integer multiplicities attached to the vertices. 
Block and G\"ottsche remarked that one can obtain a refinement 
by replacing the multiplicity
$m$ of a vertex by 
its $q$-analogue
\[[m]_q \coloneqq \frac{q^{\frac{m}{2}} -q^{-\frac{m}{2}}}{q^{\frac{1}{2}}-q^{-\frac{1}{2}}}=
q^{-\frac{m-1}{2}}(1+q+\dots +q^{m-1}) \,.\]
The $q$-multiplicity of a tropical curve 
is then the product of the 
\mbox{$q$-multiplicities} of the vertices. 
The count with $q$-multiplicity of tropical curves
specializes for $q=1$ to the ordinary count with multiplicity.
This definition is done at the tropical level so is combinatorial 
in nature and its geometric meaning is \emph{a priori} unclear. 

Let $\Delta$ be a balanced collection of vectors in $\Z^2$
and let $n$ 
be a non-negative integer\footnote{Precise definitions are given in Section
\ref{section: precise}.}.
This determines a complex
toric surface
$X_\Delta$ and a counting problem
of virtual dimension zero for
complex algebraic curves in $X_\Delta$ of some genus
$g_{\Delta, n}$, of some class $\beta_\Delta$, 
satisfying some tangency conditions with respect to the toric boundary divisor,
and passing through $n$  points of $X_\Delta$ in general position. 
Let $N^{\Delta, n} 
\in \NN$ be the solution to this counting problem. According to 
Mikhalkin's correspondence theorem, 
$N^{\Delta, n}$ is a count with multiplicity of tropical curves 
in $\R^2$, and so it has a Block-G\"ottsche
refinement 
$N^{\Delta, n}(q) \in \NN[q^{\pm \frac{1}{2}}]$.

For every $g \geqslant g_{\Delta, n}$, we consider the same counting problem as before---same curve class, 
same tangency conditions---but for curves of genus $g$. 
The virtual dimension is now $g-g_{\Delta, n}$. 
To obtain a number, we integrate a class of degree $g-g_{\Delta,n}$, the lambda class 
$\lambda_{g-g_{\Delta,n}}$, over the virtual fundamental class of
a corresponding moduli space of stable log maps. For every 
$g \geqslant g_{\Delta, n}$, we get a log Gromov-Witten invariant $N_g^{\Delta, n} \in \Q$.

\begin{thm} \label{main_thm1}
For every $\Delta$ 
balanced collection of vectors in 
$\Z^2$,
and
for every 
non-negative integer $n$ such that 
$g_{\Delta, n} \geqslant 0$, we have the equality

\[\sum_{g \geqslant g_{\Delta ,n}} N_g^{\Delta ,n}
u^{2g-2+|\Delta|}
=N^{\Delta, n}(q) \left( (-i)
(q^{\frac{1}{2}}-q^{-\frac{1}{2}}) \right)^{2g_{\Delta,n}-2+|\Delta|}\]
of power series in $u$ with rational coefficients, 
where
\[ q=e^{iu}
=\sum_{n \geqslant 0} \frac{(iu)^n}{n!} \,,\]
and $|\Delta|$ is the cardinality of $\Delta$.
\end{thm}

\textbf{Remarks}
\begin{itemize}
\item According to Theorem
\ref{main_thm1}, the knowledge of the Block-G\"ottsche invariant 
$N^{\Delta,n }(q)$ is equivalent to the knowledge of the log 
Gromov-Witten invariants 
$N^{\Delta, n}_g$ for all 
$g \geqslant g_{\Delta, n}$.
This provides a geometric meaning to Block-G\"ottsche invariants,
independent of any choice of tropical
limit, 
making their deformation invariance manifest. 
\item Given a family
$\pi \colon \cC \rightarrow B$
of nodal curves, the Hodge bundle 
$\E$ is the rank $g$ vector bundle over $B$ whose fiber over 
$b \in B$
is the space $H^0(C_b, \omega_{C_b})$
of sections of the dualizing sheaf 
$\omega_{C_b}$ of the 
curve
$C_b=\pi^{-1}(b)$. 
The lambda classes are classically
\cite{MR717614} the Chern classes of the Hodge bundle:
\[ \lambda_j \coloneqq c_j (\E) \,.\]
The log Gromov-Witten invariants 
$N_g^{\Delta, n}$ are defined 
by an insertion of
$(-1)^{g-g_{\Delta, n}}
\lambda_{g-g_{\Delta, n}}$ to cut down the 
virtual dimension from
$g-g_{\Delta, n}$ to zero.
\item One can interpret 
Theorem \ref{main_thm1} as establishing 
integrality and positivity properties for higher genus 
log Gromov-Witten invariants of $X_\Delta$ with one lambda class 
inserted.
\item The change of variables
$q=e^{iu}$ makes the correspondence of Theorem
\ref{main_thm1} quite non-trivial.
In particular, it cannot be reduced to an easy enumerative correspondence. 
It is essential to have a virtual/non-enumerative count on the Gromov-Witten side:
for $g$ large enough, most of the contributions to $N_g^{\Delta, n}$
come from maps with contracted components.
\item In Theorem \ref{thm_fixing_points}, we present a generalization of Theorem
\ref{main_thm1} where some intersection
points with the toric boundary divisor can be fixed.
\item One could ask for a generalization of
Theorem 
\ref{main_thm1} including descendant log Gromov-Witten invariants,
i.e.\ with insertion of psi classes.
In the simplest case of a trivalent vertex with insertion of one psi class, it is possible to reproduce the numerator
$q^{\frac{m}{2}}+q^{-\frac{m}{2}}$ of the
multiplicity introduced by 
G\"ottsche and Schroeter 
\cite{gottsche2016refined} in the 
context of refined broccoli invariants,
in a way similar to how we reproduce the
numerator
$q^{\frac{m}{2}}-q^{-\frac{m}{2}}$ of the Block-G\"ottsche
multiplicity in Theorem 
\ref{main_thm1}. This will be described in
some further work.
\end{itemize}

\subsection*{Relation with previous and future works}

\subsubsection*{$q$-analogues} It is a general principle 
in mathematics, going back at least to Heine's introduction of 
$q$-hypergeometric series in 1846, that many ``classical''
notions 
have a
$q$-analogue, recovering the classical one in the 
limit $q \rightarrow 1$. 
The Block-G\"ottsche refinement of the tropical curve counts in $\R^2$ 
is clearly an example of this principle. In many other examples, 
it is well known that it is a good idea to write 
$q=e^{\hbar}$, the limit $q \rightarrow 1$ becoming the limit $\hbar \rightarrow 0$.
From this point of view, the change of variable $q=e^{iu}$ in 
Theorem \ref{main_thm1} is maybe not so 
surprising.

\subsubsection*{G\"ottsche-Shende refinement by Hirzebruch genus} \label{gs}

Whereas the specialization of 
Block-G\"ottsche 
invariants at $q=1$ recovers 
a count of complex algebraic curves, the specialization 
$q=-1$ recovers in some cases a count of real 
algebraic curves in the sense of 
Welschinger \cite{MR2198329}.
This strongly suggests a motivic 
interpretation of the 
Block-G\"ottsche invariants and 
indeed one of the original
motivations of Block and 
G\"ottsche was the fact that,
under some ampleness assumptions, the refined tropical curve counts should coincide with the refined curve counts on toric surfaces defined by G\"ottsche and Shende
\cite{MR3268777}
in terms of  
Hirzebruch genera of Hilbert schemes.
Using motivic integration,
Nicaise, Payne and Schroeter
\cite{nicaise2016tropical}
have reduced this conjecture 
to a conjecture about the motivic measure of a semialgebraic piece of the Hilbert scheme attached to a
given tropical curve.

Our approach to the Block-G\"ottsche 
refined tropical curve counting is clearly different 
from the G\"ottsche-Shende approach: we interpret the refined variable $q$ as 
coming from the resummation of a genus expansion whereas 
they interpret it as a formal parameter keeping track of the refinement
 from some Euler characteristic to some Hirzebruch genus. 

The G\"ottsche-Shende refinement makes sense for 
surfaces more general than toric ones, 
as do the higher genus log Gromov-Witten invariants with one lambda class inserted. So one might ask if 
Theorem 
\ref{main_thm1} can be extended to more general surfaces, 
as a relation between G\"ottsche-Shende refined invariants 
and generating series of higher genus log Gromov-Witten invariants.
Combining known results about
G\"ottsche-Shende refined invariants 
\cite{MR3268777} 
and higher genus Gromov-Witten invariants,
\cite{MR2746343},
\cite{bryan2015curve},
one can show that it is indeed the case for K3 and abelian surfaces. 
In particular, Theorem
\ref{main_thm1} is not an isolated
fact but part of a family of similar results. 
The case of a log Calabi-Yau surface obtained as complement of a smooth anticanonical divisor 
in a del Pezzo surface, and its relation with, in physics terminology, a worldsheet
definition of the refined topological 
string of local del Pezzo 3-folds, will be discussed in a future work.

\subsubsection*{Tropical vertex}
Filippini and Stoppa \cite{MR3383167} 
have related refined tropical curve counting to
the $q$-version of the tropical vertex of \cite{MR2667135}, i.e.\
of the 2-dimensional Kontsevich-Soibelman scattering diagram.
Combined with the main result of the present paper, we get an enumerative interpretation of the $q$-version of the tropical vertex. Details will be given in a separate publication
\cite{bousseau2018quantum_tropical}. With this enumerative 
interpretation, it is possible to give an higher genus generalization of the Gross-Hacking-Keel
\cite{MR3415066} 
mirror symmetry construction for log Calabi-Yau surfaces \cite{bousseau2018quantum}.

Using the connection with the $q$-version of the tropical vertex, Filippini and Stoppa \cite{MR3383167} have related refined tropical curve counting to 
refined Donaldson-Thomas theory of quivers. 
This story was the initial
motivation for the work eventually leading to the present paper. 
Applications of the present paper in this context will be discussed elsewhere.

\subsubsection*{MNOP}
The change of variables $q=e^{iu}$
is reminiscent of the MNOP
Gromov-Witten/ Donaldson-Thomas (DT) correspondence on 3-folds \cite{MR2264664}, \cite{MR2264665}.
The log Gromov-Witten invariants 
$N_g^{\Delta, n}$ can be rewritten as 
$\C^*$-equivariant 
log Gromov-Witten invariants of the 
3-fold $X_\Delta \times \C$, where 
$\C^*$ acts by scaling on $\C$,
see Lemma 7 of \cite{MR2746343}.
If a log DT theory and a log MNOP
correspondence were developed, this would 
predict that the generating series of 
$N_g^{\Delta, n}$ is a rational function in 
$q=e^{iu}$, which is indeed true by 
\mbox{Theorem 
\ref{main_thm1}.} But it would not be enough to imply \mbox{Theorem \ref{main_thm1}} because 
the relation between log DT invariants and Block-G\"ottsche invariants is \emph{a priori} unclear.
In fact, the G\"ottsche-Shende conjecture and the result of Filippini and Stoppa suggest 
that Block-G\"ottsche invariants are refined DT invariants whereas the MNOP correspondence
involves unrefined DT invariants. This topic will be discussed in more details elsewhere.

\subsubsection*{BPS integrality} When the log Gromov-Witten invariants of $X_\Delta \times \C$ coincide with ordinary Gromov-Witten invariants of 
$X_\Delta \times \C$, which is probably the case 
if $|v|=1$ for every $v \in \Delta$ and if 
the toric boundary divisor of $X_\Delta$ is
positive enough, then the integrality  
implied by Theorem 
\ref{main_thm1} coincides with the BPS integrality predicted by Pandharipande in \cite{MR1729095},
and proved via symplectic methods by Zinger in 
\cite{MR2822239},
for generating series of Gromov-Witten invariants
of a 3-fold and of curve class intersecting positively the anticanonical divisor.

\subsubsection*{Mikhalkin refined real count}
Mikhalkin has given in  
\cite{mikhalkin2015quantum} an interpretation of some particular Block-G\"ottsche 
invariants in terms of counts of real curves. 
We do not understand the relation with our approach in terms 
of higher genus log Gromov-Witten invariants. 
We merely remark that both for us and for Mikhalkin, it is the numerator of the Block-G\"ottsche 
multiplicities which appears naturally.

\subsubsection*{Parker theory of exploded manifolds}
\label{parker}
This paper owes a great intellectual debt towards the 
paper  \cite{parker2016three} of Brett Parker, 
where a correspondence theorem between tropical curves in $\R^3$ and some generating series 
of curve counts in exploded versions of toric 3-folds is proved.
Indeed, a conjectural version of Theorem \ref{main_thm1} was known to the author 
around April 2016\footnote{And was for example presented at the Workshop: Curves on surfaces and 3-folds, EPFL,
Lausanne,
21 June 2016.} 
but it was only after the
appearance of  \cite{parker2016three} in August 2016 
that it became clear that this result should be provable with existing technology. In
particular,  the idea to reduce the amount of explicit
computations by exploiting the consistency of some gluing formula (see \mbox{Section 
\ref{section: vertex}})
follows \cite{parker2016three}. 
In the present paper, we use the theory of log Gromov-Witten invariants 
because of the algebraic bias of the author, 
but it should be possible to write a version in the language of exploded manifolds.

\subsection*{Plan of the paper}
In Section \ref{section: precise}, we fix our notations and
we  
describe precisely the objects involved in the formulation of 
Theorem \ref{main_thm1}.
In Section \ref{section_lambda}, 
we review some gluing 
and vanishing properties of the lambda classes.

The next five Sections form the proof of 
Theorem \ref{main_thm1}.

The first step of the proof, described in Section 
\ref{section: decomposition}, is 
an application of the decomposition formula of 
Abramovich, Chen, Gross and Siebert 
\cite{abramovich2017decomposition} to the toric degeneration of Nishinou, Siebert \cite{MR2259922}.
This gives a way to write our log Gromov-Witten invariants as a sum of contributions indexed by 
tropical curves.

In the second step of the proof,
described in Sections
\ref{section_statement_gluing}
and \ref{proof_gluing}, we prove a
gluing formula 
which 
gives a way to write the contribution of a tropical curve as a
product of contributions of its vertices.
Here, gluing and vanishing properties of the lambda classes reviewed in
Section 
\ref{section_lambda}, combined with 
a structure result for non-torically transverse stable log maps  proved 
in Section
\ref{section_transverse}, 
play an essential role. 
In particular, we only have to glue torically transverse stable log maps 
and we don't need to worry about the technical issues making the general 
gluing formula in log Gromov-Witten theory difficult (see 
Abramovich, Chen, Gross, Siebert
\cite{abramovich2017punctured}).

After the decomposition and gluing steps, 
what remains to do is to compute the contribution to the log Gromov-Witten invariants 
of a tropical curve with a single trivalent vertex. 
The third and final step of the proof of 
\mbox{Theorem \ref{main_thm1}}, carried out in Section 
\ref{section: vertex}, is the explicit evaluation of this vertex contribution.
Consistency of the gluing formula leads to non-trivial 
relations between these vertex contributions, which enable us 
to reduce the problem to  
particularly simple vertices. 
The
contribution of these simple vertices is computed explicitly by reduction 
to Hodge integrals previously computed by 
Bryan and Pandharipande \cite{MR2115262}
and this ends the proof of Theorem \ref{main_thm1}.

In Appendix \ref{section_ example}, we present for the sake of concreteness 
an explicit example.

\textbf{Acknowledgements}
I have already mentioned
that this paper
would probably not exist without the paper 
\cite{parker2016three} of Brett Parker.

I would like to thank my supervisor 
Richard Thomas for continuous support and innumerous discussions,
suggestions and corrections. 
I thank Mark Gross and Rahul Pandharipande 
for invitations to seminars and useful comments and 
discussions. 
I thank Navid Nabijou and Dan Pomerleano 
for useful discussions about relative and log Gromov-Witten invariants. 
I thank Vivek Shende for a discussion about the analogy with the case of K3 surfaces.
I thank the referee for many corrections and suggestions for improvement.

This work is supported by the EPSRC 
award 1513338, Counting curves in algebraic geometry, 
Imperial College London, and has benefited from the
EPRSC [EP/L015234/1],
EPSRC Centre for Doctoral Training in Geometry and Number Theory 
(The London School of Geometry and Number Theory), University College London.

\subsection*{Notation} Throughout we keep largely to the following notation.

\begin{table}[ht]
\centering 
\begin{tabular}{c l c }
$i$ & the standard square root of $-1$ in 
$\C$ 
\\
$q$ & the formal refined variable in Block-G\"ottsche invariants
\\
$u$ & a formal variable keeping track of the genus in generating 
\\
& series of Gromov-Witten invariants, related to $q$ by $q=e^{iu}$
\\
$A_*$ & a Chow group
\\
$A^*$ & an operatorial cohomology Chow group, see \cite{MR1644323}
\\
$\Delta$ & a balanced collection of vectors in $\Z^2$, of cardinality $|\Delta|$
\\
$X_\Delta$ & the toric surface defined by 
$\Delta$
\\
$\beta_\Delta$ & the curve class defined by 
$\Delta$
\\
$n$ & a number of points in $(\C^{*})^2$
or $\R^2$
\\
$g_{\Delta, n}$ & the integer $n+1-|\Delta|$
\\ 
$P$ & a set of $n$ points $P_j$
in 
$(\C^{*})^2$
\\
$p$ & a set of $n$ points $p_j$ in 
$\R^2$
\\
$\Gamma$ & a graph, often source of a parametrized tropical curve
\\
$V(\Gamma)$ & the set of vertices of $\Gamma$, of cardinality $|V(\Gamma)|$
\\
$E(\Gamma)$ & the set of edges of $\Gamma$, of cardinality $|E(\Gamma)|$
\\
$E_f(\Gamma)$ & the set of bounded edges of 
$\Gamma$, of cardinality $|E_f(\Gamma)|$
\\
$E_\infty(\Gamma)$ & the set of bounded edges of 
$\Gamma$, of cardinality $|E_\infty(\Gamma)|$
\\
$h$ & a parametrized tropical curve
\\ $m(V)$ & the multiplicity of a vertex $V$
\\ $w(E)$ & the weight of an edge $E$
\\

$\overline{M}_{g, n, \Delta}$ & a moduli space of genus $g$ stable log maps
\\
$N_g^{\Delta, n}$ & a genus $g$
log Gromov-Witten invariant
\\
$N_{\mathrm{trop}}^{\Delta, n}(q)$ & a refined tropical curve count
\\

$T_{\Delta, p}$ & a finite set of
genus $g_{\Delta, n}$ parametrized tropical curves
\\
$T^g_{\Delta, p}$ & a finite set of
genus $g$ parametrized tropical curves
\\
$\overline{\cM}$ & a monoid
\\
$\mathrm{pt}_{\overline{\cM}}$ & the log point 
of ghost monoid $\overline{\cM}$
\\
$\Sigma$ & the tropicalization functor
\\
$X_0$ & the central fiber of a 
toric degeneration of 
$X_\Delta$ 
\\
$P^0$ & a set of $n$ points $P^0_j$ in 
$X_0$, degeneration of $P$
\\
$X_{\Delta_V}$ & an irreducible component of $X_0$
\\
\end{tabular}
\end{table}
\begin{table}[ht]
\centering 
\begin{tabular}{c l c }
$N_{g,h}^{\Delta, n}$ & a genus $g$
log Gromov-Witten invariant marked by $h$
\\
$N_{g,V}^{1,2}$ & a genus $g$
log Gromov-Witten invariant attached to a
vertex $V$ 
\\ & with a preferred choice of edges
\\
$N_{g,V}$ & a genus $g$
log Gromov-Witten invariant attached to a
vertex  $V$
\\
$F_V(u)$ & a generating series of 
log Gromov-Witten invariants
\\& attached to a
vertex  $V$
\\
$F_m(u)$ & a generating series of 
log Gromov-Witten invariants
\\& attached to a
vertex of multiplicity $m$
\end{tabular}
\end{table}

\newpage

\section{Precise statement of the main result}
\label{section: precise}

\subsection{Toric geometry} \label{notations}
Let $\Delta$ be a balanced collection of vectors in 
$\Z^2$, i.e.\ a finite collection of vectors in $\Z^2 - \{0\}$
summing to zero\footnote{A given element of $\Z^2 -\{0\}$ can appear
several times in $\Delta$.
Here we follow the notation used by 
Itenberg and Mikhalkin in \cite{MR3142257}.}.
Let $|\Delta|$ be the cardinality of 
$\Delta$.
For 
$v \in \Z^2-\{0\}$, let $|v|$ the divisibility of 
$v$ in $\Z^2$, i.e.\ the largest positive integer 
$k$ such that we can write $v=kv'$ with $v' \in \Z^2$. 
Then the balanced collection
$\Delta$ defines the following data by standard toric geometry. 
\begin{itemize}
\item A projective\footnote{This is true only if the elements in $\Delta$ are not all collinear. 
If they are, we replace $X_\Delta$ by a toric compactification 
whose choice will be irrelevant for our purposes.} toric surface $X_\Delta$
over $\C$, whose fan has
rays $\R_{\geqslant 0}v$ generated by the vectors $v \in \Z^2-\{0\}$
contained in $\Delta$.
We denote $\partial X_\Delta$
the toric boundary divisor of $X_\Delta$.
\item A curve class $\beta_\Delta$ on $X_\Delta$, whose 
polytope is dual to $\Delta$. If $\rho$ is a ray in the fan of 
$X_\Delta$, we write $D_\rho$ for the prime toric divsisor of 
$X_\Delta$ dual to $\rho$ and $\Delta_\rho$ the set of elements 
$v \in \Delta$ such that $\R_{\geqslant 0} v=\rho$. Then we have
\[ \beta_\Delta.D_{\rho} =\sum_{v \in \Delta_\rho} |v| \,,\] 
and these intersection numbers
uniquely determine $\beta_\Delta$. The total intersection
number of $\beta_\Delta$ with the toric boundary divisor $\partial X_\Delta$
is given by $$\beta_\Delta.(-K_{X_\Delta})=\sum_{v \in \Delta} |v|\,.$$
\item Tangency conditions for curves of class $\beta_\Delta$
with respect to the toric boundary divisor of $X_\Delta$.
We say that a curve $C$ is of type $\Delta$
if it is of class 
$\beta_\Delta$ and
if for every ray $\rho$ in the fan of 
$X_\Delta$, the curve $C$ intersects $D_\rho$ in 
$|\Delta_\rho|$ points  with 
multiplicities $|v|$,  $v \in \Delta_\rho$.
Similarly, we have a notion 
of stable log map of type $\Delta$.

\item An asymptotic form for a parametrized tropical curve
$h \colon \Gamma \rightarrow \R^2$ in 
$\R^2$. 
We say that a parametrized tropical curve in $\R^2$ is of type $\Delta$
if it has $|\Delta|$ unbounded edges, with 
directions $v$ and with weights 
$|v|$, $v \in \Delta$.
\end{itemize}

\subsection{Log Gromov-Witten invariants} \label{log_GW}

The moduli space of $n$-pointed genus $g$
stable maps to $X_\Delta$ of class 
$\beta_\Delta$ intersecting 
properly the toric boundary divisor 
$\partial X_\Delta$ with tangency conditions prescribed
by $\Delta$ is not proper: a limit of curves intersecting  $\partial X_\Delta$
properly does not necessarily intersect 
$\partial X_\Delta$ properly.
A nice compactification of this space is obtained by considering
stable log maps. The idea is to allow maps intersecting 
$\partial X_\Delta$ non-properly, 
but to remember some additional information under the form of log structures, 
which give a way to make sense of tangency conditions even for non-proper intersections. 
The theory of stable log maps has been developed by Gross and Siebert \cite{MR3011419},
and Abramovich and Chen
\cite{MR3224717}, \cite{MR3257836}.
By stable log maps, we always mean basic stable log maps in the sense of
\cite{MR3011419}. We refer to Kato 
\cite{MR1463703} for elementary notions of log geometry.

We consider the toric divisorial log structure on 
$X_\Delta$ and use it to view $X_\Delta$ as a log scheme.
Let $\overline{M}_{g,n, \Delta}$ be the moduli space of 
$n$-pointed
genus $g$
stable log maps 
to $X_\Delta$ of type $\Delta$. 
By $n$-pointed, we mean that the source curves are equipped with $n$ marked points 
\emph{in addition} to the marked points keeping track of the tangency conditions 
with respect to the toric boundary divisor. We consider that the latter 
are notationally already included in $\Delta$.

By the work of Gross, Siebert 
\cite{MR3011419}
and Abramovich, Chen
\cite{MR3224717}, \cite{MR3257836}, 
$\overline{M}_{g,n, \Delta}$
is a proper Deligne-Mumford stack\footnote{Moduli spaces of stable log maps have a natural structure of log stack.
The structure of log stack 
is particularly important to treat correctly evaluation morphisms in log Gromov-Witten theory in general, see
\cite{abramovich2010evaluation}. In this paper, we always consider these moduli spaces as stacks over the category of schemes, not as log stacks, and we will always work with naive evaluation morphisms between stacks, not log stacks. This will be enough for us. See the remark at the end of Section \ref{subsection_toric_degeneration} for some justification.} 
of virtual dimension 
\[\vdim \overline{M}_{g,n,\Delta} = g-1+n+\beta_\Delta.(-K_{X_\Delta})- \sum_{v \in \Delta} (|v|-1)=g-1+n+|\Delta|  \,,\]
and it admits a virtual fundamental class 
$$[\overline{M}_{g,n,\Delta}]^{\virt} 
\in A_{\vdim \overline{M}_{g,n,\Delta}} (\overline{M}_{g,n,\Delta},\Q)\,.$$
The problem of counting
$n$-pointed genus $g$ curves passing though 
$n$ fixed points has virtual dimension zero if 
\[\vdim \overline{M}_{g,n,\Delta}=2n \,,\]
i.e.\ if the genus $g$ is equal to
$$g_{\Delta,n} \coloneqq n+1-|\Delta| \,.$$
In this case, the corresponding count of curves is given by 
\[ N^{\Delta,n} \coloneqq \left\langle  \tau_0(\mathrm{pt})^n \right\rangle_{g_{\Delta,n},n,\Delta}
\coloneqq \int_{[\overline{M}_{g_{\Delta ,n},n,\Delta}]^{\virt}} \prod_{j=1}^n \ev_j^{*}(\mathrm{pt})\,,\]
where $\mathrm{pt} \in A^2(X_\Delta)$ is the class of a point and 
$\ev_j$ is the evaluation map at the 
$j$-th marked point.

According to Mandel and Ruddat \cite{mandel2016descendant},
Mikhalkin's correspondence theorem can be reformulated in terms of these log
Gromov-Witten invariants.
Our refinement of the correspondence theorem will involve 
curves of genus $g \geqslant g_{\Delta, n}$.

For $g > g_{\Delta, n}$, inserting $n$
points is no longer enough to cut down the 
virtual dimension to zero.
The idea is to consider the Hodge bundle 
$\E$ over $\overline{M}_{g,n,\Delta}$.
If $\pi \colon \cC \rightarrow \overline{M}_{g,n,\Delta}$ 
is the universal curve, of relative dualizing\footnote{The dualizing line bundle of a nodal curve coincides with the log cotangent bundle
up to some twist by marked points 
and so is a completely natural object from the point of view of log geometry.}
 sheaf 
$\omega_\pi$, then 
\[\E \coloneqq \pi_* \omega_\pi\]
is a rank $g$ vector bundle over $\overline{M}_{g,n,\Delta}$. 
The Chern classes of the Hodge bundle are classically 
\cite{MR717614} 
called the lambda classes and denoted as
\[ \lambda_j \coloneqq c_j(\E) \,,\]
for $j=0,\dots,g$.
Because the virtual dimension of 
$\overline{M}_{g,n,\Delta}$ is given by 
\[\vdim \overline{M}_{g,n,\Delta}=g-g_{\Delta,n} +2n \,,\]
inserting the lambda class $\lambda_{g-g_{\Delta, n}}$
and $n$ points will cut down the virtual dimension to zero, so
it is natural to consider the log Gromov-Witten invariants
with one lambda class inserted
\[N^{\Delta,n}_g
\coloneqq \langle (-1)^{g-g_{\Delta,n}} \lambda_{g-g_{\Delta,n}} \tau_0(\mathrm{pt})^n \rangle_{g,n,\Delta} \]
\[\coloneqq \int_{[\overline{M}_{g,n,\Delta}]^{\virt}}
(-1)^{g-g_{\Delta,n}}
\lambda_{g-g_{\Delta,n}} \prod_{j=1}^n \text{ev}_j^{*}(\mathrm{pt})\,.\]
Our refined correspondence result, Theorem 
\ref{main_thm}, gives an interpretation of the generating 
series of these invariants in terms of refined tropical curve counting.

\subsection{Tropical curves} \label{tropical}

We refer to Mikhalkin \cite{MR2137980}, 
Nishinou, Siebert \cite{MR2259922}, 
Mandel, Ruddat \cite{mandel2016descendant},
and Abramovich, Chen, Gross, Siebert
\cite{abramovich2017decomposition}
for basics on tropical curves.
Each of these references uses a slightly different notion of parametrized tropical curve. 
We will use a variant of \cite{abramovich2017decomposition},
Definition 2.5.3,
because it is the one which is the most directly related to 
log geometry. It is easy to go from one to the other.

For us, a graph $\Gamma$ has a finite set 
$V(\Gamma)$ of vertices, a finite set $E_f(\Gamma)$ 
of bounded edges connecting pairs of vertices and a finite set 
$E_\infty (\Gamma)$ of legs attached to vertices that we view as unbounded edges. 
By edge, we refer to a bounded or unbounded edge.
We will always consider connected graphs.
 
A parametrized tropical curve $h \colon \Gamma \rightarrow \R^2$ 
is the following data:
\begin{itemize}
\item A non-negative integer $g(V)$ for each vertex $V$, called the genus of $V$.
\item A bijection of the set $E_\infty(\Gamma)$
of unbounded edges with 
\[ \{ 1, \dots, |E_\infty(\Gamma)| \} \,,\] 
where 
$|E_\infty(\Gamma)|$ is the cardinality of $E_\infty(\Gamma)$.
\item A vector $v_{V,E} \in \Z^2$ for every vertex $V$ 
and $E$ an edge adjacent to $V$. If $v_{V,E}$
is not zero, the divisibility $|v_{V,E}|$ of $v_{V,E}$ in $\Z^2$ 
is called the weight of $E$ and is denoted $w(E)$.
We require that $v_{V,E} \neq 0$ if $E$ is unbounded and 
that for every vertex $V$, the following balancing condition is 
satisfied:
\[\sum_E v_{V,E} =0 \,,\]
where the sum is over the edges $E$ adjacent to 
$V$.
In particular, the collection $\Delta_V$
of non-zero vectors $v_{\Delta,E}$ for $E$ adjacent to $V$ 
is a balanced collection as in Section \ref{notations}.
\item A non-negative real number $\ell(E)$ for every bounded
edge of $E$, called the length of $E$.
\item A proper map $h \colon \Gamma \rightarrow \R^2$
such that 
\begin{itemize}
\item If $E$ is a bounded edge connecting the vertices $V_1$ and $V_2$, 
then $h$ maps $E$ affine linearly on the line segment connecting 
$h(V_1)$ and $h(V_2)$, and $h(V_2)-h(V_1) = \ell(E)v_{V_1,E}$.
\item If $E$ is an unbounded edge of vertex $V$, then 
$h$ maps $E$ affine linearly to the ray 
$h(V)+\R_{\geqslant 0} v_{V,E}$.
\end{itemize}
\end{itemize}
The genus $g_h$ of a parametrized tropical 
curve $h \colon \Gamma \rightarrow \R^2$ 
is defined by
\[g_h \coloneqq g_\Gamma + \sum_{V \in V(\Gamma)} g(V) \,,\]
where $g_\Gamma$ is the genus of the graph 
$\Gamma$.

We fix $\Delta$ a balanced collection of vectors in $\Z^2$, 
as in Section \ref{notations}, and we fix a bijection of $\Delta$ with 
$\{1, \dots, |\Delta| \}$. 
We say that a 
parametrized tropical curve 
$h \colon \Gamma \rightarrow \R^2$ is of type $\Delta$
if there exists a bijection between 
$\Delta$ and $\{ v_{V,E} \}_{E \in E_\infty (\Gamma)}$
compatible with the fixed bijections to 
\[ \{ 1,\dots, |\Delta| \} = \{ 1, \dots, |E_\infty (\Gamma)|\} \,.\]
Remark that 
\[\sum_{E \in E_\infty(\Gamma)} v_{V,E}=0\]
by the balancing condition.

We say that a parametrized tropical curve 
$h \colon \Gamma \rightarrow \R^2$ is $n$-pointed if 
we have chosen a distribution of the labels $1, \dots, n$ over the vertices of 
$\Gamma$, a vertex having the possibility to have several labels. 
Vertices without any label are said to be unpointed whereas those with labels are said to be pointed.
For $j=1, \dots, n$, let $V_j$ be the pointed vertex having the label $j$. 
Let $p=(p_1, \dots, p_n)$ be a configuration of $n$ points in $\R^2$. 
We say that a $n$-pointed parametrized tropical curve 
$h \colon \Gamma \rightarrow \R^2$ passes through 
$p$ if $h(V_j)=p_j$ for every $j=1, \dots, n$.
We say that a $n$-pointed parametrized tropical curve 
$h \colon \Gamma \rightarrow 
\R^2$ passing through $p$ is rigid if it is not contained in a non-trivial family 
of $n$-pointed parametrized 
tropical curves passing through 
$p$ of the same combinatorial type.

\begin{prop} \label{prop_good}
For every 
balanced collection
$\Delta$ of vectors in $\Z^2$, and $n$
a
non-negative integer
such that $g_{\Delta, n} \geqslant 0$, 
there exists an open dense subset 
$U_{\Delta, n}$ of $(\R^2)^n$ such that
if $p=(p_1, \dots, p_n) \in U_{\Delta, n}$ then $p_j \neq p_k$ for 
$j \neq k$ and if $h \colon \Gamma \rightarrow \R^2$
is a rigid\footnote{Here, the rigidity 
assumption is only necessary to forbid contracted edges.
It happens to be the natural 
assumption in the general form of the decomposition formula of \cite{abramovich2017decomposition},
as explained and used in 
Section \ref{subsection_decomposition_formula}.}
$n$-pointed  parametrized tropical curve 
of genus $g \leqslant g_{\Delta, n}$
and of type $\Delta$ passing through $p$, then
\begin{itemize}
\item $g=g_{\Delta ,n}$.
\item We have $g(V)=0$ for every vertex $V$ of $\Gamma$.
In particular, the graph $\Gamma$ has genus $g_{\Delta, n}$.
\item Images by $h$ of distinct vertices are distinct.
\item No edge is contracted to a point.
\item Images by $h$ of two distinct edges intersect in at most one point.
\item Unpointed vertices are trivalent.
\item Pointed vertices are bivalent.
\end{itemize} 
\end{prop}

\begin{proof}
This is essentially Proposition 4.11 of Mikhalkin 
\cite{MR2137980}, which itself is essentially some counting
of dimensions.
In \cite{MR2137980}, there is no genus attached to the vertices 
but if we have a parametrized tropical curve of genus $g \leqslant g_{\Delta ,n }$ 
with some vertices of non-zero genus, the underlying graph has genus strictly less than $g$ and so strictly less than 
$g_{\Delta, n}$, which is impossible by Proposition 4.11 of \cite{MR2137980} for $p$ general enough.
\end{proof}

\begin{prop} \label{prop_finite}
If $p \in U_{\Delta,n}$, 
then the set $T_{\Delta, p}$
of rigid  $n$-pointed genus $g_{\Delta, n}$ parametrized 
tropical curves 
$h \colon \Gamma \rightarrow \R^2$ of type $\Delta$
passing through $p$ is finite.
\end{prop}

\begin{proof}
This is Proposition 4.13 if Mikhalkin \cite{MR2137980}: 
there are finitely many possible combinatorial types for a 
parametrized tropical curve as in 
\mbox{Proposition \ref{prop_good}}, and for a fixed combinatorial 
type, the set of such tropical curves passing through $p$ is a zero dimensional 
intersection of a linear subspace with an open convex polyhedron,
so is a point.
\end{proof}

\begin{lem} \label{lem_vertices}
Let $h \colon \Gamma \rightarrow \R^2$ 
be a parametrized tropical curve 
in $T_{\Delta, p}$.
Then $\Gamma$
has $2g_{\Delta, n}-2+|\Delta|$ 
trivalent vertices.
\end{lem}

\begin{proof}
By definition of $T_{\Delta, p}$,
the graph
$\Gamma$ is of genus  $g_{\Delta ,n}$
and its vertices are either trivalent or bivalent.
Replacing the two edges 
adjacent to each 
bivalent vertex by a 
unique edge, we obtain a trivalent graph 
$\hat{\Gamma}$
with the same genus and the same number of unbounded edges as 
$\Gamma$.
Let 
$|V(\hat{\Gamma})|$ be the number of vertices 
of $\hat{\Gamma}$ 
and let 
$|E_{f}(\hat{\Gamma})|$
be the number of 
bounded edges
of $\hat{\Gamma}$. 
A count of half-edges using that $\hat{\Gamma}$ is trivalent
gives 
\[3 |V(\hat{\Gamma})|
 = 2 |E_f(\hat{\Gamma})|
+ |\Delta| \,.\]
By definition of the genus, we have 
\[1-g_{\Delta ,n} =|V(\hat{\Gamma})| - 
|E_f(\hat{\Gamma})|\,.\]
Eliminating 
$|E_f(\hat{\Gamma})|$ from the two previous equalities 
gives the 
desired formula and so finishes the 
proof of Lemma \ref{lem_vertices}.
\end{proof}

For $h \colon \Gamma \rightarrow \R^2$ a 
parametrized tropical curve in $\R^2$ and $V$ a trivalent vertex of adjacent edges
$E_1$, $E_2$ and $E_3$, the multiplicity of 
$V$ is the integer defined by 
\[ m(V) \coloneqq |\det (v_{V,E_1}, v_{V,E_2})| \,.\]
Thanks to the balancing condition 
\[v_{V,E_1}+v_{V,E_2}+v_{V,E_3}=0 \,,\]
we also have 
\[ m(V)= |\det (v_{V,E_2}, v_{V,E_3})|
=| \det (v_{V,E_3}, v_{V,E_1})|\,.\]

For $(h \colon \Gamma \rightarrow \R^2) \in T_{\Delta,p}$, the multiplicity of $h$ is defined by 
\[ m_h \coloneqq \prod_{V
\in V^{(3)}(\Gamma)} m(V) \,,\]
where the product is over the trivalent, i.e.\ unpointed,
vertices of $\Gamma$.

We denote $N^{\Delta,p}_{\mathrm{trop}}$ the count with 
multiplicity of $n$-pointed genus $g_{\Delta, n}$ 
param\-etrized tropical curves of type $\Delta$ passing through 
$p$, i.e.\ 
\[N^{\Delta, p}_{\mathrm{trop}}  
\coloneqq
 \sum_{h \in T_{\Delta, p}} m_h \,. \]

This tropical count with multiplicity has a natural refinement, first suggested 
by Block and G\"ottsche \cite{MR3453390}. We can replace the integer valued multiplicity 
$m_h$ of a parametrized tropical curve 
$h \colon \Gamma \rightarrow \R^2$ 
by the $\NN[q^{\pm \frac{1}{2}}]$-valued multiplicity 
\[m_h(q) \coloneqq
\prod_{V
\in V^{(3)}(\Gamma)} \frac{q^{\frac{m(V)}{2}} -q^{-\frac{m(V)}{2}}}{q^{\frac{1}{2}}-q^{-\frac{1}{2}}}
= \prod_{V \in V^{(3)}(\Gamma)} 
\left( \sum_{j=0}^{m(V)-1} q^{-\frac{m(V)-1}{2}+j} \right) \,,\]
where the product is taken over the trivalent vertices of 
$\Gamma$. The specialization $q=1$ recovers the usual multiplicity:
\[ m_h(1)=m_h \,.\]

Counting the parametrized tropical curves in 
$T_{\Delta, p}$ as above but with   
$q$-multiplicities, we obtain a refined tropical count 
\[N^{\Delta,p}_{\mathrm{trop}} (q)
  \coloneqq
 \sum_{h \in T_{\Delta, p}} m_h (q)
 \in \NN[q^{\pm \frac{1}{2}}] \,,\]
which specializes to the tropical count $N^{\Delta, p}_{\mathrm{trop}}$
at $q=1$ : 
\[ N^{\Delta, p}_{\mathrm{trop}} (1)=N^{\Delta, p}_{\mathrm{trop}} \,.\]

\subsection{Unrefined correspondence theorem}
Let $\Delta$ be a balanced collection of vectors in $\Z^2$,
as in Section \ref{notations}, and let $n$ be a 
non-negative integer and $p \in U_{\Delta,n}$.
Then we have some log Gromov-Witten count $N^{\Delta, n}$
of $n$-pointed genus $g_{\Delta,n}$ curves of type $\Delta$
 passing through $n$ points in the toric surface $X_\Delta$ 
(see Section
\ref{log_GW}), and we have some count with multiplicity 
$N^{\Delta,n}_{\mathrm{trop}}$ of $n$-pointed genus $g_{\Delta, n}$ 
tropical curves of type $\Delta$ passing through $n$
points 
$p=(p_1, \dots, p_n)$ in $\R^2$ (see Section \ref{tropical}).
The (unrefined) correspondence theorem then takes the simple form
$$N^{\Delta,n}=N^{\Delta,p}_{\mathrm{trop}}.$$
The result proved by Mikhalkin \cite{MR2137980}
and generalized by Nishinou, Siebert
\cite{MR2259922} 
is an equality between the tropical count
$N^{\Delta,n}_{\mathrm{trop}}$ and an enumerative count of algebraic curves. 
The fact that this enumerative count coincides with the log 
Gromov-Witten count 
$N^{\Delta, n}$ 
is proved by Mandel and Ruddat in
\cite{mandel2016descendant}.

\subsection{Refined correspondence theorem}
\label{section_refined_thm}

The Block-G\"ottsche refinement from 
$N^{\Delta, p}$ to $N^{\Delta, p}(q)$,
reviewed in Section
\ref{tropical}, 
is done at the tropical level so is combinatorial in nature and its geometric meaning is a priori unclear.

The main result of the present paper is a new 
non-tropical interpretation of Block-G\"ottsche 
invariants in terms of the higher genus log Gromov-Witten invariants with one lambda class inserted $N_{\Delta, n}^g$
that we introduced in \mbox{Section \ref{log_GW}}.
In particular, this geometric interpretation is independent of any 
tropical limit and makes the tropical deformation 
invariance of Block-G\"ottsche invariants manifest.

More precisely, we prove a refined correspondence theorem,
already stated as Theorem 
\ref{main_thm1} in the Introduction.

\begin{thm} \label{main_thm}
For every $\Delta$ 
balanced collection of vectors in 
$\Z^2$,
for every 
non-negative integer $n$ such that 
$g_{\Delta, n} \geqslant 0$,
and for every  
$p \in 
U_{\Delta, n}$,
we have the equality 
\[\sum_{g \geqslant g_{\Delta ,n}} N_g^{\Delta ,n}
u^{2g-2+|\Delta|}
=N^{\Delta,p}_{\mathrm{trop}}(q) \left( (-i)
(q^{\frac{1}{2}}-q^{-\frac{1}{2}}) \right)^{2g_{\Delta,n}-2+|\Delta|}\]
of power series in $u$ with rational coefficients, 
where
\[ q=e^{iu}
=\sum_{n \geqslant 0} \frac{(iu)^n}{n!} \,.\]
\end{thm}

\textbf{Remarks} 
\begin{itemize}
\item The change of variables $q=e^{iu}$ makes the above correspondence quite non-trivial. 
In particular, 
in contrast to its unrefined version,
it cannot be reduced to a finite to one enumerative correspondence.
It is essential to have a virtual/non-enumerative count on the Gromov-Witten side: 
for $g$ large enough, most of the contributions to 
$N_g^{\Delta, n}$ come from maps with contracted components.
\item The refined tropical count has the symmetry 
$N^{\Delta, n}_{\mathrm{trop}}(q)=N^{\Delta, n}_{\mathrm{trop}}(q^{-1})$
and so, after the change of variables $q=e^{iu}$, is a even power series in $u$. 
In particular, as 
\[ (-i)
(q^{\frac{1}{2}}-q^{-\frac{1}{2}}) \in u \Q[\![u^2]\!] \,,\]
the tropical side of 
Theorem \ref{main_thm} lies in 
\[ u^{2g_{\Delta,n}-2+|\Delta|} \Q[\![u^2]\!]\,,\]
as does the Gromov-Witten side.
Taking the leading order terms 
on both sides in the limit $u \rightarrow 0$, $q \rightarrow 1$, 
we recover the unrefined correspondence theorem $N^{\Delta ,n}
=N^{\Delta, p}_{\mathrm{trop}}$.
\item By Lemma 
\ref{lem_vertices}, we know 
that $2g_{\Delta ,n }-2+|\Delta|$ is the 
number of trivalent vertices of a
parametrized tropical curve in 
$T_{\Delta, p}$.
In particular, the tropical side of Theorem \ref{main_thm}
can be obtained directly by considering 
only the numerators of the Block-G\"ottsche multiplicities, i.e.\
Theorem \ref{main_thm}
can be rewritten 
\[\sum_{g\geqslant g_{\Delta ,n}} N^{\Delta ,n}_g u^{2g-2+|\Delta|}=
\sum_{h \in T_{\Delta, p}} \prod_V (-i)\left(q^{\frac{m(V)}{2}} -q^{-\frac{m(V)}{2}}\right) \,, \]
where 
$q=e^{iu}$.
\end{itemize}

\subsection{Fixing points on the toric boundary} It is possible to generalize Theorem \ref{main_thm}
by  fixing the position of 
some of the intersection points with the toric boundary divisor.
Let $\Delta^F$ be a subset of $\Delta$
and let 
\[ \mathrm{ev}_{\Delta^F} \colon \overline{M}_{g,n,\Delta} \rightarrow (\partial X_\Delta)^{|\Delta^F|} \]
be the evaluation map at the intersection points with the toric boundary divisor
$\partial X_\Delta$ indexed by the 
elements of $\Delta^F$. 

The problem of counting $n$-pointed genus $g$ curves of type $\Delta$ passing through $n$ given points
of $X_\Delta$ 
and with fixed position of the intersection points with 
$\partial X_\Delta$
indexed by $\Delta^F$, has virtual dimension zero 
if the genus is equal to 
\[g_{\Delta,n}^{\Delta^F} \coloneqq n+1- |\Delta| + |\Delta^F| \,.\]
For every $g \geqslant g_{\Delta,n}^{\Delta^F}$,
we
define the invariants
\[ N^{\Delta,n}_{g,\Delta^F} \coloneqq
\int_{[\overline{M}_{g,n,\Delta}]^{virt}}
(-1)^{g-g_{\Delta,n}^{\Delta^F}}
\lambda_{g-g_{\Delta,n}^{\Delta^F}} 
\text{ev}_{\Delta^F}^{*} (r^{|\Delta^F|}) 
\prod_{j=1}^n \text{ev}_j^{*}(\mathrm{pt}) \,,\]
where $r \in A^1(\partial X_\Delta)$ is the class of a point on $\partial X_\Delta$.

We can consider the corresponding tropical 
problem. Fix a generic configuration
$x=(x_v)_{v \in \Delta^F}$ of points 
in $\R^2$ and say that a tropical curve of type $\Delta$ is of type $(\Delta, \Delta^F)$
if the unbounded edges
in correspondence with $\Delta^F$
asymptotically coincide with the half-lines 
$x_v + \R_{\geqslant 0} v$, $v \in \Delta^F$.

We define a refined tropical count 
\[ N^{\Delta, p, x}_{\mathrm{trop},\Delta^F}(q)
\in \NN[q^{\pm \frac{1}{2}}] \,,\]
by counting with $q$-multiplicity the 
tropical curves of genus 
$g_{\Delta,n}^{\Delta^F}$ and of type 
$(\Delta, \Delta^F)$
passing through 
a generic configuration 
$p=(p_1, \dots, p_n)$
of $n$ points in 
$\R^2$.

The following result is the generalization of 
Theorem \ref{main_thm} to the case of non-empty
$\Delta^F$.

\begin{thm} \label{thm_fixing_points}
For every $\Delta$ balanced collection of vectors in $\Z^2$, 
for every $\Delta^F$
subset of $\Delta$ and for every $n$ 
non-negative integer 
such that $g_{\Delta,n}^{\Delta^F} \geqslant 0$,
we have the equality 
\[ \sum_{g \geqslant g_{\Delta,n}^{\Delta^F}}
N^{\Delta,n}_{g,\Delta^F}
u^{2g-2+|\Delta|} \]
\[= \left( \prod_{v \in \Delta^F}\frac{1}{|v|} \right) 
N^{\Delta,p,x}_{\mathrm{trop}}(q)
\left( (-i)(q^{\frac{1}{2}}
-q^{-\frac{1}{2}}) \right)^{2g_{\Delta,n}^{\Delta^F}-2+|\Delta|}
\]
of power series in $u$ with rational coefficients, where $q=e^{iu}$.
\end{thm}

The proof of Theorem
\ref{thm_fixing_points} is entirely 
parallel to the proof of \mbox{Theorem 
\ref{main_thm}} 
(\mbox{Theorem \ref{main_thm1}} of the 
Introduction). The required 
modifications are discussed at the end of Section
\ref{section_general_vertex}.

\section{Gluing and vanishing properties of lambda classes} \label{section_lambda}

In this Section, we review some well-known facts:
a gluing result for
lambda classes, 
\mbox{Lemma 
\ref{lem_gluing1}}, and then a vanishing 
result, 
Lemma
\ref{lem_gluing2}.

\begin{lem} \label{lem_gluing1}
Let $B$ be a scheme over $\C$.
Let $\Gamma$ be a graph, of genus 
$g_\Gamma$, and let
$\pi_V \colon \cC_V
\rightarrow B$ be prestable curves over $B$ indexed by the vertices $V$ of $\Gamma$.
For every edge $E$ of $\Gamma$, connecting vertices $V_1$ and $V_2$, let $s_{E,1}$ and $s_{E,2}$ be smooth sections of $\pi_{V_1}$ and 
$\pi_{V_2}$
respectively.
Let $\pi \colon \cC \rightarrow B$
be the prestable curve over $B$ obtained by gluing together the 
sections $s_{V_1,E}$ and $s_{V_2,E}$
corresponding to a same edge $E$
of $\Gamma$.
Then, we have an exact sequence
\[ 0 \rightarrow 
\bigoplus_{V \in V(\Gamma)}
(\pi_V)_* \omega_{\pi_V}
\rightarrow \pi_* \omega_\pi
\rightarrow \cO^{\oplus g_\Gamma}
\rightarrow 0\,,\]
where 
$\omega_{\pi_V}$ and 
$\omega_\pi$ are the relative  
line bundles.
\end{lem}

\begin{proof}
Let $s_E \colon B \rightarrow \cC$  
be the gluing  
sections. Then 
we have an exact sequence 
\[ 0 \rightarrow \cO_\cC
\rightarrow \bigoplus_{V 
\in V(\Gamma)} \cO_{\cC_V}
\rightarrow \bigoplus_{E \in E(\Gamma)} \cO_{s_E(B)}
\rightarrow 0\,.\]
Applying $R \pi_*$, we obtain an exact sequence
\[0 \rightarrow \pi_* \cO_{\cC}
\rightarrow \bigoplus_{V \in V(\Gamma)} 
\pi_* \cO_{\cC_V}
\rightarrow 
\bigoplus_{E \in E(\Gamma)}
\pi_* \cO_{s_E(B)}\]
\[ \rightarrow 
R^1 \pi_* \cO_\cC 
\rightarrow 
\bigoplus_{V \in V(\Gamma)} R^1 \pi_*
\cO_{\cC_V} \rightarrow 0 \,.\]
The kernel of 
\[R^1 \pi_* \cO_\cC 
\rightarrow 
\bigoplus_{V \in V(\Gamma)} R^1 \pi_*
\cO_{\cC_V} \]
is a free sheaf of rank
$|E(\Gamma)|-|V(\Gamma)|+1=g_\Gamma$.
We obtain the desired exact sequence by Serre duality. 

Equivalently, if we choose $g_\Gamma$ edges of $\Gamma$
whose complement is a tree,
we can understand the morphism 
\[\pi_* \omega_\pi \rightarrow \cO^{\oplus g_\Gamma} \]
as taking the residues at the corresponding $g_\Gamma$ sections.
\end{proof}

\begin{lem} \label{lem_gluing2}
Let $B$ be a scheme over $\C$.
Let $\pi \colon \cC \rightarrow B$ be a prestable curve of arithmetic 
genus $g$ over $B$. 
For every integer $g'$ such that $0 \leqslant g' \leqslant g$,
let $B_{g'}$ be the closed subset  of $B$ of points $b$ such that the dual graph of
the curve $\pi^{-1}(b)$ is of genus $\geqslant g'$.
Then the lambda classes $\lambda_j \in H^{2j}(B,\Q)$, defined by 
$\lambda_j = c_j (\pi_{*} \omega_\pi)$,
satisfy 
\[\lambda_j|_{B_{g'}}=0\]
in $H^{2j}(B_{g'},\Q)$
for all $j>g-g'$.
\end{lem}

\begin{proof}
Let $\tilde{B}_{g'}$ be the finite cover of 
$B_{g'}$ given by the possible choices of $g'$ fully separating 
nodes, i.e.\ of nodes whose complement is of \mbox{genus $0$.} 
Separating these $g'$ fully separating nodes gives a way to write the pullback of $\cC$ to $\tilde{B}_{g'}$ 
as the gluing of curves according to a dual graph $\Gamma$ of genus $g'$.
According to Lemma
\ref{lem_gluing1}, the Hodge bundle of this family of curves has a trivial rank $g'$ quotient.
As $\tilde{B}_{g'}$ is finite over $B_g'$, it is enough to guarantee the desired vanishing 
in rational cohomology.

\end{proof}

\section{Toric degeneration and decomposition formula} \label{section: decomposition}

In Section \ref{trop}, we review
the natural link between log geometry and
tropical geometry given by  tropicalization.
In Section \ref{subsection_toric_degeneration}, we start the proof of 
\mbox{Theorem \ref{main_thm1}} by considering
the Nishinou-Siebert toric degeneration.
In \mbox{Section
\ref{subsection_decomposition_formula}}, we apply the decomposition formula
of Abramovich, Chen, Gross, Siebert
\cite{abramovich2017decomposition} to this 
toric degeneration to write the log
Gromov-Witten invariants 
$N_g^{\Delta,n}$ in terms of log Gromov-Witten invariants 
$N_g^{\Delta, h}$ indexed by parametrized tropical curves $h \colon \Gamma \rightarrow \R^2$. We use the vanishing result of Section
\ref{section_lambda} to restrict the 
tropical curves appearing.

\subsection{Tropicalization} \label{trop}

Log geometry is naturally 
related to tropical geometry.
Every log scheme 
$X$ admits a tropicalization 
$\Sigma (X)$.

Recall that a log scheme is a scheme $X$ endowed with a sheaf of 
monoids $\cM_X$ and a morphism of sheaves of monoids\footnote{All the monoids considered will be commutative 
and with an identity element.} 
\[ \alpha_X \colon \cM_X \rightarrow \cO_X \,,\]
where $\cO_X$ is seen as 
a sheaf of multiplicative monoids, such that the restriction of 
$\alpha_X$ to $\alpha_X^{-1}(\cO_X^*)$ is an isomorphism.

The ghost sheaf of a log scheme 
$X$
is the sheaf of monoids 
\[ \overline{\cM}_X \coloneqq \cM_X/\alpha^{-1}(\cO_X^*) \,.\]
For the kind of log schemes that we are considering, fine and saturated, 
the ghost sheaf is of combinatorial nature.
In this case, one can think of the 
log geometry of $X$ as a combination of the geometry of the underlying scheme 
$X$ and of the combinatorics of the ghost sheaf 
$\overline{\cM}_X$.
Non-trivial interactions between these two aspects of log geometry are encoded  in the sequence
\[ \cO_X^* \rightarrow \cM_X \rightarrow \overline{\cM}_X \,.\]

A cone complex is an abstract gluing
of convex rational cones along their faces.
If $X$ is a log scheme, the tropicalization 
$\Sigma (X)$ of $X$ is the cone complex 
defined by gluing together the 
convex rational cones
$\Hom (\overline{\cM}_{X,x}, \R_{\geqslant 0})$ for all
$x \in X$
according to the natural specialization maps.
 Tropicalization is a functorial construction.
For more details on tropicalization of log schemes, 
we refer to
Appendix B of \cite{MR3011419}
and Section 2 of \cite{abramovich2017decomposition}. 
Tropicalization gives a pictorial way to describe the combinatorial 
part of log geometry contained in the ghost sheaf.

\textbf{Examples}
\begin{itemize}
\item Let $X$ be a toric variety. 
We can view $X$ as a log scheme for the toric divisorial log structure, 
i.e.\ the divisorial log stucture with respect to the toric boundary divisor
$\partial X$. The sheaf $\cM_X$ is the sheaf of functions non-vanishing outside $\partial X$
and $\alpha_X$ is the natural inclusion of $\cM_X$ in $\cO_X$. 
The tropicalization
$\Sigma (X)$ of $X$
is naturally isomorphic as cone complex to the fan of $X$.
\item Let 
$\overline{\cM}$
be a monoid whose only invertible
element is $0$. Let $X$ be the log scheme of underlying scheme the point 
$\mathrm{pt}= \Spec \, \C$, with $\cM_X
= \overline{\cM} \oplus \C^*$ and 
\[ \alpha_X \colon  \overline{\cM} \oplus \C^* \rightarrow \C  \]
\[ (m, a) \mapsto a \delta_{m, 0} \,.\]
We denote this log scheme as
$\mathrm{pt}_{\overline{\cM}}$ and such a log scheme is called a log point. 
By construction, we have 
$\overline{\cM}_{\mathrm{pt}_{\overline{\cM}}} = \overline{\cM}$
and so the tropicalization  
$\Sigma (\mathrm{pt}_{\overline{\cM}})$
is the cone 
$\Hom (\overline{\cM}, \R_{\geqslant 0})$,
i.e.\ the fan of the affine toric variety
$\Spec \C[\overline{\cM}] \,.$

\item The log point 
$\mathrm{pt}_{\NN}$ obtained for 
$\overline{\cM}=\NN$ is called the standard log point.
Its tropicalization is simply
$\Sigma (\mathrm{pt}_{\NN})
=\R_{\geqslant 0}$, the fan of the affine line $\A^1$.
\item The log point $\mathrm{pt}_0$
obtained for $\overline{\cM}=0$ is 
called the trivial log point. Its tropicalization
$\Sigma (\mathrm{pt}_0)$ is reduced to a point.
\item A stable log map to 
some relative log scheme 
$X \rightarrow S$ determines a
commutative diagram in the category of log schemes,
\begin{center}
\begin{tikzcd}
C \arrow{r}{f} \arrow{d}{\pi}
& X \arrow{d}\\
\mathrm{pt}_{\overline{\cM}} \arrow{r} & S \,,
\end{tikzcd}
\end{center}
where $\mathrm{pt}_{\overline{\cM}}$
is a log point and $\pi$ is a log smooth proper integral curve. 
In particular, the scheme underlying $C$ is a projective nodal 
curve with a natural set of smooth marked points.
We can take the tropicalization of this diagram to obtain a commutative diagram of cone complexes
\begin{center}
\begin{tikzcd}
\Sigma(C) \arrow{r}{\Sigma(f)} \arrow{d}{\Sigma(\pi)}
& \Sigma(X) \arrow{d}\\
\Sigma(\mathrm{pt}_{\overline{\cM}}) \arrow{r} & \Sigma(S) \,.
\end{tikzcd}
\end{center}
$\Sigma (C)$ is a family of graphs over the cone 
$\Sigma(\mathrm{pt}_{\overline{M}})= \Hom (\overline{\cM}, \R_{\geqslant 0})$:
the fiber of $\Sigma (\pi)$ over  a point in the interior of the cone is the dual graph of $C$.
Fibers over faces of the cone are contractions of the dual graph. 
In particular, the fiber over the origin of the cone is obtained by 
fully contracting the dual graph of $C$ to a graph with a unique vertex. 
If $X$ is a toric variety with the toric divisorial log structure
and $S$ is the trivial log point, then 
$\Sigma (f)$ is a family of parametrized tropical curves in
the fan of $X$. We refer to Section 2.5 of \cite{abramovich2017decomposition}
for more details.
\end{itemize}

\subsection{Toric degeneration}
\label{subsection_toric_degeneration}

Let $\Delta$ be a balanced configuration of vectors, as in Section
\ref{notations}, and  let $n$ be a non-negative integer such that
$g_{\Delta, n} \geqslant 0$.
We fix $p=(p_1, \dots, p_n)$ a configuration of $n$ points in $\R^2$ belonging to the open dense subset 
$U_{\Delta, n}$ of $(\R^2)^n$
given by \mbox{Proposition \ref{prop_good}.}
Let $T_{\Delta, p}$ be the set of 
$n$-pointed genus $g_{\Delta, n}$
parametrized tropical curves in 
$\R^2$ of type $\Delta$
passing through $p$.
The set $T_{\Delta, p}$
is finite by \mbox{Proposition
\ref{prop_finite}.}
Proposition \ref{prop_good}
shows that the elements of 
$T_{\Delta, p}$ are particularly 
nice parametrized tropical curves.

We can slightly modify $p$
such that $p \in (\Q^2)^n \cap
U_{\Delta, n}$ without changing the combinatorial type of the elements of $T_{\Delta, p}$
and so without changing the tropical counts 
$N^{\Delta,p}_{\mathrm{trop}}$
and $N^{\Delta,p}_{\mathrm{trop}}(q)$. 
In that case, for every 
parametrized tropical curve 
$h \colon \Gamma \rightarrow \R^2$
in $T_{\Delta, p}$ 
and for every vertex $V$ of $\Gamma$, we have $h(V) \in \Q^2$
and for every edge $E$ of $\Gamma$, we have $\ell(E) \in \Q$.
Indeed, the positions $h(V)$
of vertices in $\R^2$ and the lengths $\ell(E)$ of edges are natural 
parameters on the moduli space of genus $g_{\Delta, n}$ 
parametrized tropical curves of type $\Delta$
and this 
moduli space is a rational 
polyhedron in the space of these parameters. 
The set $T_{\Delta, p}$
is obtained as zero dimensional intersection of this rational 
polyhedron with the rational (because $p \in (\Q^2)^n$)
linear space imposing to pass through $p$. It follows that
the parameters $h(V)$ and $\ell(E)$ 
are rational for elements of
$T_{\Delta, p}$.
 
We follow the toric degeneration approach 
introduced by Nishinou and Siebert 
\cite{MR2259922} (see also  Mandel and Ruddat
\cite{mandel2016descendant}). 
According to \cite{MR2259922} 
Proposition 3.9 and 
\cite{mandel2016descendant} Lemma 3.1, there exists 
a rational polyhedral decomposition
$\cP_{\Delta ,p}$ 
of $\R^2$ such that 
\begin{itemize}
\item The asymptotic fan of $\cP_{\Delta ,p}$ is the fan of $X_\Delta$.
\item For every parametrized tropical curve
$h \colon \Gamma \rightarrow \R^2$
in $T_{\Delta, p}$, the
images $h(V)$ of vertices $V$
of $\Gamma$ are vertices of 
$\cP_{\Delta ,p}$ and the images 
$h(E)$ of edges $E$ of $\Gamma$ are contained in union of edges of
$\cP_{\Delta ,p}$
\end{itemize}
Remark that the points $p_j$ in 
$\R^2$ are image of vertices of
parametrized tropical curves in 
$T_{\Delta, p}$
and so are vertices of 
$\cP_{\Delta ,p}$.

Given a parametrized tropical curve 
$h \colon \Gamma \rightarrow \R^2$
in $T_{\Delta, p}$, we construct a new
parametrized tropical curve 
$\tilde{h} \colon \tilde{\Gamma} \rightarrow \R^2$ by simply adding a
genus zero bivalent unpointed vertex to $\Gamma$
at each point $h^{-1}(V)$ for $V$ a vertex of 
$\cP_{\Delta,p}$ which is not the image by $h$ of a vertex of $\Gamma$. The image 
$\tilde{h}(E)$ of each edge $E$ of 
$\tilde{\Gamma}$ is now exactly an edge of 
$\cP_{\Delta,p}$. The graph $\tilde{\Gamma}$ has three types of vertices:
\begin{itemize}
\item Trivalent unpointed vertices, coming from $\Gamma$.
\item Bivalent pointed vertices, coming 
from $\Gamma$.
\item Bivalent unpointed vertices, not coming from $\Gamma$.
\end{itemize}

Doing a global rescaling of $\R^2$ if necessary, 
we can assume that 
$\cP_{\Delta ,p}$ is an integral 
polyhedral decomposition, i.e.\ that all the vertices of 
$\cP_{\Delta ,p}$ are in $\Z^2$,
and that all the lengths $\ell(E)$
of edges $E$ of parametrized tropical curves $\tilde{h}
\colon \tilde{\Gamma}
\rightarrow \R^2$, coming from
$h \colon \Gamma \rightarrow \R^2$
in $T_{\Delta, p}$,
are integral.

Taking the cone over $\cP_{\Delta, p} \times \{1\}$ in 
$\R^2 \times \R$, we obtain 
the fan of a three dimensional
toric variety $X_{\cP_{\Delta, p}}$
equipped with a 
morphism 
\[ \nu \colon 
X_{\cP_{\Delta, p}}
\rightarrow \A^1 \] 
coming from the projection $\R^2 \times \R \rightarrow \R$ on the third
$\R$ factor. We have 
$\nu^{-1}(t) \simeq X_\Delta$ for every $t \in \A^1 -\{0\}$. The special fiber 
$X_0 \coloneqq \nu^{-1}(0)$ 
is a reducible surface 
whose irreducible components $X_V$ are 
toric surfaces in one to one correspondence with
the vertices $V$ of $\cP_{\Delta ,p}$,
\[ X_0 = \bigcup_V X_V \,.\]
In other words, $\nu \colon 
X_{\cP_{\Delta, p}}
\rightarrow \A^1$ 
is a toric degeneration of $X_\Delta$.

We consider the toric varieties $\A^1$, $X_{\cP_{\Delta, p}}$,
$X_\Delta$ and $X_V$ as
log schemes with respect to the 
toric divisorial log structure.
In particular, the toric morphism $\nu$ induces a log smooth morphism 
\[ \nu \colon X_{\cP_{\Delta, n}} \rightarrow \A^1\,. \]
Restricting to the special fiber gives a structure of log scheme
on 
$X_0$ and a
log smooth morphism to the standard log point
\[\nu_0 \colon X_0 \rightarrow 
\mathrm{pt}_{\NN}.\]
From now on, we will denote 
$\underline{X}_0$ the scheme underlying 
the log scheme $X_0$. Beware that the toric divisorial log structure that we consider on $X_V$ is not the restriction of the log structure that we consider on $X_0$.

For every $j=1,\dots,n$, the ray $ \R_{\geqslant 0} (p_j, 1)$
in $\R^2 \times \R$ defines a one-parameter subgroup 
$\C^{*}_{p_j}$ of $(\C^*)^3 \subset X_{\cP_{\Delta ,n}}$. We choose a point $P_j \in (\C^*)^2$
and we write $Z_{P_j}$ the affine line in 
$X_{\cP_{\Delta ,n}}$ defined as the closure of the orbit of $(P_j,1)$ under the action of $\C^*_{p_j}$.
We have 
\[ Z_{P_j} \cap \nu^{-1}(1)
=Z_{P_j} \cap X_\Delta = P_j \,,\] 
and 
\[ P_j^0 \coloneqq Z_{P_j} \cap \nu^{-1}(0)\] 
is a point in the dense torus $(\C^*)^2$ contained in the toric component of $X_0$ 
corresponding to the vertex $p_j$ of $\cP_{\Delta ,p}$.
In other words, $Z_{P_j}$ is a section of 
$\nu$ degenerating 
$P_j \in X_\Delta$ to some 
$P_j^0 \in X_0$.

Recall from Section
\ref{log_GW} that the log Gromov-Witten invariants 
$N_g^{\Delta ,n}$ are defined using stable log maps of target $X_\Delta$,
\[N^{\Delta ,n}_g \coloneqq \int_{[\overline{M}_{g,n,\Delta}]^{\virt}} 
(-1)^{g-g_{\Delta ,n}}
\lambda_{g-g_{\Delta ,n}} \prod_{j=1}^n 
\text{ev}_j^{*}(\mathrm{pt}) \,,\]
where $\overline{M}_{g,n,\Delta}$ is the moduli space 
of $n$-pointed stable log maps to 
$X_\Delta$ of genus $g$ and of type $\Delta$.

Let 
$\overline{M}_{g,n,\Delta}(X_0 / \mathrm{pt}_{\NN})$ 
be the 
moduli space of $n$-pointed stable log maps to 
$\pi_0 \colon X_0 \rightarrow \mathrm{pt}_{\NN}$ of genus $g$ and of type $\Delta$. 
It is a proper Deligne-Mumford
stack of virtual dimension 
\[\vdim \overline{M}_{g,n,\Delta}(X_0 / \mathrm{pt}_{\NN})
=\vdim \overline{M}_{g,n,\Delta}=g-g_{\Delta ,n}+2n \]
and it admits a virtual fundamental class 
$$[\overline{M}_{g,n,\Delta}(X_0/\mathrm{pt}_{\NN})]^{\virt} 
\in A_{g-g_{\Delta ,n}+2n} (\overline{M}_{g,n,\Delta}(X_0/\mathrm{pt}_{\NN}), \Q)\,.$$

Considering the evaluation morphism 
\[ \text{ev} \colon \overline{M}_{g,n,\Delta}(X_0 / \mathrm{pt}_{\NN}) \rightarrow \underline{X}_0^n \] 
and the inclusion 
\[ \iota_{P^0} \colon 
(P^0 \coloneqq  (P_1^0, \dots, P_n^0)) \hookrightarrow \underline{X}_0^n \,,\]
we can define
the moduli space\footnote{As already mentioned in Section \ref{log_GW}, we consider moduli spaces of stable log maps as stacks, not log stacks. In particular, the morphisms $\text{ev}$, $\iota_{P^0}$ and the fiber product defining  
$\overline{M}_{g,n,\Delta}(X_0 / \mathrm{pt}_{\NN}, P^0)$ are defined in the category of stacks, not log stacks.}
\[   \overline{M}_{g,n,\Delta}(X_0 / \mathrm{pt}_{\NN}, P^0)
\coloneqq 
\overline{M}_{g,n,\Delta}(X_0 / \mathrm{pt}_{\NN}) \times_{\underline{X}_0^n} P^0 \,,\]
of stable log maps passing through 
$P^0$, and by the
Gysin refined 
homomorphism (see
Section 6.2 of 
\cite{MR1644323}), 
a virtual fundamental class 
\[[\overline{M}_{g,n,\Delta}(X_0 / \mathrm{pt}_{\NN}, P^0)]^{\virt}
\coloneqq \iota_{P^0}^{!}[\overline{M}_{g,n,\Delta}(X_0/\mathrm{pt}_{\NN})]^{\virt} \]
\[ \in A_{g-g_{\Delta ,n}} (\overline{M}_{g,n,\Delta}(X_0/\mathrm{pt}_{\NN}, P^0), \Q)\,.\]
Remark\footnote{I thank the referee for stressing this point.} that this definition is compatible with \cite{abramovich2017decomposition}
because each $P^0_j$, seen as a log morphism $P^0_j \colon \mathrm{pt}_{\NN}
\rightarrow X_0$, is strict. This follows from the fact that we have chosen $P^0_j$
in the dense torus $(\C^{*})^2$ contained in the toric component of $X_0$ dual to the vertex $p_j$ of $\cP_{\Delta,p}$.
If it were not the case\footnote{In  Section 6.3.2 of \cite{abramovich2017decomposition}, sections defining point constraints have  to interact non-trivially with the log structure of the special fiber to produce something interesting because the  degeneration considered there is a trivial product, whereas we are considering a 
non-trivial degeneration.}, then, following Section 6.3.2 of \cite{abramovich2017decomposition}, the definition of $\overline{M}_{g,n,\Delta}(X_0 / \mathrm{pt}_{\NN}, P^0)$ should have been replaced by a fiber product in the category of fs log stacks
and $[\overline{M}_{g,n,\Delta}(X_0 / \mathrm{pt}_{\NN}, P^0)]^{\virt}$
should have been defined by some perfect 
obstruction theory directly on 
$\overline{M}_{g,n,\Delta}(X_0 / \mathrm{pt}_{\NN}, P^0)$.

By deformation invariance of the virtual fundamental class 
on moduli spaces of stable log maps in log smooth families, we have
\[N^{\Delta ,n}_g = \int_{[\overline{M}_{g,n,\Delta}(X_0/\mathrm{pt}_{\NN}, P^0)]^{\virt}}
(-1)^{g-g_{\Delta ,n}} 
\lambda_{g-g_{\Delta ,n}} \,.\]

\subsection{Decomposition formula}
\label{subsection_decomposition_formula}

As the toric degeneration breaks the toric surface $X_\Delta$ into  many pieces, 
irreducible components of the 
special fiber $X_0$, 
one can similarly expect that it breaks the moduli space 
$\overline{M}_{g,n,\Delta}$ of stable log maps
to $X_\Delta$ into many pieces, irreducible components of the moduli space 
$\overline{M}_{g,n,\Delta}(X_0/\mathrm{pt}_{\NN})$ 
of stable log maps to $X_0$. Tropicalization gives a way 
to understand the combinatorics of this breaking into pieces. 

As we recalled in Section
\ref{trop}, a  $n$-pointed stable log map to 
$X_0 / \mathrm{pt}_{\NN}$ of type
$\Delta$ 
gives
a commutative diagram of log schemes
\begin{center}
\begin{tikzcd}
C \arrow{r}{f} \arrow{d}{\pi}
& X_0 \arrow{d}{\nu_0}\\
\mathrm{pt}_{\overline{\cM}} \arrow{r}{g} & \mathrm{pt}_{\NN} \,,
\end{tikzcd}
\end{center}
which can be tropicalized 
in a commutative diagram of cone complexes
\begin{center}
\begin{tikzcd}
\Sigma(C) \arrow{r}{\Sigma(f)} \arrow{d}{\Sigma(\pi)}
& \Sigma(X_0) \arrow{d}{\Sigma(\nu_0)}\\
\Sigma(\mathrm{pt}_{\overline{\cM}}) \arrow{r}{\Sigma(g)} & \Sigma(\mathrm{pt}_{\NN}) \,.
\end{tikzcd}
\end{center}
We have $\Sigma (\mathrm{pt}_{\NN})
\simeq \R_{\geqslant 0}$ and the fiber 
$\Sigma(\nu_0)^{-1}(1)$
is naturally identified with
$\R^2$ equipped with the 
polyhedral decomposition 
$\cP_{\Delta,p}$, whose asymptotic fan is the fan of $X_\Delta$.
So the above diagram gives a family 
over the polyhedron 
$\Sigma(g)^{-1}(1)$
of $n$-pointed 
parametrized tropical curves 
in $\R^2$ of type $\Delta$

The moduli space $\overline{M}_{g,n,\Delta}^{\mathrm{trop}}$
of $n$-pointed genus $g$
parametrized tropical curves in 
$\R^2$ of type $\Delta$ 
is a rational polyhedral complex. 
If $\overline{M}_{g,n,\Delta}^{\mathrm{trop}}$ were the tropicalization of 
$\overline{M}_{g,n,\Delta}(X_0/\mathrm{pt}_{\NN})$
(seen as a log stack over 
$\mathrm{pt}_{\NN}$), then 
$\overline{M}_{g,n,\Delta}^{\mathrm{trop}}$ would be the dual intersection complex of 
$\overline{M}_{g,n,\Delta}^{\mathrm{trop}}$. 
In particular, irreducible components of
$\overline{M}_{g,n,\Delta}(X_0/\mathrm{pt}_{\NN})$
would be in one to one correspondence with the 0-dimensional faces of
$\overline{M}_{g,n,\Delta}^{\mathrm{trop}}$.
As the polyhedral decomposition of $\overline{M}_{g,n,\Delta}^{\mathrm{trop}}$ 
is induced by the combinatorial type of tropical curves, 
the 0-dimensional faces of $\overline{M}_{g,n,\Delta}^{\mathrm{trop}}$ correspond to the rigid
parametrized tropical curves,
see Definition 4.3.1 of \cite{abramovich2017decomposition}, i.e.\ 
to parametrized tropical curves which are
not contained in a non-trivial family of parametrized tropical curves of the same combinatorial type.

According to the decomposition formula of 
Abramovich, Chen, Gross and Siebert \cite{abramovich2017decomposition}, 
this heuristic description of the pieces of 
$\overline{M}_{g,n,\Delta}(X_0/\mathrm{pt}_{\NN})$
is correct at the virtual level: one can express
$[\overline{M}_{g,n,\Delta}(X_0/\mathrm{pt}_{\NN}, P^0)]^{\virt}$
as a sum
of contributions indexed by rigid tropical curves. 

Let 
$\tilde{h} \colon \tilde{\Gamma} \rightarrow \R^2$ be a $n$-pointed genus $g$ 
rigid parametrized tropical curve to $\R^2$ of type $\Delta$ 
passing through $p$. 
For every $V$ vertex
of $\tilde{\Gamma}$, let 
$\Delta_V$ be the balanced collection of vectors 
$v_{V,E}$ for all edges $E$ adjacent to $V$.
Using the notations 
of Section \ref{notations} that we used all along for $\Delta$ 
but now for $\Delta_V$, the toric surface $X_{\Delta_V}$ is the 
irreducible component of 
$X_0$ corresponding to the vertex
$h(V)$ of the 
\mbox{polyhedral 
decomposition $\cP_{\Delta, p}$.}

A $n$-pointed genus $g$ stable log map to 
$X^0$ of type $\Delta$
passing through 
$P^0$ and marked by 
$\tilde{h}$ is the following data,
see \cite{abramovich2017decomposition},
Definition 4.4.1\footnote{In \cite{abramovich2017decomposition}, 
the marking includes also a choice of curve classes for the stable maps $f_V$. 
In our case, the curve classes are uniquely determined because 
a curve class in a toric variety is uniquely determined by its 
intersection numbers with the components of the toric boundary divisor.},
\begin{itemize}
\item A $n$-pointed genus 
$g$ stable log map 
$f \colon C/\mathrm{pt}_{\overline{\cM}}
\rightarrow 
X_0 /\mathrm{pt}_{\NN}$
of type $\Delta$
passing through 
$P^0$.
\item For every vertex $V$ of 
$\tilde{\Gamma}$, an ordinary stable map
$f_V \colon C_V \rightarrow X_{\Delta_V}$
of class $\beta_{\Delta_V}$ with marked points $x_v$ for every $v \in \Delta_V$, such that $f_V(x_v)
\in D_v$, where $D_v$ is the prime toric divisor of $X_{\Delta_V}$
dual to the ray 
$\R_{\geqslant 0} v$.
\end{itemize}
These data must satisfy the following compatibility conditions: 
the gluing of the curves $C_V$ along the points corresponding to the edges of 
$\tilde{\Gamma}$ is isomorphic to the curve underlying the log curve $C$, 
and the corresponding gluing of the maps $f_V$ is the map underlying the log map $f$.

By \cite{abramovich2017decomposition}, the moduli space
$\overline{M}_{g,n,\Delta}^{\tilde{h},P^0}$
of $n$-pointed genus $g$ stable log maps of type $\Delta$ passing through $P^0$ and marked by $\tilde{h}$
is a 
proper Deligne-Mumford stack, equipped
with a natural virtual fundamental class 
$[\overline{M}^{\tilde{h},P^0}_{g,n,\Delta}]^{\virt}$.
Forgetting the marking by 
$\tilde{h}$ gives a morphism
\[i_{\tilde{h}} \colon 
\overline{M}^{\tilde{h},P^0}_{g,n,\Delta}
\rightarrow 
\overline{M}_{g,n,\Delta}(X_0/\mathrm{pt}_{\NN}, P^0) \,. \]

According to the decomposition formula, \cite{abramovich2017decomposition} Theorem 6.3.9, we have 
\[ [\overline{M}_{g,n,\Delta}(X_0/\mathrm{pt}_{\NN}, P^0)]^{\virt} 
= \sum_{\tilde{h}} 
\frac{n_{\tilde{h}}}{|\mathrm{Aut} (\tilde{h})|}
(i_{\tilde{h}})_* [\overline{M}^{\tilde{h},P^0}_{g,n,\Delta}]^{\virt} \,, \]
where the sum is over the 
$n$-pointed genus $g$ rigid parametrized 
tropical curves to $(\R^2,
\cP_{\Delta,p})$ of type $\Delta$
passing through $p$,  $n_{\tilde{h}}$ is the smallest positive integer such that the scaling of $\tilde{h}$ by 
$n_{\tilde{h}}$ has integral vertices and integral lengths, and 
 $|\mathrm{Aut}(\tilde{h})|$ is the order of the automorphism group of $\tilde{h}$.

Recall from 
Proposition 
\ref{prop_good}
that a 
parametrized tropical curve 
\mbox{$h \colon \Gamma \rightarrow \R^2$} in $T_{\Delta, p}$
has a source graph $\Gamma$ of genus $g_{\Delta,n}$
and that all vertices 
$V$ of $\Gamma$ are of genus zero: $g(V)=0$.
In \mbox{Section \ref{subsection_toric_degeneration}}, we explained that the polyhedral decomposition
$\cP_{\Delta,p}$
defines a new parametrized tropical
$\tilde{h} \colon \tilde{\Gamma}
\rightarrow \R^2$, for each $h \colon 
\Gamma \rightarrow \R^2$ in $T_{\Delta,p}$,
by addition of unmarked genus zero bivalent vertices.
Given such parametrized tropical curve
$\tilde{h}
\colon \tilde{\Gamma}
\rightarrow \R^2$, one can construct genus $g$ 
parametrized tropical curves by changing only the genus of vertices $g(V)$ so that  
\[ \sum_{V \in V(\Gamma)} 
g(V) = g-g_{\Delta, n} \,.\]
We denote $T_{\Delta, p}^g$ the set of genus $g$ parametrized tropical curves obtained in this way.

\begin{lem}
Parametrized 
tropical curves
$\tilde{h} \colon \tilde{\Gamma} \rightarrow \R^2$ in $T_{\Delta, p}^g$
are rigid. Furthermore, for such $\tilde{h}$, we have 
$n_{\tilde{h}}=1$ and $|\mathrm{Aut} (\tilde{h})|=1$.
\end{lem}

\begin{proof}
The rigidity of parametrized tropical curves in $T_{\Delta,p}^g$ follows from
the rigidity of parametrized tropical curves in 
$T_{\Delta,p}$ because the genera attached to the vertices cannot change under a deformation preserving the 
combinatorial type, and added bivalent vertices to go from $\Gamma$ to 
$\tilde{\Gamma}$ are mapped to vertices of 
$\cP_{\Delta,p}$ and so cannot move without changing the combinatorial type.

We have $n_{\tilde{h}}=1$ because
in Section 
\ref{subsection_toric_degeneration}, we have chosen the polyhedral decomposition
$\cP_{\Delta,p}$
to be integral: vertices of $\tilde{h}$ map to integral points of 
$\R^2$ and edges $E$ of
$\tilde{\Gamma}$ have integral lengths $\ell(E)$. 
We have $|\mathrm{Aut} (\tilde{h})|=1$ because $\tilde{h}$ is an immersion.
The genus of vertices never enters in the above arguments.
\end{proof}

For every $\tilde{h} \colon \tilde{\Gamma} \rightarrow \R^2$
parametrized tropical curve
in 
$T_{\Delta ,p}^g$, we define 
\[N^{\Delta ,n}_{g,\tilde{h}} \coloneqq
\int_{[\overline{M}_{g,n,\Delta}^{\tilde{h},P^0}]^{\virt}}
(-1)^{g-g_{\Delta ,n}} 
\lambda_{g-g_{\Delta ,n}} \,.\]

\begin{prop} \label{prop_decomposition}
For every $\Delta$, $n$ and $g \geqslant g_{\Delta , n}$, we
have 
\[ N^{\Delta ,n}_g = \sum_{\tilde{h} \in T_{\Delta, p}^g}
N^{\Delta, n}_{g, \tilde{h}} \,.\]

\end{prop}

\begin{proof}
This follows from the decomposition formula and from the vanishing property of lambda classes.

If $\tilde{h}$ is a rigid parametrized tropical curve of genus $g >g_{\Delta, n}$, then every 
point in $\overline{M}_{g,n,\Delta}^{\tilde{h},P^0}$
is a stable log map whose tropicalization has  genus $g>g_{\Delta ,n}$. In particular, the dual
intersection complex of the source curve has genus 
$g > g_{\Delta, n}$.
By Lemma \ref{lem_gluing2}, $\lambda_{g-g_{\Delta ,n}}$ is zero on restriction to such family of curves.
\end{proof}

\textbf{Example}.
The generic way to deform a parametrized tropical curve in 
$T^g_{\Delta, p}$ is to open 
$g(V)$ small cycles in place of a vertex of genus 
$g(V)$. When the cycles coming from various vertices grow and meet, we can obtain curves with 
vertices of valence strictly 
greater than three which can be rigid. \mbox{Proposition \ref{prop_decomposition}}
guarantees that such rigid curves do not contribute in the decomposition formula after integration of the lambda class.

Below is an illustration of a genus one vertex opening in 
one cycle and growing until forming a 4-valent vertex.

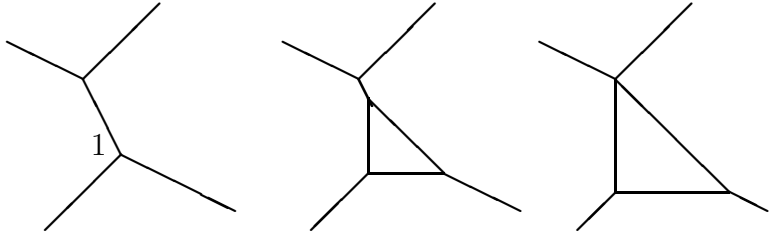
\begin{center}
\setlength{\unitlength}{1cm}
\begin{picture}(10,4)
\thicklines
\put(1,1){\line(-1,2){0.5}}
\put(1,1){\line(2,-1){1.5}}
\put(1,1){\line(-1,-1){1}}
\put(0.5,2){\line(1,1){1}}
\put(0.5,2){\line(-2,1){1}}
\put(0.6,1){$1$}

\put(4.25,0.75){\line(-1,-1){0.75}}
\put(4.125,2){\line(1,1){1}}
\put(4.125,2){\line(-2,1){1}}
\put(4.125,2){\line(1,-2){0.176}}
\put(4.25,0.75){\line(0,1){1}}
\put(4.25,0.75){\line(1,0){1}}
\put(5.25,0.75){\line(-1,1){1}}
\put(5.25,0.75){\line(2,-1){1}}

\put(7.5,0.5){\line(-1,-1){0.5}}
\put(7.5,2){\line(1,1){1}}
\put(7.5,2){\line(-2,1){1}}
\put(7.5,0.5){\line(0,1){1.5}}
\put(7.5,0.5){\line(1,0){1.5}}
\put(9,0.5){\line(-1,1){1.5}}
\put(9,0.5){\line(2,-1){0.5}}

\end{picture}
\end{center}

\section{Non-torically transverse stable log maps in $X_\Delta$}
\label{section_transverse}

Let $\Delta$ be a
balanced collections of vectors 
in $\Z^2$,
as in Section \ref{notations}.
We consider the toric surface
$X_\Delta$ with the toric 
divisorial log structure.
In this Section, we prove some general 
properties of stable log maps of type $\Delta$ in 
$X_\Delta$, using as tool the tropicalization procedure 
reviewed in 
Section \ref{trop}.

We say that a stable log map 
$(f \colon C/
 \mathrm{pt}_{\overline{\cM}}
\rightarrow X_\Delta)$ to $X_\Delta$ is  torically transverse\footnote{We allow a torically transverse stable 
log map to have components
contracted to 
points of $\partial X_\Delta$ which are not torus fixed points. 
In particular, we use a notion of torically transverse map which is slightly 
different from
the one used by Nishinou and Siebert in \cite{MR2259922}. } 
if its image does not contain any of the torus 
fixed points of $X_\Delta$, i.e.\ 
if its image does not pass through the ``corners'' of the toric boundary
divisor $\partial X_\Delta$. 
The difficulty of log Gromov-Witten theory, with respect to relative 
Gromov-Witten theory for example, comes from the stable log maps which are not torically transverse:
the ``corners''
of $\partial X_\Delta$ are the points where $\partial X_\Delta$ 
is not smooth and so are exactly the points where the log structure of $X_\Delta$ 
is locally more complicated that the divisorial log structure along a 
smooth divisor.

The following Proposition is a structure result for stable log maps of type $\Delta$ 
which are not torically transverse.
Combined with vanishing properties of lambda classes reviewed
in Section 
\ref{section_lambda}, 
this will give us in Section \ref{proof_gluing} a way to completely discard 
stable log maps which are not torically transverse.

\begin{prop} \label{cycle}
Let $f \colon C/\mathrm{pt}_{\overline{\cM}}
\rightarrow X_\Delta$ be a stable 
log map to $X_\Delta$
of type $\Delta$.
Let $\Sigma(f) \colon \Sigma(C)/\Sigma (\mathrm{pt}_{\NN})  \rightarrow \Sigma
(X_\Delta)$ be the family of tropical curves obtained as tropicalization of $f$.
Assume that $f$ is not torically transverse and that the unbounded edges of the fibers 
of $\Sigma(f)$ are mapped to rays of 
the fan of $X_\Delta$. 
Then the dual graph of $C$ has positive genus,
i.e.\ $C$ contains at least one 
non-separating node.
\end{prop}

\begin{proof}

Recall that $\Sigma (f)$ is a family
over the 
cone 
$\Sigma (\mathrm{pt}_{\NN})=\Hom (\overline{\cM}, \R_{\geqslant 0})$ 
of parametrized tropical curves in $\R^2$.
We assume that the
unbounded edges of 
these parametrized tropical curves are mapped to rays of the fan of $X_\Delta$.

We fix a point in the interior of the cone 
$\Hom (\overline{\cM}, \R_{\geqslant 0})$ and we consider the corresponding parametrized tropical curve 
$h \colon \Gamma \rightarrow \R^2$
in $\R^2$. 
Combinatorially, $\Gamma$ is the dual graph 
of $C$.

\begin{lem} \label{lemma}
There exists a 
vertex $V$ of $\Gamma$
mapping away from the origin in $\R^2$
and a non-contracted edge $E$
adjacent to $V$ such that 
$h(E)$ is not included in a ray of the fan of $X_\Delta$.

\end{lem}

\begin{proof}
We are assuming that $f$ is not torically transverse. This means that at least one 
component of $C$ maps dominantly to a component of the toric boundary divisor $\partial X_\Delta$
or that at least one component of $C$
is contracted to a torus
fixed point of 
$X_\Delta$.

If one component of $C$ is contracted to a torus fixed point of $X_\Delta$, then we are done because the 
corresponding vertex $V$ 
of 
$\Gamma$ is mapped away from the origin and from the rays of the fan of $X_\Delta$, 
and any non-contracted edge of $\Gamma$ adjacent to $V$ is not mapped to a ray of the fan of $X_\Delta$.
Remark that there exists such non-contracted edge 
because if not, as 
$\Gamma$ is connected, 
all the vertices of 
$\Gamma$ would be mapped to $h(V)$ and so the curve $C$ 
would be entirely contracted to a torus fixed point, contradicting $\beta_\Delta \neq 0$.

So we can assume that no
component of $C$ is contracted to a torus 
fixed point,
i.e.\
that all the vertices of 
$\Gamma$ are mapped either to the origin or to 
a point on a ray of the fan 
of $X_\Delta$, and that at least one component of $C$
maps dominantly to a component of 
$\partial X_\Delta$.
We argue by contradiction
by assuming further that  
every edge of $\Gamma$ 
is  either contracted to a point or mapped inside a ray of the fan of $X_\Delta$.

Let $\Gamma_0$ be the subgraph of $\Gamma$ formed by vertices mapping 
to the origin and edges between them. For every ray $\rho$ of the fan of $X_\Delta$, let 
$\Delta_\rho$ be the set of 
$v \in \Delta$ such that 
$\R_{\geqslant 0} v=\rho$, and let 
$\Gamma_\rho$ be the subgraph of $\Gamma$ formed by vertices of $\Gamma$
mapping to the ray $\rho$ away from the origin and the edges between them.

By our assumption, there is no edge in $\Gamma$ connecting $\Gamma_\rho$
and $\Gamma_{\rho'}$ for two different rays 
$\rho$ and $\rho'$. 
For every ray $\rho$, let 
$E(\Gamma_0, \Gamma_\rho)$ the set of edges of 
$\Gamma$ connecting a vertex
$V_0(E)$ of $\Gamma_0$ and a vertex $V_\rho(E)$ of $\Gamma_\rho$.
It follows from the balancing condition that, for every ray $\rho$, we have
\[\sum_{E \in E(\Gamma_0,
\Gamma_\rho)}
v_{V_0(E),E} 
= \sum_{v \in \Delta_\rho} 
v \,.\]

Let $C_0$ be the curve obtained by taking the components of $C$ 
intersecting properly the toric boundary divisor $\partial X_\Delta$.
The dual graph of $C_0$
is $\Gamma_0$ and the total 
intersection number of $C_0$
with the toric divisor $D_\rho$ is 
\[\sum_{E \in E(\Gamma_0,
\Gamma_\rho)}
|v_{V_0(E),E}|\,,\]
where $|v_{V_0(E),E}|$ is the divisibility of 
$v_{V_0(E),E}$ in $\Z^2$, 
i.e.\ the multiplicity  of the corresponding intersection point of $C_0$ and $D_\rho$.

From the previous equality, we obtain that the intersection numbers of 
$C_0$ with the components of  $\partial X_\Delta$ 
are equal to the intersection numbers of $C$ with the components of 
$\partial X_\Delta$ so $[f(C_0)]=\beta_\Delta$.
It follows that all the components of $C$ not in 
$C_0$ are contracted,
which contradicts the fact that at least one component of $C$ maps dominantly 
to a component of 
$\partial X_\Delta$.
\end{proof}

We continue the proof
of Proposition 
\ref{cycle}.
By Lemma
\ref{lemma}, there exists 
a 
vertex $V$ of $\Gamma$
mapping away from the origin in $\R^2$
and a non-contracted edge $E$
adjacent to $V$ such that 
$h(E)$ is not included in a ray of the fan of $X_\Delta$.
We will use $(V,E)$ as
initial data for a recursive construction of a 
non-trivial cycle in $\Gamma$.

There exists a unique 
two-dimensional cone of the 
fan of $X_\Delta$,
containing 
$h(V)
\in \R^2-\{0\}$ and delimited
by rays 
$\rho_1$ and $\rho_2$, 
such that the rays 
$\rho_1$, 
$\R_{\geqslant 0}
h(V)$ and $\rho_2$
are ordered in the clockwise way and such that 
$h(V) \in \rho_1$ if 
$h(V)$ is on a ray.
Let $v_1$ and $v_2$
be vectors 
in $\R^2-\{0\}$ 
such that 
$\rho_1
=\R_{\geqslant 0}v_1$
and 
$\rho_2=\R_{\geqslant 0}v_2$.
The vectors $v_1$
and $v_2$ form a basis 
of $\R^2$ and for every 
$v \in \R^2$, we write 
$(v,v_1)$ and $(v,v_2)$
for the coordinates of
$v$ in this basis, 
i.e.\ the real numbers such that 
\[v=(v,v_1)v_1+(v,v_2)v_2 \,.\]
By construction, we have 
$(h(V), v_1) >
0$ and $(h(V), v_2) \geqslant 0$.
As $v_{V,E} \neq 0$, we have 
$(v_{V,E}, v_1) \neq 0$ or 
$(v_{V,E}, v_2) 
\neq 0$.

If
$(v_{V,F}, v_2)=0$ for every edge $F$ adjacent to $V$, then 
$(v_{V,E}, v_1) \neq 0$ and $(h(V), v_2)>0$.
In particular, $E$ is not an unbounded edge. 
By the balancing condition, up to replacing $E$ by another edge adjacent to $V$, one can assume that 
$(v_{V,E}, v_1)>0$. 
Then, the  
edge $E$ is adjacent
to another vertex 
$V'$ with $(h(V'), v_1)>(h(V),v_1)$
and $(h(V'),v_2)
=(h(V),v_2)$.
By the balancing condition, there exists an edge $E'$ adjacent to 
$V'$ such that 
$(v_{V',E'}, v_1)>0.$
If $(v_{V,F'}, v_2)=0$ for every edge $F'$ adjacent to $V'$, 
then in particular we have $(v_{V,E'}, v_2)=0$ and so $E'$ is adjacent to another vertex $V''$ 
with $(h(V''), v_1)>(h(V'), v_1)$
and $(h(V''), v_2)=(h(V'), v_2)$, 
and we can iterate the argument. Because $\Gamma$ has finitely many vertices, this process has to stop: 
there exists a vertex $\tilde{V}$ in the cone generated by 
$\rho_1$ and $\rho_2$ and an edge $\tilde{E}$ adjacent to 
$\tilde{V}$ such that $(v_{\tilde{V},\tilde{E}}, v_2) \neq 0$.

The upshot of the previous paragraph is that, up to changing $V$ and $E$, one can assume that 
$(v_{V,E}, v_2) \neq 0$.
By the balancing condition, up to replacing $E$
by another edge adjacent to $V$, one can assume that 
$(v_{V,E}, v_2)>0$. The  edge $E$ is adjacent to another vertex $V'$ with 
$(h(V'),v_2)>(h(V),
v_2)$.
By the balancing condition, one can find an edge $E'$ adjacent to $V'$ such that 
$(v_{V',E'}, v_2)>0$. 
If $h(V')$ is in the interior of the cone generated by 
$\rho_1$ and 
$\rho_2$, then 
$E'$ is not an unbounded edge and so is adjacent to another vertex $V''$ with 
$(h(V''),v_2)>
(h(V'),v_2)$.
Repeating this construction, we obtain a sequence of vertices of image in the cone generated by 
$\rho_1$ and $\rho_2$.
Because $\Gamma$ has finitely many vertices, this process has to terminate: there exists a vertex $\tilde{V}$
of $\Gamma$ such that 
$h(\tilde{V})
\in \rho_2$ and connected to $V$ by a path of edges mapping to the interior 
of the cone delimited 
by \mbox{$\rho_1$ and 
$\rho_2$.}

Repeating the argument starting from
$\tilde{V}$, and so on, we construct a path of edges in $\Gamma$ whose projection in 
$\R^2$ intersects 
successive rays in the clockwise
order. Because the combinatorial type of $\Gamma$ is finite, 
this path has to close eventually and so $\Gamma$ contains a non-trivial closed cycle, i.e.\ $\Gamma$
has positive genus.
\end{proof}

\textbf{Remark:} It follows from 
Proposition \ref{cycle} that the ad hoc
genus zero invariants 
defined in terms of relative 
Gromov-Witten invariants of some open 
geometry 
used by Gross, Pandharipande, Siebert  
in
\cite{MR2667135}(Section 4.4), 
and Gross, Hacking, Keel
in
\cite{MR3415066} (Section 3.1), coincide with
log Gromov-Witten invariants\footnote{This result was expected: see Remark 3.4 of \cite{MR3415066}  
but it seems that no proof 
was published until now.}.
In fact, our proof of Proposition 
\ref{cycle} can be seen as a tropical analogue of the main properness argument of 
\cite{MR2667135}
(Proposition 4.2) which guarantees that the 
ad hoc invariants are well-defined.

\section{Statement of the gluing formula}
\label{section_statement_gluing}

We continue the proof of 
Theorem \ref{main_thm1}
started in Section
\ref{section: decomposition}.
In \mbox{Section \ref{section_statement_gluing}}, we state a gluing formula,
Corollary \ref{cor_gluing},
expressing the invariants 
$N_{g,\tilde{h}}^{\Delta,n}$
attached to a parametrized tropical curve
$\tilde{h} \colon \tilde{\Gamma} \rightarrow \R^2$ in terms of invariants $N^{1,2}_{g,V}$
attached to the vertices $V$ of 
$\Gamma$. This gluing formula is proved 
in Section \ref{proof_gluing}, using the structure result of
Section \ref{section_transverse} and
the vanishing result of Section
\ref{section_lambda} to reduce the argument
to the locus of torically transverse stable log maps.

\subsection{Preliminaries}

We fix 
$\tilde{h} \colon \tilde{\Gamma} 
\rightarrow 
\R^2$ a 
parametrized tropical curve in $T_{\Delta, p}^g$.
The purpose of the gluing 
formula is to write the log Gromov-Witten invariant 
\[N^{\Delta ,n}_{g,\tilde{h}} = 
\int_{[\overline{M}_{g,n,\Delta}^{\tilde{h},P^0}]^{\virt}}
(-1)^{g-g_{\Delta ,n}} 
\lambda_{g-g_{\Delta ,n}} \,,\]
introduced in Section
\ref{subsection_decomposition_formula},
in terms of log Gromov-Witten invariants of the toric surfaces $X_{\Delta_V}$
attached to the vertices 
$V$ of $\tilde{\Gamma}$.
Recall from Section \ref{subsection_toric_degeneration} that 
$\tilde{\Gamma}$ has three types of vertices:
\begin{itemize}
\item Trivalent unpointed vertices, coming from $\Gamma$.
\item Bivalent pointed vertices, coming 
from $\Gamma$.
\item Bivalent unpointed vertices, not coming from $\Gamma$.
\end{itemize}

According to 
Lemma 4.20 
of Mikhalkin
\cite{MR2137980}, the connected components of the 
complement of the bivalent pointed vertices of 
$\tilde{\Gamma}$ are trees with exactly one unbounded edge.

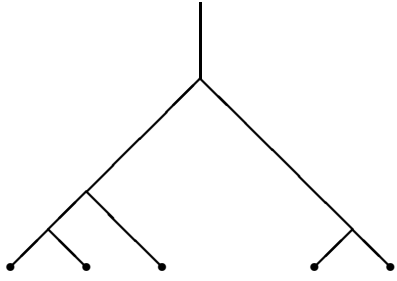
\begin{center}
\setlength{\unitlength}{1cm}
\begin{picture}(10,5)
\thicklines
\put(2,0.5){\circle*{0.1}}
\put(3,0.5){\circle*{0.1}}
\put(4,0.5){\circle*{0.1}}
\put(6,0.5){\circle*{0.1}}
\put(7,0.5){\circle*{0.1}}
\put(2,0.5){\line(1,1){0.5}}
\put(3,0.5){\line(-1,1){0.5}}
\put(6,0.5){\line(1,1){0.5}}
\put(7,0.5){\line(-1,1){0.5}}
\put(2.5,1){\line(1,1){0.5}}
\put(4,0.5){\line(-1,1){1}}
\put(3,1.5){\line(1,1){1.5}}
\put(6.5,1){\line(-1,1){2}}
\put(4.5,3){\line(0,1){1}}
\end{picture}
\end{center}

In particular, we can fix an orientation of edges of $\tilde{\Gamma}$
consistently from the bivalent pointed vertices to the unbounded edges. 
Every trivalent vertex of $\tilde{\Gamma}$ has two ingoing and one outgoing edges with respect to this orientation. Every bivalent pointed vertex has two outgoing edges with respect to this orientation. Every bivalent unpointed vertex has one ingoing and one outgoing edges with respect to this orientation. 

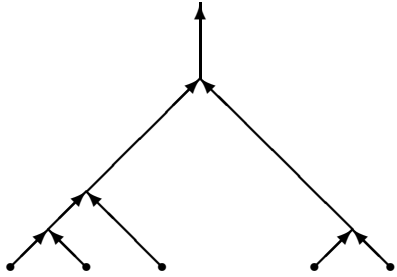
\begin{center}
\setlength{\unitlength}{1cm}
\begin{picture}(10,5)
\thicklines
\put(2,0.5){\circle*{0.1}}
\put(3,0.5){\circle*{0.1}}
\put(4,0.5){\circle*{0.1}}
\put(6,0.5){\circle*{0.1}}
\put(7,0.5){\circle*{0.1}}
\put(2,0.5){\vector(1,1){0.5}}
\put(3,0.5){\vector(-1,1){0.5}}
\put(6,0.5){\vector(1,1){0.5}}
\put(7,0.5){\vector(-1,1){0.5}}
\put(2.5,1){\vector(1,1){0.5}}
\put(4,0.5){\vector(-1,1){1}}
\put(3,1.5){\vector(1,1){1.5}}
\put(6.5,1){\vector(-1,1){2}}
\put(4.5,3){\vector(0,1){1}}
\end{picture}
\end{center}

\subsection{Contribution of trivalent vertices}
Let $V$ be a trivalent vertex 
of $\tilde{\Gamma}$.
Let $\overline{M}_{g,\Delta_V}$ be the moduli space of 
stable log maps to $X_{\Delta_V}$ of genus $g$ and of type $\Delta_V$. 
It has virtual dimension
\[\vdim \overline{M}_{g,\Delta_V}=g+2 \,, \]
and admits a virtual fundamental class
\[ [\overline{M}_{g,\Delta_V}]^{\virt} \in 
A_{g+2}(\overline{M}_{g,\Delta_V}, \Q).\]
Let $E_V^{\mathrm{in},1}$ and 
$E_V^{\mathrm{in},2}$
be the two ingoing edges adjacent to $V$, and let 
$E_V^{\mathrm{out}}$ be the outgoing edge adjacent to $V$.
Let $D_{E_V^{\mathrm{in},1}}$, $D_{E_V^{\mathrm{in},2}}$
and $D_{E_V^{\mathrm{out}}}$ be the corresponding toric divisors of
$X_{\Delta_V}$.
We have evaluation morphisms
\[ (\text{ev}_V^{E_V^{\mathrm{in},1}}, 
\text{ev}_V^{E_V^{\mathrm{in},2}}, 
\text{ev}_V^{E_V^{\mathrm{out}}}) \colon \overline{M}_{g,\Delta_V} 
\rightarrow  D_{E_V^{\mathrm{in},1}}
\times
D_{E_V^{\mathrm{in},2}}
\times 
D_{E_V^{\mathrm{out}}} \,.\]
We define 
\[ N_{g,V}^{1, 2} \coloneqq \int_{[\overline{M}_{g,\Delta_V}]^{\virt}}(-1)^g \lambda_g 
(\text{ev}_V^{E_V^{\mathrm{in},1}})^* 
(\mathrm{pt}_{E_V^{\mathrm{in},1}})
(\text{ev}_V^{E_V^{\mathrm{in},2}})^* (\mathrm{pt}_{E_V^{\mathrm{in},2}}) \,,\]
where $\mathrm{pt}_{E_V^{\mathrm{in},1}} \in A^1 (D_{E_V^{\mathrm{in},1}})$ and $\mathrm{pt}_{E_V^{\mathrm{in},2}} \in A^1(D_{E_V^{\mathrm{in},2}})$ are 
classes of a point on 
$D_{E_V^{\mathrm{in},1}}$ and
$D_{E_V^{\mathrm{in},2}}$
respectively.

\subsection{Contribution of bivalent pointed vertices}
Let $V$ be a bivalent pointed vertex 
of $\tilde{\Gamma}$.
Let $\overline{M}_{g,\Delta_V}$ be the moduli space of \mbox{$1$-pointed}\footnote{As in Section 
\ref{log_GW}, $1$-pointed means that the source curves are equipped with one marked point in addition to the marked points keeping track of the tangency conditions.} 
stable log maps to $X_{\Delta_V}$ of genus $g$ and of type $\Delta_V$. 
It has virtual dimension
\[\vdim \overline{M}_{g,\Delta_V}=g+2 \,, \]
and admits a virtual fundamental class
\[ [\overline{M}_{g,\Delta_V}]^{\virt} \in 
A_{g+2}(\overline{M}_{g,\Delta_V}, \Q).\]
We have the evaluation morphism at the extra marked point,
\[\ev \colon \overline{M}_{g,\Delta_V}
\rightarrow X_{\Delta_V} \,,\]
and we define 
\[N_{g,V}^{1,2} \coloneqq \int_{[\overline{M}_{g,\Delta_V}]^{\virt}}(-1)^g \lambda_g \text{ev}^* (\mathrm{pt}) \,,\] 
where $\mathrm{pt} \in A^2 (X_{\Delta_V})$ is the class of a point on $X_{\Delta_V}$.

\subsection{Contribution of bivalent unpointed vertices}

Let $V$ be a bivalent unpointed vertex 
of $\tilde{\Gamma}$.
Let $\overline{M}_{g,\Delta_V}$ be the moduli space of  
stable log maps to $X_{\Delta_V}$ of genus $g$ and of type $\Delta_V$. 
It has virtual dimension
\[\vdim \overline{M}_{g,\Delta_V}=g+1 \,, \]
and admits a virtual fundamental class
\[ [\overline{M}_{g,\Delta_V}]^{\virt} \in 
A_{g+1}(\overline{M}_{g,\Delta_V}, \Q).\]
Let $E_V^{\mathrm{in}}$
be the ingoing edge adjacent to $V$ and
$E_V^{\mathrm{out}}$ the outgoing edge adjacent to $V$. 
Let
$D_{E_V^{\mathrm{in}}}$ and  $D_{E_V^{\mathrm{out}}}$
be the corresponding toric divisors of
$X_{\Delta_V}$. 
We have evaluation morphisms 
\[ (\mathrm{ev}_V^{E_V^{\mathrm{in}}}, \mathrm{ev}_V^{E_V^{\mathrm{out}}} ) \colon \overline{M}_{g,\Delta_V} \rightarrow  D_{E_V^{\mathrm{in}}}\times D_{E_V^{\mathrm{out}}} \,.\] 
We define 
\[ N_{g,V}^{1,2} \coloneqq \int_{[\overline{M}_{g,\Delta_V}]^{\virt}}(-1)^g \lambda_g 
(\mathrm{ev}_V^{E_V^{\mathrm{in}}})^* (\mathrm{pt}_{E_V^{\mathrm{in}}}) \,,\] where $\mathrm{pt}_{E_V^{\mathrm{in}}} \in A^1 (D_{E_V^{\mathrm{in},1}})$ is the  class of a point on $D_{E_V^{\text{in}}}$.

\subsection{Statement of the gluing formula}
The following gluing formula expresses the log 
Gromov-Witten invariant 
$N^{\Delta, n}_{g, \tilde{h}}$ attached to a parametrized tropical curve 
$\tilde{h} \colon \tilde{\Gamma} \rightarrow \R^2$ in terms of the log Gromov-Witten invariants 
$N^{1,2}_{g,V}$
attached to the vertices $V$ of $\tilde{\Gamma}$
and of the weights $w(E)$ of the edges of 
$\tilde{\Gamma}$. 

\begin{prop}
\label{prop_gluing1}
For every 
$\tilde{h} \colon \tilde{\Gamma} 
\rightarrow \R^2$ 
parametrized tropical curve in 
$T_{\Delta, p}^g$, we have 
\[N_{g, \tilde{h}}^{\Delta, n}
= 
\left(
\prod_{V
\in V(\tilde{\Gamma})} N_{g(V), V}^{1, 2} 
\right)
\left( 
\prod_{E \in E_f(\tilde{\Gamma})} w(E)
\right) \,,\]
where the first product is over the 
vertices of 
$\tilde{\Gamma}$ and the second product is over the bounded edges of 
$\tilde{\Gamma}$.
\end{prop}

The proof of Proposition 
\ref{prop_gluing1} is given in Section
\ref{proof_gluing}.

In the following Lemmas, we compute the contributions 
$N^{1,2}_{g(V), V}$ of the bivalent vertices.

\begin{lem} \label{lem_bivalent_p}
Let $V$ be a bivalent pointed vertex of 
$\tilde{\Gamma}$.
Then we have 
\[N^{1,2}_{g,V}=0 \]
for every $g >0$, and 
\[ N^{1,2}_{0,V}=1 \]
for $g=0$.
\end{lem}

\begin{proof}
Let $w$ be the weight of the two edges of
$\tilde{\Gamma}$ adjacent to $V$.
We can take 
$X_{\Delta_V}=\PP^1 \times 
\PP^1$ and 
$\beta_{\Delta_V}
= w ([\PP^1] \times [\mathrm{pt}])$.
We have the evaluation map at the extra marked point
\[\text{ev} \colon 
\overline{M}_{g, \Delta_V}
\rightarrow \PP^1
\times \PP^1 \,.\]
We fix a point 
$p=(p_1,p_2) \in \C^* \times 
\C^* \subset \PP^1 \times \PP^1$ 
and we 
denote 
$\iota_p \colon p \hookrightarrow 
\PP^1 \times \PP^1$
and $\iota_{p_1} \colon p 
\hookrightarrow \PP^{1} \times
\{p_2\} \simeq \PP^1$ the 
inclusion morphisms.

Let
$\overline{M}_{g,1}(\PP^1/\{0\}
\cup \{ \infty \}, w; w, w)$
be the moduli space of genus $g$ $1$-pointed stable maps
to $\PP^1$, of degree $w$, relative to the divisor $\{0\} \cup \{ \infty\}$, with intersection multiplicities
$w$ both along $\{0\}$ and 
$\{\infty \}$. We have an
evaluation morphism at the extra marked 
point
\[ \ev_1 \colon \overline{M}_{g,1}(\PP^1/\{0\}
\cup \{ \infty \}, w; w, w) \rightarrow \PP^1 \,,\]

Because an element $(f \colon C \rightarrow \PP^1 \times \PP^1)$ of $\ev^{-1}(p)$
factors through 
$\PP^1 \times \{p_2\} \simeq \PP^1$,
we have a natural identification of moduli spaces $\ev^{-1}(p)=\ev_1^{-1}(p)$,
but the natural
virtual fundamental classes are different.
The class $\iota_p^! [\overline{M}_{g, \Delta_V}]^{\virt}$, defined
by the refined Gysin homomorphism
(see Section 6.2 of \cite{MR1644323}),
has degree $g$ whereas the class
$\iota_{p_1}^![\overline{M}_{g,1}(\PP^1/\{0\}
\cup \{ \infty \}, w ; w, w)]^{\virt}$
is of degree
\[ 2g-2+2w-(w-1)-(w-1)+(1-1)=2g\,. \]
The two obstruction theories differ by the bundle whose fiber at 
\[ f \colon C \rightarrow \PP^1\]
is $H^1(C, f^*N_{f(C)|\PP^1 \times \PP^1})$.
Because $\beta_{\Delta_V}^2=0$, the normal bundle 
$N_{f(C)|\PP^1 \times \PP^1}$ is trivial of rank one, so the pullback 
$f^*N_{f(C)|\PP^1 \times \PP^1}$ is trivial
of rank one
and the two obstruction theories differ by 
the dual of the Hodge bundle. 
Therefore, we have
\[ \iota_p^! [\overline{M}_{g, \Delta_V}]^{\virt}= c_g(\E^*) \cap
\iota_{p_1}^![\overline{M}_{g,1}(\PP^1/\{0\}
\cup \{ \infty \}, w; w, w)]^{\virt}\,,\]
and so 
\[ N^1_{g,V}=\int_{\iota_p^! [\overline{M}_{g, \Delta_V}]^{\virt}}  (-1)^g \lambda_g
=\int_{\iota_{p_1}^![\overline{M}_{g,1}(\PP^1/\{0\}
\cup \{ \infty \}, w; w, w)]^{\virt}} \lambda_g^2 \,.\]
But $\lambda_g^2=0$ for $g>0$, as follows from 
Mumford's relation \cite{MR717614} 
\[ c(\E)c(\E^*)=1\,,\]
and so $N^1_{g,V}=0$ if $g >0$.

If $g=0$, we have $\lambda_0^2=1$, the moduli space is a point, given by the degree $w$ map $\PP^1 \rightarrow \PP^1$ 
fully ramified over $0$ and $\infty$,
with trivial automorphism group
(there is no non-trivial automorphism of $\PP^1$ fixing $0$, $\infty$ and the extra marked point), and so \[ N^{1,2}_{0,V}=1\,.\] 
\end{proof}

\begin{lem} \label{lem_bivalent_up}
Let $V$ be a bivalent unpointed vertex of 
$\tilde{\Gamma}$ and $w(E_V)$ the common weight of the two edges adjacent to $V$.
Then we have 
\[N^{1,2}_{g,V}=0 \]
for every $g >0$, and 
\[ N^{1,2}_{0,V}=\frac{1}{w(E_V)} \]
for $g=0$.
\end{lem}

\begin{proof}
The argument is parallel to the one used to prove Lemma \ref{lem_bivalent_p}.
The only difference is that the vertex is no longer pointed and the invariant $N_{g,V}^{1,2}$ is defined using the evaluation map at one of the tangency point.
The vanishing for $g>0$ still follows from $\lambda_g^2=0$. For $g=0$, the moduli space is a point, given by the degree $w(E_V)$ map $\PP^1 \rightarrow \PP^1$ fully ramified over $0$ and $\infty$, but now with an automorphism group $\Z/w(E_V)$
(the extra marked point in \mbox{Lemma
\ref{lem_bivalent_p}} is no longer there to kill all non-trivial automorphisms). It follows that 
$N^{1,2}_{0,V}=\frac{1}{w(E_V)}$.

\end{proof}

\begin{cor} \label{cor_gluing}
Let $\tilde{h} \colon \tilde{\Gamma} \rightarrow \R^2$
be a parametrized tropical curve in 
$T_{\Delta, p}^g$.

\begin{itemize}
\item If there exists one bivalent vertex $V$
of $\tilde{\Gamma}$ with $g(V) \neq 0$, then 
\[N^{\Delta, n}_{g,\tilde{h}}=0\,.\]
\item If $g(V)=0$ for all the bivalent vertices $V$ of $\tilde{\Gamma}$, then 
\[N_{g, \tilde{h}}^{\Delta, n}
= 
\left(
\prod_{V
\in V^{(3)}(\tilde{\Gamma})} N_{g(V), V}^{1, 2} 
\right)
\left( 
 \prod_{E \in E_f(\Gamma)} w(E)
 \right) \,,\]
where the first product is over the trivalent vertices of $\Gamma$ (or
$\tilde{\Gamma}$), 
and the 
second product is over the bounded edges of 
$\Gamma$ (not $\tilde{\Gamma}$).
\end{itemize}
\end{cor}

\begin{proof}
If $\tilde{\Gamma}$ has a bivalent vertex $V$ with $g(V)>0$, then, 
according to Lemmas \ref{lem_bivalent_p}
and \ref{lem_bivalent_up}, we have $N_{g(V),V}^{1,2}=0$ 
and so $N_{g,\tilde{h}}^{\Delta,n}=0$ by
Proposition \ref{prop_gluing1}.

If $g(V)=0$ for all the bivalent vertices 
$V$ of $\tilde{\Gamma}$, then,
according to Lemma \ref{lem_bivalent_p}, we have 
$N_{g(V),V}^{1,2}=1$ for all the bivalent pointed vertices $V$ of $\tilde{\Gamma}$
and according to Lemma
\ref{lem_bivalent_up}, we have
$N_{g(V),V}^{1,2}=\frac{1}{w(E_V)}$ for all the bivalent unpointed vertices $V$ of $\tilde{\Gamma}$ .
It follows that Proposition \ref{prop_gluing1}
can be rewritten 
\[N_{g,\tilde{h}}^{\Delta,n}
= \left( \prod_{V \in V^{(3)}(\tilde{\Gamma})} N_{g(V),V}^{1,2}
\right)
\left(
\prod_{V \in V^{(2up)}(\tilde{\Gamma})}
\frac{1}{w(E_V)} 
\right)
\left( 
\prod_{E \in E_f(\tilde{\Gamma})}
w(E)
\right) \,,\]
where the first product is over the trivalent vertices of 
$\tilde{\Gamma}$
(which can be naturally identified with the trivalent vertices of $\Gamma$)
and the second product is over the bivalent 
unpointed vertices
of $\tilde{\Gamma}$. 
Recalling from Section 
\ref{subsection_toric_degeneration} that the edges of $\tilde{\Gamma}$ are obtained as subdivision of the edges of $\Gamma$ by adding the bivalent unpointed vertices, we have 
\[\left( \prod_{V \in V^{(2up)}(\tilde{\Gamma})} \frac{1}{w(E_V)}
\right)
\left( \prod_{E \in E_f(\tilde{\Gamma})}
w(E) \right)
=\prod_{E \in E_f(\Gamma)} w(E)
\,.\] 
\end{proof}

\section{Proof of the gluing formula}
\label{proof_gluing}

This Section is devoted to the proof of 
Proposition 
\ref{prop_gluing1}.
Part of it is inspired the proof by Chen 
\cite{MR3166394}
of the degeneration formula for expanded stable log maps, and the proof by Kim, Lho
and Ruddat \cite{kim2018degeneration}
of the degeneration formula for stable log maps in
degenerations along a smooth divisor.
In Section \ref{section_cutting}, we define a cut morphism. Restricted to some open substack of torically transverse stable maps, we show in Section 
\ref{section_counting} that the cut morphism is étale, and in Section
\ref{section_comparing}, that the cut morphism is compatible with the natural obstruction theories of the pieces.
Using in addition Proposition \ref{cycle}
and the results of Section
\ref{section_lambda}, we prove a gluing formula in Section \ref{sub_gluing}. 
To finish the proof of Proposition
\ref{prop_gluing1}, 
we explain in Section \ref{section_pieces} 
how to organize the glued pieces.

\subsection{Cutting}
\label{section_cutting}
Let
$\tilde{h} \colon \tilde{\Gamma} 
\rightarrow \R^2$ be a
parametrized tropical curve in 
$T_{\Delta, p}^g$.
We denote $V^{(2p)}(\tilde{\Gamma})$ the set of bivalent pointed vertices of 
$\tilde{\Gamma}$ and 
$V^{(2up)}(\tilde{\Gamma})$ 
the set of bivalent unpointed vertices of 
$\tilde{\Gamma}$.

Evaluations 
$\ev_V^E
\colon 
\overline{M}_{g(V), \Delta_V}
\rightarrow
D_E$
at the tangency points dual to the bounded edges of $\tilde{\Gamma}$ give a morphism
\[\ev^{(e)} 
\colon 
\prod_{V \in V(\tilde{\Gamma})}\overline{M}_{g(V), \Delta_V}
\rightarrow
\prod_{E \in E_f(\tilde{\Gamma})}
(D_E)^2 \,,\]
where $D_E$ is the divisor of $X_0$
dual to an edge $E$ of 
$\tilde{\Gamma}$. 

Evaluations
$\ev^{(p)}_V
\colon 
\overline{M}_{g(V), \Delta_V}
\rightarrow X_{\Delta_V}$ 
at the extra marked points 
corresponding to the bivalent 
pointed vertices give a morphism 
\[\ev^{(p)}
\colon 
\prod_{V \in V(\tilde{\Gamma})}\overline{M}_{g(V), \Delta_V}
\rightarrow
\prod_{V \in V^{(2p)}(\tilde{\Gamma})} X_{\Delta_V}\,.\]

Let 
\[ \delta
\colon \prod_{E \in E_f(\tilde{\Gamma})}
D_E \rightarrow 
\prod_{E \in E_f(\tilde{\Gamma})}
(D_E)^2\]
be the diagonal morphism.
Let 
\[ \iota_{P^0} \colon 
\left(P^0 = (P_V^0)_{V \in V^{(2p)}(\tilde{\Gamma})} \right) \hookrightarrow 
\prod_{V \in V^{(2p)}(\tilde{\Gamma})}
X_{\Delta_V} \,,\]
be the inclusion morphism of $P^0$.

Using the fiber product diagram in the category of stacks 
\begin{center}
\begin{tikzcd}
\underset{V \in V(\tilde{\Gamma})}
{\bigtimes}
\overline{M}_{g(V), \Delta_V} \arrow{r}{(\delta \times \iota_{P^0})_M} \arrow{d}{\text{}}
& \underset{V \in V(\tilde{\Gamma})}{\prod}
\overline{M}_{g(V), \Delta_V} 
\arrow{d}{\text{ev}^{(e)} \times \text{ev}^{(p)}} \\
\left(\underset{E \in E_f(\tilde{\Gamma})}{\prod}
D_E \right) \times P^0 
\arrow{r}{\delta
\times \iota_{P^0}} &
\underset{E \in E_f(\tilde{\Gamma})}
{\prod}
(D_E)^2
\times \underset{V \in V^{(2p)}(\tilde{\Gamma})}{\prod} 
X_{\Delta_V} 
\,,
\end{tikzcd}
\end{center}
we define the substack
$\bigtimes_{V \in V(\tilde{\Gamma})}\overline{M}_{g(V), \Delta_V}$
of
$\prod_{V \in V(\tilde{\Gamma})}\overline{M}_{g(V), \Delta_V}$ 
consisting of curves whose marked points keeping track of the tangency conditions match over the divisors $D_E$
and whose extra marked points associated to the bivalent pointed vertices map to $P^0$.

\begin{lem} \label{lem_edges}
Let 
\begin{center}
\begin{tikzcd}
C \arrow{r}{f} \arrow{d}{\pi}
& X_0 \arrow{d}{\nu_0}\\
\mathrm{pt}_{\overline{\cM}} \arrow{r}{g} & \mathrm{pt}_{\NN} \,,
\end{tikzcd}
\end{center}
be a $n$-pointed genus $g$ stable log map
of type $\Delta$ passing through $P^0$
and marked by 
$\tilde{h} \colon \tilde{\Gamma} \rightarrow \R^2$, 
i.e.\ a
point of 
$\overline{M}^{\tilde{h},P^0}_{g,n,\Delta}$.
Let 
\begin{center}
\begin{tikzcd}
\Sigma(C) \arrow{r}{\Sigma(f)} \arrow{d}{\Sigma(\pi)}
& \Sigma(X_0) \arrow{d}{\Sigma(\nu_0)}\\
\Sigma(\mathrm{pt}_{\overline{\cM}}) \arrow{r}{\Sigma(g)} & \Sigma(\mathrm{pt}_{\NN}) \,.
\end{tikzcd}
\end{center}
be its tropicalization.
For every $b \in \Sigma(g)^{-1}(1)$,
let 
\[ \Sigma(f)_b \colon
\Sigma(C)_b \rightarrow \Sigma(\nu_0)^{-1}
(1) \simeq \R^2 \]
be the fiber of $\Sigma(f)$ over $b$.
Let $E$ be an edge of $\Gamma$ and let 
$E_{f,b}$ be the edge of 
$\Sigma(C)_b$ marked by $E$.
Then 
$\Sigma(f)_b(E_{f,b}) \subset \tilde{h}(E)$.
\end{lem}

\begin{proof}
We recalled in Section \ref{section_statement_gluing} that 
the connected components of the complement of the bivalent pointed vertices of 
$\tilde{\Gamma}$ are trees with exactly one unbounded edge.
We prove Lemma 
\ref{lem_edges} by induction, starting with 
the edges connected to the bivalent pointed vertices and then we go through each tree
following the orientation introduced
in Section
\ref{section_statement_gluing}.

Let $E$ be an edge of $\tilde{\Gamma}$ adjacent to 
a bivalent pointed vertex $V$ of $\tilde{\Gamma}$. 
Let $P^0_V \in X_{\Delta_V}$ be the corresponding marked point.
As $f$ is marked by $\tilde{h}$, we have an ordinary stable map $f_V \colon C_V
\rightarrow X_{\Delta_V}$, a marked
point $x_E$ in $C_V$ such that 
$f(x_E) \in D_E$ and $f_V(C_V)$
contains $P^0_V$.
We can assume that 
$X_{\Delta_V}=\PP^1 \times \PP^1$,
$D_E=\{0\} \times \PP^1$,
$\beta_{\Delta_V}=w(E)([\PP^1]
\times [\mathrm{pt}])$, and 
$P^0_V = (P^0_{V,1}, P^0_{V,2})
\in \C^{*}
\times \C^{*}
\subset \PP^1 \times \PP^1$.
Then 
$f_V$ factors through 
$\PP^1 \times \{P^0_{V,2}\}$ and 
$x_E =(0, P^0_{V,2})$.
It follows that 
$\Sigma (f)_b(E_{f,b}) \subset
\tilde{h}(E)$.

Let $E$ be the outgoing edge of 
a trivalent vertex of 
$\tilde{\Gamma}$, of ingoing edges 
$E^1$ and $E^2$. By the induction 
hypothesis, we know that 
$\Sigma (f)_b(E^1_{f,b}) \subset
\tilde{h}(E^1)$
and $\Sigma (f)_b(E^2_{f,b}) \subset
\tilde{h}(E^2)$. We conclude that 
$\Sigma (f)_b(E_{f,b}) \subset
\tilde{h}(E)$ by an application of the balancing
condition, as in 
Proposition 30 (tropical Menelaus theorem) of Mikhalkin 
\cite{mikhalkin2015quantum}.
\end{proof}

For a stable log map
\begin{center}
\begin{tikzcd}
C \arrow{r}{f} \arrow{d}{\pi}
& X_0 \arrow{d}{\nu_0}\\
\mathrm{pt}_{\overline{\cM}} \arrow{r}{g} & \mathrm{pt}_{\NN} 
\end{tikzcd}
\end{center}
marked by $\tilde{h}$, we have nodes of $C$
in correspondence with the 
bounded edges of $\tilde{\Gamma}$.
Cutting $C$ along these nodes, 
we obtain a morphism 
\[ \mathrm{cut} \colon 
 \overline{M}_{g, n, \Delta}^{\tilde{h},P^0} \rightarrow
\bigtimes_{V \in V(\tilde{\Gamma})}\overline{M}_{g(V), \Delta_V} \,. \]
Let us give a precise definition of the cut morphism\footnote{We are considering a stable log map over a point. It is a notational exercise to extend the argument to a stable log map over a general base, which is required to really define a morphism between moduli
spaces}.
By definition of the marking,
for every vertex $V$ of $\tilde{\Gamma}$, we have 
an ordinary stable map 
$f_V \colon C_V \rightarrow 
X_{\Delta_V}$, such that the underlying
stable map to $f$ is obtained by gluing 
together the maps $f_V$
along nodes corresponding 
to the 
edges of $\tilde{\Gamma}$.

We have to give $C_V$ the structure of a log curve, and enhance $f_V$ to a log morphism.
In particular, we need to construct
a monoid 
$\overline{\cM}_V$.

We fix a point $b$
in the interior of 
$\Sigma(g)^{-1}(1)$.
Let 
$\Sigma(f)_b \colon 
\Sigma(C)_b \rightarrow 
\R^2$ be the corresponding 
parametrized tropical curve.
Let $\Sigma(C)_{V,b}$ 
be the subgraph of 
$\Sigma(C)_b$ 
obtained by taking 
the vertices of $\Sigma(C)_b$ 
dual to irreducible components of $C_V$,
the edges between them, and considering the edges to other vertices of 
$\Sigma(C)_b$ as unbounded edges.
Let $\Sigma(f)_{V,b}$
be the restriction of 
$\Sigma(f)_b$
to 
$\Sigma(C)_{V,b}$.
It follows from Lemma
\ref{lem_edges}
that one can view $\Sigma(f)_{V,b}$
as a parametrized tropical curve 
of type $\Delta_V$ to the fan of 
$X_{\Delta_V}$.

We define 
$\overline{\cM}_V$ as being the monoid 
whose dual is the
monoid of integral 
points of the moduli space of deformations 
of $\Sigma(f)_{V,b}$ preserving its 
combinatorial type\footnote{The base monoid of a basic stable log map has always such description in terms of deformations of tropical curves. See Remark 1.18 and Remark 1.21 of \cite{MR3011419} for more details}.
Let $i_{C_V} \colon C_V 
\rightarrow C$ and 
$i_{X_{\Delta_V}} \colon 
X_{\Delta_V}
\rightarrow X_0$ be the inclusion morphisms 
of ordinary (not log) schemes.
The parametrized tropical curves 
$\Sigma(f)_V$ encode 
a sheaf of monoids 
$\overline{\cM}_{C_V}$ and a map
$f_V^{-1} \overline{\cM}_{X_{\Delta_V}} \rightarrow 
\overline{\cM}_{C_V}$.
We define a log structure on $C_V$
by 
\[\cM_{C_V}
=\overline{\cM}_{C_V}
\times_{i_{C_V}^{-1} \overline{\cM}_C}
i_{C_V}^{-1} \cM_C \,.\]
The natural diagram 
\begin{center}
\begin{tikzcd}
f_V^{-1} \cM_{X_{\Delta_V}} \arrow{d}
& \cM_{C_V} \arrow{d} \\
f_V^{-1} i_{X_{\Delta_V}}^{-1}
\cM_{X_0} \arrow{r} & 
i_{C_V}^{-1} 
\cM_C \,
\end{tikzcd}
\end{center}
can be uniquely completed, by restriction,
with a map 
\[ f_V^{-1} \cM_{X_{\Delta_V}} \rightarrow
\cM_{C_V} \] 
compatible with 
 $f_V^{-1}\overline{\cM}_{X_{\Delta_V}}
\rightarrow \overline{\cM}_{C_V}$.
This defines a log enhancement of $f_V$ and finishes the construction of the cut
morphism.

\textbf{Remark:}
If one considers a general log smooth degeneration and if one applies the decomposition formula, 
it is in general impossible to write the contribution 
of a tropical curves in terms of log Gromov-Witten invariants attached to the vertices.
This is already clear at the tropical level. The theory of punctured invariants 
developed by Abramovich, Chen, Gross, Siebert in \cite{abramovich2017punctured}
is the correct extension of log Gromov-Witten theory which is needed in order to be able to write down a general gluing formula. 
In our present case, the Nishinou-Siebert toric degeneration is extremely special because 
it has been constructed knowing a priori 
the relevant tropical curves.
It follows from Lemma 
\ref{lem_edges} that we always cut edges 
contained in an edge of the polyhedral decomposition, and so we don't have to consider punctured invariants.

\subsection{Counting log structures}
\label{section_counting}
We say that a map to $X_0$ is torically transverse if its image does not contain any of the torus fixed points of the toric components $X_{\Delta_V}$.
In other words, its corestriction to each  toric surface $X_{\Delta_V}$ is torically transverse in the sense of
Section \ref{section_transverse}.

Let 
$ \overline{M}_{g, n, \Delta}^{\tilde{h},P^0, \circ}$
be the open locus of 
$\overline{M}_{g, n, \Delta}^{\tilde{h},P^0}$
formed by the torically transverse stable log
maps to $X_0$, and for every vertex $V$ of $\tilde{\Gamma}$, let 
$\overline{M}_{g(V), \Delta_V}^\circ$ be the 
open locus of
$\overline{M}_{g(V), \Delta_V}$ formed by the 
torically transverse stable log maps
to $X_{\Delta_V}$.
The morphism cut restricts to a morphism
\[ \mathrm{cut}^\circ \colon 
 \overline{M}_{g, n, \Delta}^{\tilde{h},P^0,\circ} \rightarrow
\bigtimes_{V \in V(\tilde{\Gamma})}\overline{M}_{g(V), \Delta_V}^\circ \,. \]

\begin{prop} \label{prop_etale}
The morphism 
\[ \mathrm{cut}^\circ \colon 
 \overline{M}_{g, n, \Delta}^{\tilde{h},P^0,\circ}\rightarrow
\bigtimes_{V \in V(\tilde{\Gamma})}\overline{M}_{g(V), \Delta_V}^\circ  \]
is \'etale of degree 
\[  \prod_{E \in E_f(\tilde{\Gamma})} w(E)
\,,\]
where the product is over the bounded edges 
of $\tilde{\Gamma}$.
\end{prop}

\begin{proof}
Let $(f_V \colon C_V 
\rightarrow X_{\Delta_V})_V 
\in \bigtimes_{V \in V(\tilde{\Gamma})}\overline{M}_{g(V), \Delta_V}^\circ$.
We have to glue the stable log maps 
$f_V$ together. 
Because we are assuming that the maps $f_V$ are torically transverse, 
the image in $X_0$ by $f_V$ of the curves $C_V$ is away from the torus fixed points of the components $X_{\Delta_V}$. The gluing operation corresponding to the bounded edge $E$ of 
$\tilde{\Gamma}$ happens entirely along the torus $\C^{*}$ contained in the divisor $D_E$.

It follows that it is enough to study the 
following local model. Denote 
$\ell \coloneqq \ell(E) w(E)$, where 
$\ell(E)$ is the length of $E$ and $w(E)$ the weight of $E$.
Let $X_E$ be the toric variety $\Spec \C[x,y,u^{\pm}, t]/(xy=t^\ell)$, equipped with a morphism $\nu_E \colon
X_E \rightarrow \C$ given by the coordinate $t$. Using the natural toric divisorial log structures on $X_E$ and 
$\C$, we define by restriction a log structure on the special fiber $X_{0,E} \coloneqq \nu_E^{-1}(0)$ and a log smooth
morphism to the standard log point
$\nu_{0,E} \colon X_{0,E} \rightarrow 
\mathrm{pt}_{\NN}$. 
The scheme underlying
$X_{0,E}$ has two irreducible components, 
$X_{1,E} \coloneqq \C_x \times \C^{*}_u$ and 
$X_{2,E} \coloneqq \C_y \times \C^{*}_u$, glued along the smooth divisor $D_E^\circ \coloneqq \C^{*}_u$. We endow $X_{1,E}$ and $X_{2,E}$ with their toric divisorial log structures.

Let $f_1 \colon C_1/\mathrm{pt}_{\overline{\cM}_1} \rightarrow X_{1,E}$ be  the restriction to $X_{1,E}$ of a torically transverse stable log map to some toric compactification of $X_{1,E}$, with one point $p_1$ of tangency order 
$w(E)$ along $D_E$, and let $f_2 \colon 
C_2/\mathrm{pt}_{\overline{\cM}_2} \rightarrow X_{2,E}$
be the restriction to $X_{2,E}$ of a torically transverse stable log map to some toric compactification of $X_{2,E}$, with one point $p_2$ of tangency order $w(E)$ along 
$D_E$. We assume that $f(p_1)=f(p_2)$ and so we can glue the underlying maps 
$\underline{f}_1 \colon \underline{C}_1
\rightarrow \underline{X}_{1,E}$ and 
$\underline{f}_2 \colon 
\underline{C}_2 \rightarrow \underline{X}_{2,E}$ to obtain a map 
$\underline{f} \colon \underline{C}
\rightarrow \underline{X}_{0,E}$ where 
$\underline{C}$ is the curve obtained 
from $\underline{C}_1$ and 
$\underline{C}_2$ by identification of 
$p_1$ and $p_2$. We denote $p$ the corresponding node of $\underline{C}$.
We have to show that there are $w(E)$ ways to lift this map to a log map in a way compatible with the log maps $f_1$ and $f_2$ and with the basic condition. If $C_1$ and $C_2$ had no component contracted to 
$f(p) \in D_E^\circ$, this would follow 
from Proposition 7.1 of Nishinou, Siebert
\cite{MR2259922}. But we allow contracted components, so we have to present a variant of the proof of Proposition 7.1 of 
\cite{MR2259922}.

We first give a tropical description of the relevant objects.
The tropicalization of $X_{0,E}$
is the cone $\Sigma(X_{0,E})
=\Hom (\overline{\cM}_{X_{0,E},f(p)}, \R_{\geqslant 0})$. 
It is the fan of $X_{E}$, a two-dimensional cone generated by rays 
$\rho_1$ and $\rho_2$ dual to the divisors $X_{1,E}$ and $X_{2,E}$.
The toric description $X_E=\Spec \C[x,y,u^{\pm},t]/(xy=t^\ell)$ defines a natural chart for the log structure of $X_{0,E}$. Denote
$s_x, s_y, s_t$ the corresponding elements of $\cM_{X_{0,E}, f(p)}$ and 
$\overline{s}_x, \overline{s}_y, \overline{s}_t$ their projections in 
$\overline{\cM}_{X_{0,E}, f(p)}$.
We have $s_x  s_y= s_t^\ell$.
Seeing elements of $\overline{\cM}_{X_{0,E}, f(p)}$
as functions on $\Sigma (X_{0,E})$, we have 
$\rho_1=\overline{s}_y^{-1}(0)$, 
$\rho_2=\overline{s}_x^{-1}(0)$ and 
$\overline{s}_t \colon 
\Sigma (X_{0,E}) \rightarrow \R_{\geqslant 0}$ is the tropicalization of the 
projection 
$X_{0,E} \rightarrow \mathrm{pt}_{\NN}$.
Level sets $\overline{s}_t^{-1}(c)$ are 
line segments $[P_1,P_2]$ in $\Sigma (X_{0,E})$, connecting a point
$P_1$ of $\rho_1$ to a point $P_2$ of $\rho_2$, of length $\ell c$.

Denote $\underline{C}_{1,E}$ and 
$\underline{C}_{2,E}$ the irreducible components of $\underline{C}_1$ and 
$\underline{C}_{2}$ containing $p_1$ and 
$p_2$ respectively. We can see them as the two irreducible components of $\underline{C}$ meeting at the node $p$.
Fix $j=1$ or $j=2$.
The tropicalization of $C_j/\mathrm{pt}_{\overline{\cM}_j}$ is a family $\Sigma(C_j)$ of tropical curves
$\Sigma(C_j)_b$ parametrized by 
$b \in \Sigma(\mathrm{pt}_{\overline{\cM}_j})
=\Hom(\overline{\cM}_j,\R_{\geqslant 0})$.
Let $V_{j,E}$ be the vertex of these tropical curves dual to the irreducible component
$\underline{C}_{j,E}$. The image
$\Sigma(f_j)(V_{j,E})$ of $V_{j,E}$ by the tropicalization $\Sigma(f_j)$ of 
$f_j$ is a point in the tropicalization 
$\Sigma(X_{j,E})=\R_{\geqslant 0}$. This induces a map
$\Hom (\overline{\cM}_j,\R_{\geqslant 0})
\rightarrow \R_{\geqslant 0}$ 
defined by an element $v_j \in \overline{\cM}_j$.
The component $\underline{C}_{j,E}$ is contracted by $f_j$ onto $f_j(p_j)$ if and only if $v_j \neq 0$. In other words, $v_j$ is the measure according to the log structures of ``how" $\underline{C}_{j,E}$ is contracted by $f_j$.
The marked point $p_j$ on $C_{j,E}$ defines an unbounded edge $E_j$, of weight $w(E)$, whose image by $\Sigma(f_j)$ is the unbounded interval
$[\Sigma(f_j)(V_{j,E}),+\infty)
\subset \Sigma(X_{j,E})=\R_{\geqslant 0}$. 

We explain now the gluing at the tropical level. Let $j=1$ or $j=2$. Let $[0,\ell_j] \subset \Sigma (X_{j,E})=\R_{\geqslant 0}$
be an interval. If $c$ is a large enough positive real number, we denote $\varphi_c^j \colon [0,\ell_j]
\hookrightarrow \overline{s}_t^{-1}(c)=[P_1,P_2]$ the linear inclusion such that $\varphi_c^j(0)=P_j$ and $\varphi_c^j
([0,\ell_j])$ is a subinterval of $[P_1,P_2]$
of length $\ell_j$.  
Let $b_j \in \Sigma (\mathrm{pt}_{\overline{\cM}_j})$.
There exists $\ell_j$ large enough such that all images by 
$\Sigma(f_j)$ of vertices
of $\Sigma(f_j)_{b_j}$ are contained in $[0,\ell_j]  \subset \Sigma (X_{j,E})=\R_{\geqslant 0}$. 

For $c$ large enough, the line segments $\varphi_c^1([0,\ell_1])$
and $\varphi_c^2([0,\ell_2])$ are disjoint. 
We have 
\[
[P_1,P_2] \]
\[=[P_1,\varphi_c^1(\Sigma(f_1)(V_1)))
\cup 
[\varphi_c^1(\Sigma(f_1)(V_1)),
\varphi_c^2(\Sigma(f_2)(V_2))] \cup
(\varphi_c^2(\Sigma(f_2)(V_2)),P_2]
\,.\]

We construct a new tropical curve
$\Sigma_{b_1,b_2,c}$ by removing the unbounded edges $E_1$ and $E_2$ of $\Sigma(f_1)_{b_1}$ and $\Sigma(f_2)_{b_2}$, and gluing the remaining curves by an edge $F$ connecting $V_{1,E}$ and 
$V_{2,E}$, of weight $w(E)$, and length
$\frac{1}{w(E)}$ times the length of the line segment $[\varphi^1_c(\Sigma(f_1)(V_1))
,\varphi^2_c(\Sigma(f_2)(V_2))]$. 
We construct a tropical map
$\Sigma_{b_1,b_2,c} \rightarrow \Sigma(X_{0,E})$ using $\Sigma(f_1)_{b_1}$,
$\Sigma(f_2)_{b_2}$ and mapping the edge $F$ to 
$[\varphi^1_c(\Sigma(f_1)(V_1)),
\varphi^2_c(\Sigma(f_2)(V_2))]$.
We define $\overline{\cM}$ as being the monoid whose dual is the monoid of integral points of the moduli space of deformations of these tropical maps.

We have $\overline{\cM}
=\overline{\cM}_1 \oplus \overline{\cM}_2
\oplus \NN$. The element $(0,0,1) \in \overline{\cM}$ defines the function on 
the moduli space of tropical curves
$\Sigma (\mathrm{pt}_{\overline{\cM}})
=\Hom (\overline{\cM}, \R_{\geqslant 0})$
given by the length of the gluing edge $F$.
The function given by $\frac{1}{\ell}$ times the length of the line segment 
$[P_1,P_2]$ defines an element 
$\overline{s}_t^{\overline{\cM}} \in \overline{\cM}$. The morphism of monoids
$\NN \rightarrow \overline{\cM}$,
$1 \mapsto \overline{s}_t^{\overline{\cM}}$,
induces a map 
$g \colon \mathrm{pt}_{\overline{\cM}}
\rightarrow \mathrm{pt}_{\NN}$.
The decomposition of $
[P_1,P_2]$ into the three
intervals
$[P_1,\varphi_c^1(\Sigma(f_1)(V_1)))$, 
$ [\varphi_c^1(\Sigma(f_1)(V_1)),
\varphi_c^2(\Sigma(f_2)(V_2))]$ and
$(\varphi_c^2(\Sigma(f_2)(V_2)),P_2]$,
implies the relation
\[ \ell \, \overline{s}_t^{\overline{\cM}}
= (v_1,0,0) + (0,0,w(E))+(0,v_2,0) \]
in $\overline{\cM}
=\overline{\cM}_1 \oplus \overline{\cM}_2
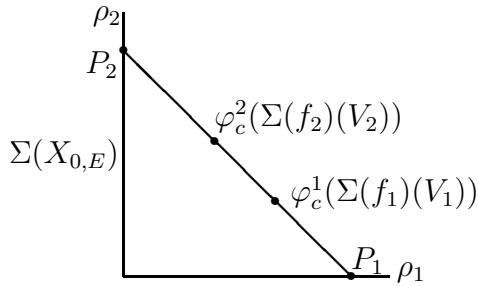
\oplus \NN$.

\begin{center}
\setlength{\unitlength}{1cm}
\begin{picture}(5,4)
\thicklines
\put(0,0){\line(1,0){3.5}}
\put(0,0){\line(0,1){3.5}}
\put(3,0){\circle*{0.1}}
\put(0,3){\circle*{0.1}}
\put(3,0){\line(-1,1){3}}
\put(1.2,1.8){\circle*{0.1}}
\put(2,1){\circle*{0.1}}
\put(3.6,0){$\rho_1$}
\put(-0.4,3.4){$\rho_2$}
\put(-0.5,2.7){$P_2$}
\put(3, 0.1){$P_1$}
\put(2.2, 1){$\varphi^1_c(\Sigma(f_1)(V_1))$}
\put(1.2, 2){$\varphi^2_c(\Sigma(f_2)(V_2))$}
\put(-1.5,1.5){$\Sigma(X_{0,E})$}
\end{picture}
\end{center}

From the tropical description of the gluing and from the fact that we want to obtain a 
basic log map, we find that there is a unique structure of log smooth curve 
$C/\mathrm{pt}_{\overline{\cM}}$ compatible with the structures of log smooth curves on $C_1$ and $C_2$. 
As $p$ is a node of $C$, we have for the ghost sheaf of $C$ at $p$: 
$\overline{\cM}_{C,p}=\overline{\cM} \oplus_{\NN} \NN^2$, with $\NN \rightarrow \NN^2$, $1 \mapsto (1,1)$, and 
$\NN \rightarrow \overline{\cM}=\overline{\cM}_1 \oplus \overline{\cM}_2
\oplus \NN $,
$1 \mapsto \rho_p=(0,0,1)$.

It remains to lift 
$\underline{f} \colon 
\underline{C} \rightarrow \underline{X}_{0,E}$ to a log map 
$f \colon C \rightarrow X_{0,E}$ such that the diagram 
\begin{center}
\begin{tikzcd}
C \arrow{r}{f} \arrow{d}{\pi}
& X_{0,E} \arrow{d}{\nu_{0,E}}\\
\mathrm{pt}_{\overline{\cM}} \arrow{r}{g} & \mathrm{pt}_{\NN} 
\end{tikzcd}
\end{center}
commutes. The restriction of $f$ to $C_j/\mathrm{pt}_{\overline{\cM}_j}$ 
has to coincide with $f_j$, for $j=1$
and $j=2$. 
It follows from the explicit description of 
$\overline{\cM}$ and $\overline{\cM}_C$ that such $f$ exists and is unique away from the node $p$.

It follows from the tropical description of the gluing that at the ghost sheaves level, 
$f$ at $p$ is given by 
\[\overline{f}^{\flat} \colon \overline{\cM}_{X_{0,E},f(p)}
 \rightarrow \overline{\cM}_{C,p}
=\overline{\cM} \oplus_{\NN} \NN^2
=(\overline{\cM}_1 \oplus \overline{\cM}_2
\oplus \NN) \oplus_{\NN} \NN^2  \]
\begin{align*}
\overline{s}_x &\mapsto ((v_1,0,0), (w(E),0)) \\
\overline{s}_y & \mapsto ((0,v_2,0), (0,w(E))) \\
\overline{s}_t  & \mapsto \overline{\pi}^{\flat}(\overline{s}_t^{\overline{\cM}})
=(\overline{s}_t^{\overline{\cM}},(0,0))\,.
\end{align*}
The relation 
$\ell \, \overline{s}_t^{\overline{\cM}}
= (v_1,v_2,w(E))$
in $\overline{\cM}
=\overline{\cM}_1 \oplus \overline{\cM}_2
\oplus \NN$
implies that 
\begin{align*}
\overline{f}^{\flat}(\overline{s}_x)
+\overline{f}^{\flat}(\overline{s}_y)
&=((v_1,v_2,0),(w(E),w(E)))
=((v_1,v_2,w(E)),(0,0)) \\
&=
\overline{f}^{\flat} (\ell \overline{s}_t^{\overline{\cM}}) \,,
\end{align*}
and so that this map is indeed well-defined.

The log maps $f_1 \colon C_1/\mathrm{pt}_{\overline{\cM}_1} \rightarrow X_{1,E}$ and $f_2 \colon C_2/\mathrm{pt}_{\overline{\cM}_2} \rightarrow X_{2,E}$ define morphisms \[f_1^{\flat} \colon \cM_{X_{1,E},f(p_1)} \rightarrow \cM_{C_1,p_1} \,,\] and \[f_2^{\flat} \colon \cM_{X_{2,E},f(p_2)} \rightarrow \cM_{C_2,p_2} \,.\]
For $j=1$ or $j=2$, let $\overline{\cM}_j \oplus 
\NN \rightarrow \cO_{C_j,p_j}$ be a chart of the log structure of $C_j$ at $p_j$. This realizes $\cM_{C_j,p_j}$ as a quotient of $(\overline{\cM}_j \oplus
\NN) \oplus \cO_{C,p}^{*}$. Denote $s_{j,m} \in \cM_{C_j,p_j}$ the image of $(m,1)$ for 
$m \in \overline{\cM}_j \oplus
\NN$. 

We fix a coordinate $u$ on 
$C_1$ near $p_1$ such that 
\[f_1^{\flat}(s_x)=s_{1,(v_1,0)} u^{w(E)}\]
and a coordinate $v$ on $C_2$ near $p_2$
such that 
\[f_2^{\flat}(s_y)=s_{2,(v_2,0)} v^{w(E)}\,.\]

We are trying to define some 
$f^{\flat} \colon \cM_{X_{0,E},f(p)}
\rightarrow \cM_{C,p}$, lift of $\overline{f}^{\flat}$, compatible with
$f_1^{\flat}$ and $f_2^{\flat}$.
For every $\zeta$ a $w(E)$-th root of 
unity, the map
\[\overline{\cM} \oplus_{\NN}
\NN^2 \rightarrow \cO_{C,p}\]
\[
(m,(a,b)) \mapsto 
\begin{dcases}
\zeta^a u^a v^b & \text{if  } m=0 \\
0 & \text{if  } m \neq 0
\end{dcases}\]
defines a chart for the log structure of $C$ at $p$. 
This realizes $\cM_{C,p}$
as a quotient of 
$(\overline{\cM} \oplus_{\NN} \NN^2) \oplus \cO_{C,p}^{*}$. Denote $s_m^{\zeta} \in \cM_{C,p}$ the image of
$(m,1)$ for $m \in \overline{\cM} \oplus_{\NN} \NN^2$.
Remark that 
$s^\zeta_{((v_1,0,0),(0,0))}$,
$s^\zeta_{((0,v_2,0),(0,0))}$
and $s^\zeta_{((0,0,0),(1,1))}$
are independent of $\zeta$ and we denote them simply as 
$s_{((v_1,0,0),(0,0))}$, 
$s_{((0,v_2,0),(0,0))}$
and $s_{((0,0,0),(1,1))}$.

Then 
\[f^{\flat, \zeta} \colon 
\cM_{X_{0,E},f(p)} \rightarrow \cM_{C,p}\]
\begin{align*}
s_x &\mapsto s^\zeta_{((v_1,0,0),(w(E),0))} \\
s_y &\mapsto s^\zeta_{((0,v_2,0),(0,w(E)))}\\
s_t &\mapsto \pi^{\flat}((\overline{s}_t^{\overline{\cM}},1))
\end{align*}
is a lift of $\overline{f}^{\flat}$, compatible with
$f_1^{\flat}$ and $f_2^{\flat}$.

Assume that $f^{\flat, \zeta}
\simeq f^{\flat,\zeta'}$ for 
$\zeta$ and $\zeta'$ two $w(E)$-th roots of unity. It follows from the compatibility 
with $f_1^{\flat}$ and $f_2^{\flat}$
that there exists 
$\varphi_1 \in \cO_{C,p}^{*}$
and $\varphi_2 \in \cO_{C,p}^{*}$
such that $s^{\zeta'}_{((0,0,0),(1,0))}
=\varphi_1 s^{\zeta}_{((0,0,0),(0,1))}$
and 
$s^{\zeta'}_{((0,0,0),(0,1))}
=\varphi_2 s^{\zeta}_{((0,0,0),(0,1))}$.
It follows from the definition of the charts that 
$\varphi_1=\zeta' \zeta^{-1}$ in 
$\cO_{C_1,p_1}$ and $\varphi_2=1$ in 
$\cO_{C_2,p_2}$.
Compatibility with $\mathrm{pt}_{\overline{\cM}} \rightarrow \mathrm{pt}_{\NN}$ implies that $\varphi_1 \varphi_2=1$. This implies that 
$\varphi_1=\varphi_2=1$ and $\zeta=\zeta'$.

It remains to show that any $f^{\flat}$, 
lift of $\overline{f}^{\flat}$ compatible with $f_1^{\flat}$ and $f_2^{\flat}$, is of the form $f^{\flat, \zeta}$ for some 
$\zeta$ a $w(E)$-th root of unity.
For such $f^{\flat}$, there exists unique 
$s'_{(1,0)} \in \cM_{C,p}$
and $s'_{(0,1)} \in \cM_{C,p}$
such that $\alpha_C(s'_{(1,0)})=u$, 
$\alpha_C(s'_{(0,1)})=v$, and 
$f^{\flat}(s_x)=s_{((v_1,0,0),(0,0))} (s'_{(1,0)})^{w(E)}$ and 
$f^{\flat}(s_y)=s_{((0,v_2,0),(0,0))} (s'_{(0,1)})^{w(E)}$. From 
$s_x s_y =s_t^\ell$, we get 
$(s'_{(1,0)}s'_{(0,1)})^{w(E)}
=s_{((0,0,0),(1,1))}^{w(E)}$ and so
$s'_{(1,0)}s'_{(0,1)}=\zeta^{-1} 
s_{((0,0,0),(1,1))}$ for some $\zeta$ a 
$w(E)$-th root of unity.
It is now easy to check that 
$s'_{(1,0)}
=\zeta^{-1} s^{\zeta}_{((0,0,0),(1,0))}$,
$s'_{(0,1)}
= s^{\zeta}_{((0,0,0),(0,1))}$
and $f^\flat = f^{\flat, \zeta}$.

\end{proof}

\textbf{Remarks:}
\begin{itemize}
\item When $v_1=v_2=0$, i.e. when the components $C_{1,E}$ and $C_{2,E}$ are not contracted, the above proof reduces to the proof of Proposition 7.1 of \cite{MR2259922} (see also the proof of Proposition 4.23 of
\cite{MR2722115}).
In general, log geometry remembers enough information about the contracted components, such as $v_1$ and $v_2$, to make possible a parallel argument.
\item The gluing of stable log maps along a smooth divisor is discussed in 
\mbox{Section 6} of 
\cite{kim2018degeneration}, proving the degeneration formula along a smooth divisor. In the above proof, we only have to glue along one edge connecting two vertices. In Section 6 of 
\cite{kim2018degeneration}, further work is required to deal with pair of vertices connected by several edges.
\end{itemize}

\subsection{Comparing obstruction theories}
\label{section_comparing}

As in the previous Section
\ref{section_counting},
let 
$\overline{M}_{g, n, \Delta}^{\tilde{h}, P^0,\circ}$
be the open locus of 
$\overline{M}_{g, n, \Delta}^{\tilde{h},P^0}$
formed by the torically transverse stable log
maps to $X_0$, and for every vertex $V$ of $\tilde{\Gamma}$, let 
$\overline{M}_{g(V), \Delta_V}^\circ$ be the 
open locus of
$\overline{M}_{g(V), \Delta_V}$ formed by the 
torically transverse stable log maps
to $X_{\Delta_V}$.
The morphism cut restricts to a morphism
\[ \mathrm{cut}^\circ \colon 
 \overline{M}_{g, n, \Delta}^{\tilde{h},P^0,\circ} \rightarrow
\bigtimes_{V \in V(\tilde{\Gamma})}\overline{M}_{g(V), \Delta_V}^\circ \,. \]
The goal of the present Section is to use the morphism
$\mathrm{cut}^\circ$ to compare the virtual classes $[\overline{M}_{g, n, \Delta}^{\tilde{h},P^0,\circ}]^{\virt}$ and 
$[\overline{M}^\circ_{g(V), \Delta_V}]^{\virt}$,
which are obtained by restricting the virtual classes
$[\overline{M}_{g, n, \Delta}^{\tilde{h},P^0}]^{\virt}$ and 
$[\overline{M}_{g(V), \Delta_V}]^{\virt}$  to the open loci of torically transverse stable log maps.

Recall that $X_0=\nu^{-1}(0)$, where 
$\nu \colon X_{\cP_{\Delta,n}}
\rightarrow \A^1$. Following Section 4.1 
of \cite{abramovich2017decomposition}, we define \mbox{$\cX_0 \coloneqq \cA_X \times_{\cA_{\A^1}} \{0\}$}, where $\cA_X$
and $\cA_{\A^1}$ are Artin fans, see Section 2.2 of \cite{abramovich2017decomposition}.
It is an algebraic log stack over 
$\mathrm{pt}_{\NN}$. There is a natural morphism $X_0 \rightarrow \cX_0$.

Following Section 4.5
of \cite{abramovich2017decomposition}, let
$\mathfrak{M}_{g,n,\Delta}^{\tilde{h}}$ be the stack of $n$-pointed genus $g$ prestable basic log maps to $\cX_0/\mathrm{pt}_{\NN}$ marked by
$\tilde{h}$ and of type $\Delta$. There is a natural morphism of stacks
$\overline{M}^{\tilde{h},P^0}_{g,n,\Delta}
\rightarrow \mathfrak{M}_{g,n,\Delta}^{\tilde{h}}\,.$
Let 
$\pi \colon \cC \rightarrow 
\overline{M}_{g,n,\Delta}^{\tilde{h},P^0}$ 
be the universal curve 
and let $f \colon \cC \rightarrow
X_0/ \mathrm{pt}_{\NN}$
be the universal stable log map.
According to 
Proposition 4.7.1 and Section 6.3.2 of
\cite{abramovich2017decomposition},
the virtual fundamental class 
$[\overline{M}_{g,n,\Delta}^{\tilde{h},P^0}]^{\virt}$
is defined by $\mathbf{E}$, 
the cone of the morphism
$
(\ev^{(p)})^{*}
L_{\iota_{P^0}}[-1]
\rightarrow 
(R\pi_{*} f^{*} T_{X_0|
\cX_0})^\vee$, 
seen as a 
perfect obstruction theory relative to 
$\mathfrak{M}_{g,n,\Delta}^{\tilde{h}}$. 
Here,
$T_{X_0|\cX_0}$ is the relative 
log tangent bundle, and 
$L_{\iota_{P^0}}=\oplus_{V \in V^{(2p)}(\tilde{\Gamma})} (T_{X_{\Delta_V}}|_{P^0_V})^\vee [1]$ is the cotangent complex of $\iota_{P^0}$.
As $\cX_0$ is log \'etale over 
$\mathrm{pt}_\NN$, we have $T_{X_0|\cX_0}
=T_{X_0|\mathrm{pt}_{\NN}}$.
We denote $\mathbf{E}^\circ$ the restriction of $\mathbf{E}$ to the open locus $\overline{M}_{g,n,\Delta}^{\tilde{h},P^0, \circ}$
of torically transverse stable log maps.

For every vertex $V$ of $\tilde{\Gamma}$, let 
$\pi_V 
\colon \cC_V \rightarrow \overline{M}_{g(V),
\Delta_V}$
be the universal curve and let 
$f_V \colon \cC_V \rightarrow X_{\Delta_V}$
be the universal stable log map.
Let $\cA_{X_{\Delta_V}}$ be the Artin fan of $X_{\Delta_V}$ and
let $\mathfrak{M}_{g(V),\Delta_V}$
be the stack of prestable basic log maps to 
$\cA_{X_{\Delta_V}}$, of genus $g(V)$ and of type $\Delta_V$. 
There is a natural morphism of stacks
$\overline{M}_{g(V),\Delta_V}
\rightarrow \mathfrak{M}_{g(V),\Delta_V}$.
According to Section 6.1 of
\cite{abramovich2013invariance},
the virtual fundamental class 
$[\overline{M}_{g(V), \Delta_V}]^{\virt}$
is defined by 
$(R (\pi_V)_{*} f_V^{*} T_{X_{\Delta_V}})^\vee$, seen as a perfect obstruction
theory relative to
$\mathfrak{M}_{g(V),\Delta_V}$. Here, $T_{X_{\Delta_V}}$
is the log tangent bundle.

Recall that $\underset{V \in V(\tilde{\Gamma})}
{\bigtimes}
\overline{M}_{g(V), \Delta_V}$
is defined by the fiber product diagram
\begin{center}
\begin{tikzcd}
\underset{V \in V(\tilde{\Gamma})}
{\bigtimes}
\overline{M}_{g(V), \Delta_V} \arrow{r}{(\delta \times
\iota_{P^0})_M} \arrow{d}{\text{ev}^{(e)} \times \text{ev}^{(p)}}
& \underset{V \in V(\tilde{\Gamma})}{\prod}
\overline{M}_{g(V), \Delta_V} 
\arrow{d}{\text{ev}^{(e)} \times \text{ev}^{(p)}} \\
\left(\underset{E \in E_f(\tilde{\Gamma})}{\prod}
D_E \right) \times P^0 
\arrow{r}{\delta
\times \iota_{P^0}} &
\underset{E \in E_f(\tilde{\Gamma})}
{\prod}
(D_E)^2
\times \underset{V \in V^{(2p)}(\tilde{\Gamma})}{\prod} 
X_{\Delta_V} 
\,.
\end{tikzcd}
\end{center}

We compare the deformation theory of 
the individual stable log maps 
$f_V$ and the deformation theory of the stable log maps $f_V$ constrained to match at the gluing nodes.
Let $\mathbf{F}$
be the cone of the natural morphism 
\[ (\text{ev}^{(e)}
\times \text{ev}^{(p)})^{*} L_{\delta
\times \iota_{P^0}}
[-1] \rightarrow 
(\delta \times \iota_{P^0})_M^{*} 
\left(
\underset{V \in V(\tilde{\Gamma})}{\mathlarger{\mathlarger{\mathlarger{\boxtimes}}}}
(R(\pi_V)_{*} f_V^{*} T_{X_{\Delta_V}})^\vee \right) \,,\]
where $L_{\delta \times \iota_{P^0}}$ is the cotangent 
complex of the morphism $\delta
\times \iota_{P^0}$.
It defines a
perfect obstruction theory on 
$\underset{V \in V(\tilde{\Gamma})}{\bigtimes}
\overline{M}_{g(V), \Delta_V}$
relative to
$\underset{V \in V(\tilde{\Gamma})}
{\prod}
\mathfrak{M}_{g(V), \Delta_V}$,
whose corresponding 
virtual fundamental class is,
using Proposition 5.10 of 
\cite{MR1437495}, 
\[(\delta \times \iota_{P^0})^! \prod_{V \in V(\tilde{\Gamma})}
[\overline{M}_{g(V),\Delta_V}]^{\virt}\,,\]
where 
$(\delta \times \iota_{P^0})^!$ is the refined 
Gysin homomorphism 
(see Section 6.2 of \cite{MR1644323}).
We denote $\mathbf{F}^\circ$ 
the restriction of $\mathbf{F}$ to the 
open locus
$\bigtimes_{V \in V(\tilde{\Gamma})}\overline{M}_{g(V), \Delta_V}^\circ$
of torically transverse stable log maps.

The cut operation naturally extends to prestable log maps to 
$\cX_0/\mathrm{pt}_{\NN}$
marked by $\tilde{h}$, and so we have a 
commutative diagram

\begin{center}
\begin{tikzcd}
\overline{M}_{g,n,\Delta}^{\tilde{h},P^0,\circ}  \arrow{r}{\mathrm{cut}^\circ} \arrow{d}{\mu}
& \underset{V \in V(\tilde{\Gamma})}{\bigtimes}
\overline{M}^\circ_{g(V), \Delta_V}
\arrow{d}\\
\mathfrak{M}_{g,n,\Delta}^{\tilde{h}}
\arrow{r}{\mathrm{cut}_C} &
\underset{V \in V(\tilde{\Gamma})}
{\prod}
\mathfrak{M}_{g(V), \Delta_V}\,.
\end{tikzcd}
\end{center}

By Proposition \ref{prop_etale}, the morphism $\mathrm{cut}^\circ$ is \'etale and so 
$(\mathrm{cut}^\circ)^{*} \mathbf{F}^\circ$ 
defines a perfect obstruction theory on 
$\overline{M}_{g,n,\Delta}^{\tilde{h},P^0,\circ}$ relative to 
$\underset{V \in V(\tilde{\Gamma})}
{\prod}
\mathfrak{M}_{g(V), \Delta_V}$.

The maps 
$\overline{M}_{g,n,\Delta}^{\tilde{h},P^0,\circ}
\xrightarrow{\mu}
\mathfrak{M}_{g,n,\Delta}^{\tilde{h}}
(\cX_0 /\mathrm{pt}_{\NN})
\xrightarrow{\mathrm{cut}_C}
 \underset{V \in V(\tilde{\Gamma})}
{\prod}
\mathfrak{M}_{g(V), \Delta_V}$
define an exact triangle of cotangent complexes
\begin{center}
\begin{tikzcd}[column sep=tiny]
L_{\overline{M}_{g,n,\Delta}^{\tilde{h},P^0,\circ} |\kern-1.5ex \underset{V \in V(\tilde{\Gamma})}
{\prod}
\mathfrak{M}_{g(V), \Delta_V}}
\arrow[r] &
L_{\overline{M}_{g,n,\Delta}^{\tilde{h},P^0,\circ} |\mathfrak{M}_{g,n,\Delta}^{\tilde{h}}}
\arrow[r] &
\mu^{*} L_{\mathfrak{M}_{g,n,\Delta}^{\tilde{h}}| \kern-1.5ex \underset{V \in V(\tilde{\Gamma})}
{\prod}
\mathfrak{M}_{g(V), \Delta_V}}[1]
\arrow{r}{[1]}& \,.
\end{tikzcd}
\end{center}
Adding the perfect obstruction theories 
$(\mathrm{cut}^\circ)^{*} \mathbf{F}^\circ$
and $\mathbf{E}^\circ$, we get a diagram
\begin{center}
\begin{tikzcd}[column sep=tiny]
(\mathrm{cut}^\circ)^{*} \mathbf{F}^\circ \arrow{d}
&\mathbf{E}^\circ \arrow{d} & &
\\
L_{\overline{M}_{g,n,\Delta}^{\tilde{h},P^0,\circ} | \kern-1.5ex \underset{V \in V(\tilde{\Gamma})}
{\prod}
\mathfrak{M}_{g(V), \Delta_V}}
\arrow{r} &
L_{\overline{M}_{g,n,\Delta}^{\tilde{h},P^0,\circ} |\mathfrak{M}_{g,n,\Delta}^{\tilde{h}}}
\arrow{r}&
\mu^{*} L_{\mathfrak{M}_{g,n,\Delta}^{\tilde{h}}|
\kern-1.5ex \underset{V \in V(\tilde{\Gamma})}
{\prod}
\mathfrak{M}_{g(V), \Delta_V}}[1]
\arrow{r}{[1]}&
\,.
\end{tikzcd}
\end{center}

\begin{prop}\label{prop_compatible}
The above diagram can be completed into a morphism of exact triangles
\begin{center}
\begin{tikzcd}[column sep=tiny]
(\mathrm{cut}^\circ)^{*} \mathbf{F}^\circ \arrow{d} \arrow{r}
&\mathbf{E}^\circ \arrow{d} \arrow{r}
&\mu^{*} L_{\mathfrak{M}_{g,n,\Delta}^{\tilde{h}}| \kern-1.5ex \underset{V \in V(\tilde{\Gamma})}
{\prod}
\mathfrak{M}_{g(V), \Delta_V}}[1]
\arrow[equal]{d} \arrow{r}{[1]}
 &\,
\\
L_{\overline{M}_{g,n,\Delta}^{\tilde{h},P^0,\circ} | \kern-1.5ex \underset{V \in V(\tilde{\Gamma})}
{\prod}
\mathfrak{M}_{g(V), \Delta_V}}
\arrow{r} &
L_{\overline{M}_{g,n,\Delta}^{\tilde{h},P^0,\circ} |\mathfrak{M}_{g,n,\Delta}^{\tilde{h}}}
\arrow{r}&
\mu^{*} L_{\mathfrak{M}_{g,n,\Delta}^{\tilde{h}}| \kern-1.5ex \underset{V \in V(\tilde{\Gamma})}
{\prod}
\mathfrak{M}_{g(V), \Delta_V}}[1]
\arrow{r}{[1]}&
\,.
\end{tikzcd}
\end{center}

\end{prop}

\begin{proof}
Denote $X_0^\circ$, $X_{\Delta_V}^\circ$, $D_E^\circ$ the objects obtained from 
$X_0$, $X_{\Delta_V}$, $D_E$ by removing the torus fixed points of the toric surfaces $X_{\Delta_V}$. Denote 
$\iota_{X_{\Delta_V}^\circ}$ the inclusion morphism of $X_{\Delta_V}^\circ$ in $X_0^\circ$.

If $E$ is a bounded edge of $\tilde{\Gamma}$, we denote $V_E^1$ and $V_E^2$ the two vertices of $E$. 
Let $\cF$ be the sheaf on the universal curve $\cC|_{\overline{M}_{g,n,\Delta}^{\tilde{h},P^0,\circ}}$ defined as the kernel of 
\[\bigoplus_{V \in V(\tilde{\Gamma})} f^{*}(\iota_{X_{\Delta_V}^\circ})_* T_{X_{\Delta_V}^\circ}
\rightarrow \bigoplus_{E \in E_f(\tilde{\Gamma})} (\iota_{E})_* (\ev^E)^* T_{D_E^\circ} \]
\[(s_V)_V \mapsto (s_{V_E^1}|_{D_E^\circ}
-s_{V_E^2}|_{D_E^\circ})_E \,,\]
where $\ev^E$ is the evaluation at the node
$p_E$ dual to $E$, and $\iota_E$ the section of $\cC$ given by
$p_E$. It follows from the exact triangle obtained by applying 
$R \pi_*$ to the short exact sequence defining 
$\cF$ and from 
$L_{\delta}=\oplus_{E \in E_f(\tilde{\Gamma})}T_{D_E}^\vee [1]$ that
 $(\mathrm{cut}^\circ)^* \mathbf{F}^\circ$ 
is given by the cone of the morphism
$(\ev^{(p)})^{*}L_{\iota_{P^0}}[-1]
\rightarrow (R\pi_{*}\cF)^\vee$.
So in order to compare $\mathbf{E}^\circ$ and 
$(\mathrm{cut}^\circ)^{*}\mathbf{F}^\circ$, we have to compare 
$f^{*} T_{X_0^\circ|\mathrm{pt}_\NN}$ and 
$\cF$. 
The sheaf $f^* T_{X_0^\circ|\mathrm{pt}_{\NN}}$ can be written as the kernel of 
\[ f^* \bigoplus_{V \in V(\tilde{\Gamma})}
(\iota_{X_{\Delta_V}^\circ})_* 
(\iota_{X_{\Delta_V}^\circ})^* T_{X_0^\circ|\mathrm{pt}_{\NN}}
\rightarrow 
 \bigoplus_{E \in E_f(\tilde{\Gamma})} 
(\iota_{E})_* 
(\ev^E)^*
T_{X_0^\circ|\mathrm{pt}_{\NN}}
 \,\]
\[(s_V)_V \mapsto (s_{V_E^1}|_{D_E^\circ}
-s_{V_E^2}|_{D_E^\circ})_E \,.\]
Remark that because $X_0$ is the special fiber of a toric degeneration, all the log tangent bundles $T_{X_0}$, $T_{X_{\Delta_V}}$, $T_{D_E}$ are free
sheaves
(see e.g.\ Section 7 of
\cite{MR2259922}). 
In particular, the restrictions 
$(\iota_{X_{\Delta_V}^\circ})^* T_{X_0^\circ|\mathrm{pt}_{\NN}} \rightarrow 
T_{X_{\Delta_V}^\circ}$ 
are isomorphisms, the restriction 
\[\bigoplus_{E \in E_f(\tilde{\Gamma})} 
(\ev^E)^*
T_{X_0^\circ|\mathrm{pt}_{\NN}}
\rightarrow 
 \bigoplus_{E \in E_f(\tilde{\Gamma})} 
(\ev^E)^{*} 
T_{D_E^\circ}
\] 
has
kernel $\bigoplus_{E \in E_f(\tilde{\Gamma})} (\ev^E)^{*} \cO_{D_E^\circ}$
and so there is an induced exact sequence
\[0 \rightarrow f^* T_{X_0^\circ|\mathrm{pt}_{\NN}}
\rightarrow  \cF
\rightarrow \bigoplus_{E \in E_f(\tilde{\Gamma})} 
(\iota_E)_*
(\ev^E)^* \cO_{D_E^\circ} 
\rightarrow 0\,, \]
which induces an exact triangle
on 
$\overline{M}_{g, n, \Delta}^{\tilde{h},P^0 \circ}$:
\[ (\mathrm{cut}^\circ)^* \mathbf{F}^\circ
\rightarrow \mathbf{E}^\circ \rightarrow 
\bigoplus_{E \in E_f(\tilde{\Gamma})} (\ev^E)^*
\cO_{D_E^\circ}[1] 
\xrightarrow{[1]} \,.\]

It remains to check the compatibility of this exact triangle with the exact triangle
of cotangent complexes.
We have
\[\mu^{*} L_{\mathfrak{M}_{g,n,\Delta}^{\tilde{h}}| \underset{V \in V(\tilde{\Gamma})}
{\prod}
\mathfrak{M}_{g(V), \Delta_V}}
=\bigoplus_{E\in E_f(\tilde{\Gamma})}
(\iota_E)^* \cO_{p_E}\,.\]
Indeed, restricted to the locus of torically transverse stable log maps,
$\mathrm{cut}_C$ is smooth, and, given a torically transverse 
stable log map to
$\cX_0/\mathrm{pt}_\NN$, a basis of first order infinitesimal deformations fixing its image
by $\mathrm{cut}_C$ in $\prod_{V \in V(\tilde{\Gamma})}
\mathfrak{M}_{g(V), \Delta_V}$  is indexed by the cutting nodes. The dual of the natural map 
\[\bigoplus_{E \in E_f(\tilde{\Gamma})} (\ev^E)^*
\cO_{D_E^\circ}
\rightarrow 
\mu^{*} L_{\mathfrak{M}_{g,n,\Delta}^{\tilde{h}}| \underset{V \in V(\tilde{\Gamma})}
{\prod}
\mathfrak{M}_{g(V), \Delta_V}}
=\bigoplus_{E\in E_f(\tilde{\Gamma})}
(\iota_E)^* \cO_{p_E}\] 
sends the canonical first order infinitesimal deformation indexed by the cutting node $p_E$
to
the canonical summand $\cO_{D_E^\circ}$ in the normal bundle to the diagonal
$\prod_{E \in E_f(\tilde{\Gamma})} D_E^\circ$ in $\prod_{E \in E_f(\tilde{\Gamma})} (D_E^\circ)^2$, and so is an isomorphism.
This guarantees the compatibility with the exact triangle of cotangent complexes.
\end{proof}

\textbf{Remark:} Restricted to the open locus of torically transverse stable maps,
the discussion is essentially reduced to a collection of gluings along the smooth divisors $D_E^\circ$. A comparison of the obstruction theories in the context of the degeneration formula along a smooth divisor is given with full details in \mbox{Section 7} of 
\cite{kim2018degeneration}.

\begin{prop} \label{prop_virtual_class}
We have
\[ (\mathrm{cut}^\circ)_*
\left(
[\overline{M}_{g, n, \Delta}^{\tilde{h},P^0,\circ}]^{\virt}
\right)\]
\[= \left(
\prod_{E \in E_f(\tilde{\Gamma})} w(E)
\right)
\left(
(\delta \times \iota_{P^0})_M^!
\prod_{V \in V(\tilde{\Gamma})}[\overline{M}_{g(V), \Delta_V}^\circ]^{\virt}
\right)
 \,.
\]
\end{prop}

\begin{proof}
It follows from Proposition
\ref{prop_compatible} and 
from Theorem 4.8 of \cite{MR2877433}
that the relative obstruction theories $\mathbf{E}^\circ$ and 
$(\mathrm{cut}^\circ)^* \mathbf{F}^\circ$ 
define the same virtual fundamental class 
on $\overline{M}_{g, n, \Delta}^{\tilde{h},P^0,\circ}$.
By Proposition \ref{prop_etale},
$\mathrm{cut}^\circ$ is \'etale, and so, 
by Proposition 7.2 of 
\cite{MR1437495}, the virtual fundamental class defined by $(\mathrm{cut}^\circ)^* \mathbf{F}^\circ$ is the image by 
$(\mathrm{cut}^\circ)^*$ of the virtual fundamental class defined by $\mathbf{F}^\circ$. It follows that
\[
[\overline{M}_{g, n, \Delta}^{\tilde{h},P^0,\circ}]^{\virt}=
(\text{cut}^\circ)^* (\delta \times \iota_{P^0})_M^!
\prod_{V \in V(\tilde{\Gamma})}[\overline{M}_{g(V), \Delta_V}^\circ]^{\virt}
 \,.
\]
According to 
Proposition \ref{prop_etale}, the morphism $\mathrm{cut}^\circ$ is \'etale of degree 
\[ \prod_{E \in E_f(\tilde{\Gamma})} w(E)\,,\]
and so the result follows from the projection formula.
\end{proof}

\subsection{Gluing}
\label{sub_gluing}
Recall that we have the morphism
\[(\delta \times \iota_{P^0})_M \colon \underset{V \in V(\tilde{\Gamma})}
{\bigtimes}
\overline{M}_{g(V), \Delta_V} 
\rightarrow  \underset{V \in V(\tilde{\Gamma})}{\prod}
\overline{M}_{g(V), \Delta_V} \,. \]
For every $V \in V(\tilde{\Gamma})$, we have a projection morphism
\[ \mathrm{pr}_V \colon \underset{V' \in V(\tilde{\Gamma})}{\prod}
\overline{M}_{g(V'), \Delta_{V'}}
\rightarrow \overline{M}_{g(V), \Delta_V}\,.\]
On each moduli space
$\overline{M}_{g(V), \Delta_V}$, we have the top lambda class 
$(-1)^{g(V)} \lambda_{g(V)}$.

\begin{prop}\label{prop_gluing_proof}
We have 
\[ N_{g,\tilde{h}}^{\Delta,n}=\int_{(\delta \times \iota_{P^0})^!
\underset{V \in V(\tilde{\Gamma})}{\prod} [\overline{M}_{g(V), \Delta_V}]^{\virt}} 
(\delta \times \iota_{P^0})_M^{*} \prod_{V \in V(\tilde{\Gamma})}
\mathrm{pr}_V^{*} 
\left( (-1)^{g(V)} \lambda_{g(V)} \right)\,.\]
\end{prop}

\begin{proof}
By definition
(see Section
\ref{subsection_decomposition_formula}),
we have 
\[N_{g,\tilde{h}}^{\Delta,n}
= 
\int_{[\overline{M}_{g,n,\Delta}^{\tilde{h},P^0}]^{\virt}}
(-1)^{g-g_{\Delta ,n}} 
\lambda_{g-g_{\Delta ,n}} \,.\]

Using the gluing properties of
lambda classes given by 
Lemma \ref{lem_gluing1}, we 
obtain that 
\[
(-1)^{g-g_{\Delta, n}}
\lambda_{g-g_{\Delta, n}}= (\text{cut})^*
(\delta \times \iota_{P^0})_M^{*} \prod_{V \in V(\tilde{\Gamma})}
\mathrm{pr}_V^{*} 
\left( (-1)^{g(V)} \lambda_{g(V)} \right)
 \,.\]

It follows from the projection formula that 
\[N_{g,\tilde{h}}^{\Delta,n}
= 
\int_{(\text{cut})_*[\overline{M}_{g,n,\Delta}^{\tilde{h},P^0}]^{\virt}}
(\delta \times \iota_{P^0})_M^{*} \prod_{V \in V(\tilde{\Gamma})}
\mathrm{pr}_V^{*} 
\left( (-1)^{g(V)} \lambda_{g(V)} \right)\,.\]

According to
Proposition \ref{prop_virtual_class},
the cycles
\[ (\text{cut})_*
\left(
[\overline{M}_{g, n, \Delta}^{\tilde{h},P^0}]^{\virt}
\right)\]
and 
\[ \left(
\prod_{E \in E_f(\tilde{\Gamma})} w(E)
\right)
\left(
(\delta \times \iota_{P^0})^!
\prod_{V \in V(\tilde{\Gamma})}[\overline{M}_{g(V), \Delta_V}]^{\virt}
\right) \]
have the same restriction to the open substack
\[ \bigtimes_{V \in V(\tilde{\Gamma})}\overline{M}_{g(V), \Delta_V}^\circ \] of 
\[ \bigtimes_{V \in V(\tilde{\Gamma})}\overline{M}_{g(V), \Delta_V}\,.\]
It follows, by Proposition 1.8 of \cite{MR1644323},
that their difference is rationally equivalent to a cycle supported on the closed substack  
\[ Z \coloneqq \left( \bigtimes_{V \in V(\tilde{\Gamma})}\overline{M}_{g(V), \Delta_V}
\right) - \left( \bigtimes_{V \in V(\tilde{\Gamma})}\overline{M}_{g(V), \Delta_V}^\circ
\right) \,.\]

If we have
\[(f_V \colon C_V \rightarrow X_{\Delta_V})_{V \in V(\tilde{\Gamma})}
\in Z\,,\] 
then at least one stable log 
map $f_V \colon C_V \rightarrow X_{\Delta_V}$
is not torically transverse. By 
Lemma \ref{lem_edges}, the unbounded edges of the tropicalization of $f_V$ are contained
in the rays of the fan of 
$X_{\Delta_V}$. It follows that we can apply Proposition \ref{cycle}
to obtain that at least one of the source curves $C_V$ contains a non-trivial cycle of
components.
By the vanishing result of Lemma
\ref{lem_gluing2}, this implies that  
\[ \int_Z 
(\delta \times \iota_{P^0})_M^{*} \prod_{V \in V(\tilde{\Gamma})}
\mathrm{pr}_V^{*} 
\left( (-1)^{g(V)} \lambda_{g(V)} \right)
=0 \,.\]
It follows that
\[\int_{(\text{cut})_*[\overline{M}_{g,n,\Delta}^{\tilde{h},P^0}]^{\virt}}
(\delta \times \iota_{P^0})_M^{*} \prod_{V \in V(\tilde{\Gamma})}
\mathrm{pr}_V^{*} 
\left( (-1)^{g(V)} \lambda_{g(V)} \right)\]
\[=\int_{(\delta \times \iota_{P^0})^!
\underset{V \in V(\tilde{\Gamma})}{\prod} [\overline{M}_{g(V), \Delta_V}]^{\virt}} 
(\delta \times \iota_{P^0})_M^{*} \prod_{V \in V(\tilde{\Gamma})}
\mathrm{pr}_V^{*} 
\left( (-1)^{g(V)} \lambda_{g(V)} \right)\,.\]
This finishes the proof of 
Proposition \ref{prop_gluing_proof}.

\end{proof}

\subsection{Identifying the pieces}
\label{section_pieces}

\begin{prop}\label{prop_pieces}
We have 
\[
\int_{(\delta \times \iota_{P^0})^!
\underset{V \in V(\tilde{\Gamma})}{\prod} [\overline{M}_{g(V), \Delta_V}]^{\virt}} 
(\delta \times \iota_{P^0})_M^{*} \prod_{V \in V(\tilde{\Gamma})}
\mathrm{pr}_V^{*} 
\left( (-1)^{g(V)} \lambda_{g(V)} \right)\]
\[=\prod_{V \in V(\tilde{\Gamma})}
N_{g(V),V}^{1,2}\,.\]
\end{prop}

\begin{proof}
Using the definitions of $\delta$ and
$\iota_{P^0}$, we have 
\[
\int_{(\delta \times \iota_{P^0})^!
\underset{V \in V(\tilde{\Gamma})}{\prod} [\overline{M}_{g(V), \Delta_V}]^{\virt}} 
(\delta \times \iota_{P^0})_M^{*} \prod_{V \in V(\tilde{\Gamma})}
\mathrm{pr}_V^{*} 
\left( (-1)^{g(V)} \lambda_{g(V)} \right) \]
\[=
\int_{
\underset{V \in V(\tilde{\Gamma})}{\prod} [\overline{M}_{g(V), \Delta_V}]^{\virt}}
(\ev^{(p)})^{*}([P^0])
(\ev^{(e)})^{*}([\delta])
\prod_{V \in V(\tilde{\Gamma})}
\mathrm{pr}_V^{*} 
\left( (-1)^{g(V)} \lambda_{g(V)} \right) 
\,,\]
where 
\[[P^0] = \prod_{V \in V^{(2p)}(\tilde{\Gamma})} P^0_V \in A^{*}\left(\prod_{V \in 
V^{(2p)}(\tilde{\Gamma})} X_{\Delta_V}
\right)\] is the class of $P^0$ and 
\[ [\delta] \in A^{*}\left(\prod_{E \in E_f(\tilde{\Gamma})} (D_E)^2\right)\] 
is the class of the diagonal $\prod_{E \in E_f(\tilde{\Gamma})} D_E$.
As each $D_E$ is a projective line, 
we have 
\[[\delta] = \prod_{E \in E_f(\tilde{\Gamma})}  
(\mathrm{pt}_E  \times 1
+1 \times \mathrm{pt}_E)\,,\]
where $\mathrm{pt}_E \in A^1(D_E)$ is the class of a point.

We fix an orientation of edges of 
$\tilde{\Gamma}$ as described in Section
\ref{section_statement_gluing}.
In particular, every trivalent vertex has two ingoing and one outgoing adjacent edges, every bivalent pointed vertex has two outgoing adjacent edges,
every bivalent unpointed vertex has one ingoing and one outgoing edges.
For every bounded edge $E$ of
$\tilde{\Gamma}$, we denote $V_E^s$ the source vertex of $E$ and $V_E^t$ the 
target vertex of $E$, as defined by the orientation.
Furthermore, 
the connected components of the complement of the bivalent pointed vertices of 
$\tilde{\Gamma}$ are trees with exactly one unbounded edge.

We argue that the effect of the insertion
$(\ev^{(p)})^{*}([P^0])
(\ev^{(e)})^{*}([\delta])$ can be computed in terms of the combinatorics of ingoing and outgoing edges of $\tilde{\Gamma}$\footnote{It is essentially a cohomological reformulation and generalization of the way the gluing is organized in Mikhalkin's proof of the tropical correspondence theorem, \cite{MR2137980}.}.
More precisely, we claim that the only term
in 
\[ 
(\ev^{(e)})^{*}([\delta]) 
=
\prod_{E \in E_f(\tilde{\Gamma})}
\left( 
(\ev_{V_E^s}^E)^{*}
(\mathrm{pt}_E)+(\ev_{V_E^t}^E)^{*}
(\mathrm{pt}_E)
\right)\,,\]
giving a non-zero contribution after 
multiplication by 
\[
\left( \prod_{V \in V^{(2p)}(\tilde{\Gamma})}
(\ev_V^{(p)})^{*}(P^0_V) \right)
\left( \prod_{V \in V(\tilde{\Gamma})}
\mathrm{pr}_{V}^{*} 
\left( (-1)^{g(V)} \lambda_{g(V)} \right)
\right)\] 
and integration over 
$\prod_{V \in V(\tilde{\Gamma})} [\overline{M}_{g(V), \Delta_V}]^{\virt}$
is
$\prod_{E \in E_f(\tilde{\Gamma})}
(\ev_{V_E^t}^E)^{*}
(\mathrm{pt}_E)$.

We prove this claim by induction, starting at the bivalent pointed vertices, where things are constrained by the marked points $P^0$, and propagating these constraints following the flow on $\tilde{\Gamma}$ defined by the orientation of edges.

Let $V$ be a bivalent pointed vertex, $E$ an edge adjacent to 
$V$ and $V'$ the other vertex of $E$.
The edge $E$ is outgoing for $V$
and ingoing for $V'$, so
$V'=V_E^t$. 
We have in $(\ev^{(p)})^{*}([P^0])
(\ev^{(e)})^{*}([\delta])$ a corresponding factor 
\[ (\ev_{V}^{(p)})^{*}(P_V^0)
\left((\ev_{V}^E)^{*}(\mathrm{pt}_E)
+(\ev_{V'}^E)^{*}(\mathrm{pt}_E)\right)\,.\]
But $(\ev_{V}^{(p)})^{*}(P_V^0)(\ev_{V}^E)^{*}(\mathrm{pt}_E)(-1)^{g(V)} \lambda_{g(V)}=0$ for dimension reasons
(its insertion over
$\overline{M}_{g(V), \Delta_V}$ defines an enumerative problem of virtual dimension $-1$) and so only the factor
$(\ev_{V}^{(p)})^{*}(P_V^0)
(\ev_{V'}^E)^{*}(\mathrm{pt}_E)$ 
survives, which proves the initial step of the induction.

Let $E$ be an outgoing edge of a trivalent vertex $V$, of ingoing edges $E^1$ and
$E^2$. Let $V_E^t$ be the target vertex of $E$. By the induction hypothesis,
every possibly non-vanishing 
term contains the insertion
of $(\ev_V^{E^1})^{*}(\mathrm{pt}_{E^1})
(\ev_V^{E^2})^{*}(\mathrm{pt}_{E^2})$.
But 
$(\ev_V^{E^1})^{*}(\mathrm{pt}_{E^1})
(\ev_V^{E^2})^{*}(\mathrm{pt}_{E^2})
(\ev_V^{E})^{*}(\mathrm{pt}_E)
(-1)^{g(V)} \lambda_{g(V)}=0$
for dimension reasons (its
insertion over $\overline{M}_{g(V), \Delta_V}$ defines an enumerative problem of virtual dimension $-1$)
and so only the factor 
$(\ev_V^{E^1})^{*}(\mathrm{pt}_E^1)
(\ev_V^{E^2})^{*}(\mathrm{pt}_E^2)
(\ev_{V_E^t}^{E})^{*}(\mathrm{pt}_E)$
survives.

Let $E$ be an outgoing edge of a 
bivalent unpointed vertex $V$, of ingoing edges $E^1$. Let $V_E^t$ the target vertex of $E$. By the induction hypothesis,
every possibly non-vanishing 
term contains the insertion
of $(\ev_V^{E^1})^{*}(\mathrm{pt}_{E^1})$.
But 
$(\ev_V^{E^1})^{*}(\mathrm{pt}_{E^1})
(\ev_V^{E})^{*}(\mathrm{pt}_E)
(-1)^{g(V)} \lambda_{g(V)}=0$
for dimension reasons (its
insertion over $\overline{M}_{g(V), \Delta_V}$ defines an enumerative problem of virtual dimension $-1$)
and so only the factor 
$(\ev_V^{E^1})^{*}(\mathrm{pt}_{E^1})
(\ev_{V_E^t}^{E})^{*}(\mathrm{pt}_E)$
survives. This finishes the proof by induction of the claim.

Using the notations introduced 
in Section 
\ref{section_statement_gluing}, we can rewrite 
\[ \prod_{E \in E_f(\tilde{\Gamma})}
(\ev_{V_E^t}^E)^{*}
(\mathrm{pt}_E)\] 
as
\[
\left(
\prod_{V \in V^{(3)}(\tilde{\Gamma})} (\ev_V^{E_V^{\mathrm{in},1}})^{*}
(\mathrm{pt}_{E_V^{\mathrm{in},1}})
(\ev_V^{E_V^{\mathrm{in},2}})^{*}
(\mathrm{pt}_{E_V^{\mathrm{in},2}})
\right)
\kern-1.5ex
\left(
\prod_{V \in V^{(2up)}(\tilde{\Gamma})}
(\ev_V^{E_V^{\mathrm{in}}})^{*}
(\mathrm{pt}_{E_V^{\mathrm{in}}})
\right)
\]
and so we proved
\[
\int_{(\delta \times \iota_{P^0})^!
\underset{V \in V(\tilde{\Gamma})}{\prod} [\overline{M}_{g(V), \Delta_V}]^{\virt}} 
(\delta \times \iota_{P^0})_M^{*} \prod_{V \in V(\tilde{\Gamma})}
\mathrm{pr}_V^{*} 
\left( (-1)^{g(V)} \lambda_{g(V)} \right) \]
\[=
\left( 
\prod_{V \in V^{(3)}(\tilde{\Gamma})}
N^{1,2}_{g(V),V}
\right)
\left(
\prod_{V \in V^{(2p)}(\tilde{\Gamma})}
N^{1,2}_{g(V),V}
\right)
\left(
\prod_{V \in V^{(2up)}(\tilde{\Gamma})}
N^{1,2}_{g(V),V}
\right) \,.  \]
This finishes the proof of Proposition 
\ref{prop_pieces}.
\end{proof}

\subsection{End of the proof of the gluing formula}
The gluing identity given by Proposition
\ref{prop_gluing1}
follows from the combination of 
Proposition 
\ref{prop_gluing_proof} and 
Proposition \ref{prop_pieces}.

\section{Vertex contribution} \label{section: vertex}

In this Section, we evaluate the invariants
$N_{g,V}^{1,2}$ attached to the vertices
$V$ of 
$\Gamma$ and appearing in the gluing formula of Corollary \ref{cor_gluing}.
The first step, carried out in Section 
\ref{section_multiplicity} is to rewrite these invariants in terms of more symmetric invariants $N_{g,V}$ depending only on the 
multiplicity of the vertex $V$.
In \mbox{Section \ref{subsection_reduction}}, we 
use the consistency of the gluing formula to deduce non-trivial relations between these invariants and to reduce the question to the computation of 
the invariants attached to 
vertices of multiplicity one and two.
Invariants attached to vertices of 
multiplicity one and two are explicitly
computed in 
\mbox{Section \ref{subsection_multiplicity_one}}
and this finishes the proof of 
Theorem \ref{main_thm1}. Modifications needed to prove Theorem 
\ref{thm_fixing_points} are discussed at the end of
\mbox{Section \ref{section_general_vertex}}.

\subsection{Reduction to a function of the multiplicity}
\label{section_multiplicity}

The gluing formula of the previous Section,
Corollary \ref{cor_gluing}, expresses the log Gromov-Witten 
invariant 
$N^{\Delta, n}_{g,h}$ attached to a parametrized tropical curve 
$h \colon \Gamma \rightarrow \R^2$
as a product of log Gromov-Witten
$N^{1,2}_{g(V),V}$ attached to 
the trivalent vertices $V$ of $\Gamma$, and of the weights $w(E)$ of the edges $E$ of $\Gamma$.
The definition of $N^{1,2}_{g(V),V}$
given in Section \ref{section_statement_gluing}
depends on a specific choice of orientation on the edges of $\Gamma$.
In particular, the definition of  
$N^{1,2}_{g(V),V}$ does not treat the three edges adjacent to $V$ in a symmetric way.

Let $E_V^{\text{in},1}$ and 
$E_V^{\text{in},2}$
be the two ingoing edges adjacent to $V$, and let 
$E_V^{\text{out}}$ be the outgoing edge adjacent to $V$.
Let $D_{E_V^{\text{in},1}}$, $D_{E_V^{\text{in},2}}$
and $D_{E_V^{\text{out}}}$ be the corresponding toric divisors of 
$X_{\Delta_V}$.
We have evaluation morphisms
\[\text{ev} = (\text{ev}_1, 
\text{ev}_2, \text{ev}_{\text{out}}) \colon \overline{M}_{g,\Delta_V} 
\rightarrow  D_{E_V^{\text{in},1}}
\times
D_{E_V^{\text{in},2}}
\times 
D_{E_V^{\text{out}}} \,.\]
In Section \ref{section_statement_gluing}, we defined 
\[ N_{g,V}^{1, 2} = \int_{[\overline{M}_{g,\Delta_V}]^{\virt}}(-1)^g \lambda_g 
\text{ev}_1^* (\mathrm{pt}_1) \text{ev}_2^* (\mathrm{pt}_2) \,,\]
where $\mathrm{pt}_1 \in A^1 (D_{E_V^{\text{in},1}})$
and
$\mathrm{pt}_2 \in A^1(D_{E_V^{\text{in},2}})$ are 
classes of a point on 
$D_{E_V^{\text{in},1}}$
and 
$D_{E_V^{\text{in},2}}$ respectively.

But one could similarly define
\[ N_{g,V}^{2 , \text{out}} \coloneqq \int_{[\overline{M}_{g,\Delta_V}]^{\virt}}(-1)^g \lambda_g 
\text{ev}_2^* (\mathrm{pt}_2) \text{ev}_{\text{out}}^* (\mathrm{pt}_{\text{out}}) \,,\]
and 
\[ N_{g,V}^{\text{out}, 1} \coloneqq \int_{[\overline{M}_{g,\Delta_V}]^{\virt}}(-1)^g \lambda_g 
\text{ev}_{\text{out}}^* (\mathrm{pt}_{\text{out}}) \text{ev}_1^* (\mathrm{pt}_1) \,,\]
where $\mathrm{pt}_{\text{out}}
\in A^* (D_{E_V^{\text{out}}})$
is the class of a point on 
$E_V^{\text{out}}$.
The following Lemma gives a relation between 
these various invariants.

\begin{lem} \label{lem_vertex}
We have 
\[ N_{g,V}^{1, 2}
w(E_V^{\mathrm{in},1}) w(E_V^{\mathrm{in},2}) = 
N_{g,V}^{2 , \mathrm{out}} w(E_V^{\mathrm{in},2})w(E_V^{\mathrm{out}})
=  N_{g,V}^{\mathrm{out}, 1} w(E_V^{\mathrm{out}}) w(E_V^{\mathrm{in},1})\]
and we denote by $N_{g,V}$ this number.
\end{lem}

\begin{proof}
Let
$\Gamma_V$ 
be the trivalent tropical curve given by $V$ and its three edges $E_V^{\text{in},1}$,
$E_V^{\text{in},2}$ and
$E_V^{\text{out}}$.
Let $\Gamma_{V'}$ be the trivalent tropical curve with a unique vertex $V'$ and edges
$E_{V'}^{\text{in},1}$,
$E_{V'}^{\text{in},2}$ and
$E_{V'}^{\text{out}}$, such that 
\[w(E_V^{\text{in},1})=w(E_{V'}^{\text{in},1})
\,, w(E_V^{\text{in},2})=w(E_{V'}^{\text{in},2})\,, 
w(E_V^{\text{out}})=w(E_{V'}^{\text{out}})\,, \]
and 
\[v_{V,E_V^{\text{in},1}}=-v_{V',E_{V'}^{\text{in},1}}
\,, 
v_{V, E_V^{\text{in},2}}=-v_{V',E_{V'}^{\text{in},2}}\,, 
v_{V, E_V^{\text{out}}}=-v_{V', E_{V'}^{\text{out}}}\,. \]
Let $\Gamma_{V,V'}$ be the tropical curve obtained by gluing 
$E_V^{\text{out}}$ and $E_{V'}^{\text{out}}$
together.

Taking \[ \Delta= \{ v_{V,E_V^{\text{in},1}}, -v_{V',E_{V'}^{\text{in},1}}, 
v_{V, E_V^{\text{in},2}},-v_{V',E_{V'}^{\text{in},2}} \}\]
and $n=3$, we have $g_{\Delta, n}=0$
and 
$T_{\Delta, p}$ consists of a unique tropical curve $\Gamma_{V,V'}^p$, obtained from 
$\Gamma_{V,V'}$ by adding three bivalent vertices
corresponding to the three point 
$p_1$, $p_2$ and $p_3$ in $\R^2$.

Choosing differently $p=(p_1, p_2, p_3)$, the tropical curve 
$\Gamma_{V,V'}^p$ can look like

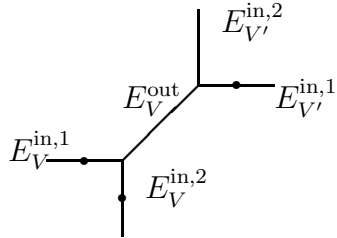
\begin{center}
\setlength{\unitlength}{1cm}
\begin{picture}(10,3)
\thicklines
\put(4,1){\line(1,0){1}}
\put(5,1){\line(0,-1){1}}
\put(5,1){\line(1,1){1}}
\put(6,2){\line(1,0){1}}
\put(6,2){\line(0,1){1}}
\put(3.5,1){$E_V^{\text{in},1}$}
\put(5.3,0.5){$E_V^{\text{in},2}$}
\put(5, 1.7){$E_V^{\text{out}}$}
\put(7, 1.7){$E_{V'}^{\text{in},1}$}
\put(6.3, 2.7){$E_{V'}^{\text{in},2}$}
\put(4.5,1){\circle*{0.1}}
\put(5,0.5){\circle*{0.1}}
\put(6.5,2){\circle*{0.1}}
\end{picture}
\end{center}

or like 

\begin{center}
\setlength{\unitlength}{1cm}
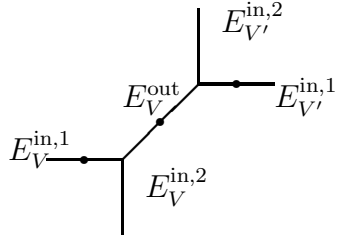
\begin{picture}(10,3)
\thicklines
\put(4,1){\line(1,0){1}}
\put(5,1){\line(0,-1){1}}
\put(5,1){\line(1,1){1}}
\put(6,2){\line(1,0){1}}
\put(6,2){\line(0,1){1}}
\put(3.5,1){$E_V^{\text{in},1}$}
\put(5.3,0.5){$E_{V}^{\text{in},2}$}
\put(5, 1.7){$E_V^{\text{out}}$}
\put(7, 1.7){$E_{V'}^{\text{in},1}$}
\put(6.3, 2.7){$E_{V'}^{\text{in},2}$}
\put(4.5,1){\circle*{0.1}}
\put(5.5,1.5){\circle*{0.1}}
\put(6.5,2){\circle*{0.1}}
\end{picture}
\end{center}

But the log Gromov-Witten invariants 
$N^{\Delta, 3}_g$ are independent of the choice of $p$ and so can be computed for any choice of $p$.
For each of the two above choices of $p$, the gluing formula of 
Corollary \ref{cor_gluing} gives an expression for $N^{\Delta, 3}_g$. 
These two expressions have to be equal. 
Writing 
\[ F(u) = \sum_{g \geqslant 0} N_g u^{2g+1} \, \]
we obtain\footnote{Recall that we are considering marked points as bivalent vertices and that this affects the notion of bounded edge. According to
the gluing formula of 
Corollary \ref{cor_gluing}, we need to include one weight factor for each bounded edge. } 
\[ F^{1,2}_V(u) 
F^{1, \text{out}}_{V'}(u)
w(E_{V}^{\text{in},1})
w(E_{V}^{\text{in},2})
w(E_{V}^{\text{out}})
w(E_{V'}^{\text{in,1}})\]
\[= F^{1, \text{out}}_{V}(u)  
F^{1, \text{out}}_{V'}(u)
w(E_{V}^{\text{in},1})
w(E_V^{\text{out}})
w(E_V^{\text{out}})
w(E_{V'}^{\text{in},1})\,, \]
and so after simplification
\[ F^{1,2}_V(u) 
F^{1, \text{out}}_{V'}(u) w(E_{V}^{\text{in},2})
= F^{1, \text{out}}_{V}(u)  
F^{1, \text{out}}_{V'}(u)w(E_V^{\text{out}})\,. \]
By $GL_2(\Z)$ invariance, we have 
$F^{1,2}_V(u)=F^{1,2}_{V'}(u)$
and
$F^{1, \text{out}}_{V}(u)=F^{1, \text{out}}_{V'}(u)$.
By the unrefined correspondence theorem, we know that 
$F^{1, \text{out}}_{V}(u) \neq 0$, so 
we obtain 
\[F^{1,2}_V(u) 
w(E_{V}^{\text{in},2})
= F^{1, \text{out}}_{V}(u) w(E_V^{\text{out}})\,,\]
which finishes the proof of Lemma
\ref{lem_vertex}.
\end{proof}

We define the contribution $F_V(u)
\in \Q[\![u]\!]$
of a trivalent vertex $V$ 
of $\Gamma$ as being the
power series 
\[F_V(u)= \sum_{g \geqslant 0} N_{g,V} u^{2g+1}.\]

\begin{prop} \label{prop_gluing2}
For every $\Delta$ and $n$ such that 
$g_{\Delta, n} \geqslant 0$, and for every 
$p \in U_{\Delta, n}$, we have 
\[ \sum_{g \geqslant g_{\Delta,n}} 
N_g^{\Delta, n} u^{2g-2+|\Delta|} 
= \sum_{(h \colon \Gamma 
\rightarrow \R^2) \in T_{\Delta, p}}
\prod_{V \in V^{(3)}(\Gamma)}
F_V(u) \, \]
where the product is over the trivalent vertices of $\Gamma$.
\end{prop}

\begin{proof}
This follows from the decomposition formula, 
Proposition \ref{prop_decomposition}, from the 
gluing formula, Corollary 
\ref{cor_gluing}, and from 
Lemma \ref{lem_vertex}.
Indeed, every bounded edge of $\Gamma$ is an
ingoing edge for exactly one trivalent vertex of $\Gamma$ and every trivalent vertex of 
$\Gamma$ has exactly two ingoing edges. 
Combining the invariant 
$N^{1,2}_{g(V),V}$ of a trivalent vertex $V$ with the weights of its two ingoing 
edges, one can rewrite the double 
product of Corollary 
\ref{cor_gluing} as a single product
in terms of the invariants defined by 
Lemma \ref{lem_vertex}.
\end{proof}

\begin{prop} \label{prop_mult}
The contribution $F_V(u)$ of a vertex $V$ only depends on 
the multiplicity
$m(V)$ of $V$.

In particular,
for every $m$
positive integer,
one can define the contribution 
$F_m(u)
\in \Q[\![u]\!]$ 
as the contribution 
$F_V(u)$ of a vertex 
$V$
of multiplicity 
$m$.
\end{prop}

\begin{proof}
We follow closely Brett Parker, \cite{parker2016three} (Section 3).

For $v_1, v_2 \in \Z^2-\{0\}$,
let us denote by 
$F_{v_1, v_2}(u)$
the contribution
$F_V(u)$
of a vertex $V$ of adjacent edges
$E_1$, $E_2$ and $E_3$
such that 
$v_{V,E_1}=v_1$ and 
$v_{V,E_2}=v_2$.
The contribution 
$F_{v_1, v_2}(u)$ depends on 
$(v_1, v_2)$ only up to 
linear action of 
$GL_2(\Z)$ on $\Z^2$.
In particular, we can change the 
sign of $v_1$ and/or $v_2$ without changing 
$F_{v_1, v_2}(u)$.

By the balancing condition, we have 
$v_{V,E_3}=-v_{V,E_1}-v_{V,E_2}$
and so
\[F_{v_1, v_2}(u)
= F_{-v_1, v_2}(u)
=F_{v_1-v_2, v_2}(u)\,.\]
By $GL_2(\Z)$ invariance, we can 
assume $v_1=(|v_1|,0)$
and $v_2=(v_{2x},*)$
with $v_{2x} \geqslant 0$.
If $|v_1|$ divides $v_{2x}$, 
$v_{2x}=a|v_1|$, then replacing
$v_2$ by $v_2-av_1$, which does not change $F_{v_1,v_2}$, 
we can assume that $v_1=(|v_1|,0)$ and 
$v_2=(0, *)$.
If not, we do the Euclidean division of 
$v_{2x}$ by $|v_1|$,
$v_{2x}=a|v_1|+b$, $0\leqslant b<|v_1|$, and
we replace $v_2$ by $v_2-av_1$ to obtain $v_2=(b,*)$.
Exchanging the roles of $v_1$ and $v_2$, we can assume by 
$GL_2(\Z)$ invariance that 
$v_1=(|v_1|,0)$, for some $|v_1| \leqslant b$ and $v_2=(v_{2x},*)$
for some $v_{2x}\geqslant 0$, and we repeat the above procedure. 
By the Euclidean algorithm, this process terminates and at the end we have
$v_1=(|v_1|,0)$ and $v_2=(0,|v_2|)$.
In particular, for every 
$v_1, v_2 \in \Z^2-\{0\}$, the contribution 
$F_{v_1,v_2}$
only depends on $\text{gcd}(|v_1|,
|v_2|)$
and on the multiplicity 
$|\det (v_1, v_2)|$.

By the previous paragraph, we can 
assume that $v_1=(|v_1|,0)$
and $v_2=(0,|v_2|)$. 

Taking 
\[ \Delta = 
\{ (|v_1|,0), (0,|v_2|), (0,1),
(-|v_1|, -|v_2|-1) \}\,,\]
and $n=3$, we have $g_{\Delta, n}
=0$ and $T_{\Delta, p}$
contains a unique 
tropical curve 
$\Gamma^p$.

Choosing differently 
$p=(p_1, p_2, p_3)$, the tropical curve $\Gamma^p$ can look like 

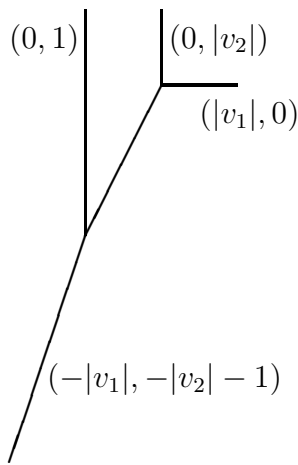
\begin{center}
\setlength{\unitlength}{1cm}
\begin{picture}(10,6)
\thicklines
\put(5,3){\line(-1,-3){1}}
\put(5,3){\line(0,1){3}}
\put(6,5){\line(-1,-2){1}}
\put(6,5){\line(1,0){1}}
\put(6,5){\line(0,1){1}}
\put(6.5,4.5){$(|v_1|,0)$}
\put(6.1,5.5){$(0,|v_2|)$}
\put(4,5.5){$(0,1)$}
\put(4.5,1){$(-|v_1|,-|v_2|-1)$}
\end{picture}
\end{center}

or like 

\begin{center}
\setlength{\unitlength}{1cm}
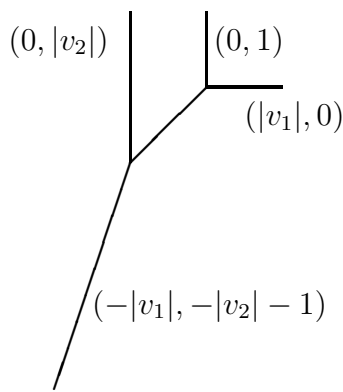
\begin{picture}(10,6)
\thicklines
\put(5,3){\line(-1,-3){1}}
\put(5,3){\line(0,1){2}}
\put(6,4){\line(-1,-1){1}}
\put(6,4){\line(1,0){1}}
\put(6,4){\line(0,1){1}}
\put(6.5,3.5){$(|v_1|,0)$}
\put(6.1,4.5){$(0,1)$}
\put(3.4,4.5){$(0,|v_2|)$}
\put(4.5,1){$(-|v_1|,-|v_2|-1)$}
\end{picture}
\end{center}

But the log Gromov-Witten invariants
$N^{\Delta, 3}_g$ are independent of the choice of $p$ 
and so can be computed for any choice of $p$. For each of the two above choices of $p$, 
the gluing formula of 
Proposition \ref{prop_gluing2}
gives an expression for 
$N^{\Delta, 3}_g$. These two expressions have to be equal and we obtain
\[ F_{(|v_1|,0), (0,|v_2|)}(u)
F_{(0,1), (-|v_1|,-|v_2|-1)}(u)\]
\[=F_{(|v_1|,0),(0,1)}(u)
F_{(0,|v_2|),(-|v_1|, -|v_2|-1)}(u) \,.\]
For both pairs of vectors 
$(|v_1|,0), (0,1)$ and $(0,1), (-|v_1|,-|v_2|-1)$, the gcd of the
divisibilities is equal to one and the absolute value of the
determinant is equal to 
$|v_1|$, so we have
\[F_{(0,1), (-|v_1|,-|v_2|-1)}(u) = F_{(|v_1|,0),(0,1)}(u)\,. \]
As this quantity is non-zero 
by the unrefined correspondence theorem, we can simplify it from the previous equality to obtain 
\[ F_{(|v_1|,0),(0,
|v_2|)}(u)=
F_{(0,|v_2|),(-|v_1|, -|v_2|-1)}(u) \,.\]
As 
\[ \text{gcd}(|(0,|v_2|)|, |(-|v_1|, -|v_2|-1)|)=1\,,\] 
we obtain the desired result.
\end{proof}

\subsection{Reduction to vertices of multiplicity $1$ and $2$}
\label{subsection_reduction}
We start reviewing  the key step in the argument of Itenberg and Mikhalkin
\cite{MR3142257} 
proving
the tropical deformation invariance of 
Block-G\"ottsche invariants.
We consider a tropical curve
with 
a 4-valent vertex $V$. 
Let $Q$ be the
quadrilateral dual to $V$.
We assume that $Q$ has no pair of 
parallel sides.
In that case, there exists a unique parallelogram $P$ having two sides
in common with $Q$ and being 
contained in $Q$.
Let $A$,$B$,$C$ and $D$ denote the 
four vertices of $Q$, such that 
$A$,$B$ and $D$ are vertices of 
$P$. Let $E$ be the fourth vertex of
$P$, contained in the interior of $Q$.
There are three combinatorially 
distinct ways to deform this tropical curve into a simple one, 
corresponding to the three ways to decompose $Q$ into triangles or parallelograms: 
\begin{enumerate}
\item We can decompose $Q$ into the triangles $ABD$ and $BCD$.
\item We can decompose $Q$ into the triangles $ABC$ and $ACD$.
\item We can decompose $Q$ into the 
triangles $BCE$, $DEC$ and the 
parallelogram $P$.
\end{enumerate}

Case (1):

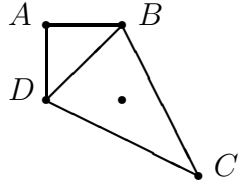
\begin{center}
\setlength{\unitlength}{1cm}
\begin{picture}(10,3)
\thicklines
\put(4,2){\circle*{0.1}}
\put(5,2){\circle*{0.1}}
\put(4,3){\line(1,0){1}}
\put(4,3){\line(0,-1){1}}
\put(4,3){\circle*{0.1}}
\put(5,3){\circle*{0.1}}
\put(6,1){\circle*{0.1}}
\put(4,2){\line(2,-1){2}}
\put(5,3){\line(1,-2){1}}
\put(4,2){\line(1,1){1}}
\put(3.5,3){$A$}
\put(5.2,3){$B$}
\put(3.5,2){$D$}
\put(6.2,1){$C$}
\end{picture}
\end{center}

Case (2):

\begin{center}
\setlength{\unitlength}{1cm}
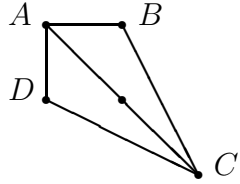
\begin{picture}(10,3)
\thicklines
\put(4,2){\circle*{0.1}}
\put(5,2){\circle*{0.1}}
\put(4,3){\line(1,0){1}}
\put(4,2){\line(0,1){1}}
\put(4,3){\circle*{0.1}}
\put(5,3){\circle*{0.1}}
\put(6,1){\circle*{0.1}}
\put(4,2){\line(2,-1){2}}
\put(5,3){\line(1,-2){1}}
\put(4,3){\line(1,-1){2}}
\put(3.5,3){$A$}
\put(5.2,3){$B$}
\put(3.5,2){$D$}
\put(6.2,1){$C$}
\end{picture}
\end{center}

Case (3):

\begin{center}
\setlength{\unitlength}{1cm}
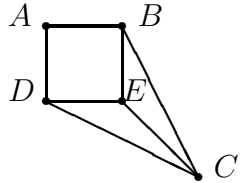
\begin{picture}(10,3)
\thicklines
% \put(8,3.3){{\footnotesize $3$-simplex}}
\put(4,2){\circle*{0.1}}
\put(5,2){\circle*{0.1}}
\put(4,2){\line(1,0){1}}
\put(4,3){\line(1,0){1}}
\put(4,3){\line(0,-1){1}}
\put(5,3){\line(0,-1){1}}
\put(4,3){\circle*{0.1}}
\put(5,3){\circle*{0.1}}
\put(6,1){\circle*{0.1}}
\put(4,2){\line(2,-1){2}}
\put(5,3){\line(1,-2){1}}
\put(5,2){\line(1,-1){1}}
\put(3.5,3){$A$}
\put(5.2,3){$B$}
\put(3.5,2){$D$}
\put(6.2,1){$C$}
\put(5,2){$E$}
\end{picture}
\end{center}

The deformation invariance result then follows from the identity
\[(q^{|ACD|}-q^{-|ACD|})
(q^{|ABC|}-q^{-|ABC|})\]
\begin{align*}
=&(q^{|BCD|}-q^{-|BCD|})
(q^{|ABD|}-q^{-|ABD|})\\ 
&+(q^{|BCE|}-q^{-|BCE|})
(q^{|DEC|}-q^{-|DEC|})\,,
\end{align*}
where $|-|$ denotes the area. This identity can be proved by elementary geometry
considerations.

The following result goes in the opposite direction and shows that the 
constraints imposed by tropical deformation invariance are quite strong.
The generating series of log Gromov-Witten
invariants $F_m(u)$ will satisfy these constraints.
Indeed, they are defined independently of any tropical limit, so applications of the gluing formula to different 
degenerations have to give the same result.

\begin{prop} \label{prop_red_mult}
Let $F \colon \Z_{>0} \rightarrow R$
be a function of positive integers 
valued in a commutative ring $R$, such that, 
for any 
quadrilateral $Q$ as above, we have\footnote{All the relevant areas are half-integers and so their doubles are indeed integers.}
\[F(2|BCD|)F(2|ABD|)=
F(2|ACD|)F(2|ABC|)+
F(2|BCE|)F(2|DEC|).\]
Then for every integer $n \geqslant 2$, we have 
\[F(n)^2=F(2n-1)F(1)+F(n-1)^2\]
and for every integer 
$n \geqslant 3$, we have 
\[F(n)^2=F(2n-2)F(2)+F(n-2)^2.\]
In particular, if $F(1)$ and $F(2)$ are invertible in $R$, then 
the function $F$ is completely
determined by its values $F(1)$ and $F(2)$.
\end{prop}

\begin{proof}
The first equality is obtained by taking $Q$
to be the quadrilateral of vertices 
$(-1,0)$, $(-1,1)$, $(0,1)$, $(n-1, -(n-1))$.

Picture of $Q$ for $n=2$:

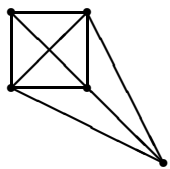
\begin{center}
\setlength{\unitlength}{1cm}
\begin{picture}(10,4)
\thicklines
\put(4,2){\circle*{0.1}}
\put(5,2){\circle*{0.1}}
\put(4,2){\line(1,0){1}}
\put(4,3){\line(1,0){1}}
\put(4,3){\line(0,-1){1}}
\put(5,3){\line(0,-1){1}}
\put(4,3){\circle*{0.1}}
\put(5,3){\circle*{0.1}}
\put(6,1){\circle*{0.1}}
\put(4,2){\line(2,-1){2}}
\put(5,3){\line(1,-2){1}}
\put(4,3){\line(1,-1){2}}
\put(4,2){\line(1,1){1}}
\end{picture}
\end{center}

The second equality is obtained by taking 
$Q$ to be the quadrilateral of vertices 
$(-1,0)$, $(-1,1)$, $(1,0)$, $(n-1,-(n-1))$.

Picture of $Q$ for $n=3$:

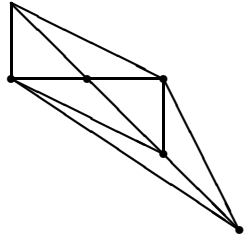
\begin{center}
\setlength{\unitlength}{1cm}
\begin{picture}(10,5)
\thicklines
\put(4,3){\circle*{0.1}}
\put(5,3){\circle*{0.1}}
\put(6,3){\circle*{0.1}}
\put(4,4){\line(2,-1){2}}
\put(4,4){\line(0,-1){1}}
\put(6,2){\circle*{0.1}}
\put(4,3){\line(3,-2){3}}
\put(4,4){\line(1,-1){3}}
\put(7,1){\circle*{0.1}}
\put(4,3){\line(2,-1){2}}
\put(4,3){\line(1,0){2}}
\put(6,3){\line(1,-2){1}}
\put(6,3){\line(0,-1){1}}
\end{picture}
\end{center}

\end{proof}

\subsection{Contribution of vertices of multiplicity $1$ and $2$}
\label{subsection_multiplicity_one}

\subsubsection{Vertex of multiplicity one}
We now evaluate the contribution 
$F_1(u)$
of a vertex of multiplicity $1$
by \mbox{direct computation.}

We consider 
$\Delta = \{(-1,0), (0,-1), (1,1)\}$.
The corresponding toric surface
$X_\Delta$ is simply $\PP^{2}$,
of fan

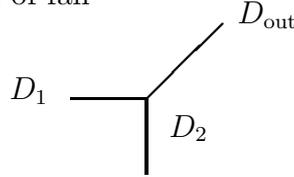
\begin{center}
\setlength{\unitlength}{1cm}
\begin{picture}(10,2)
\thicklines
\put(4,1){\line(1,0){1}}
\put(5,1){\line(0,-1){1}}
\put(5,1){\line(1,1){1}}
\put(3.2,1){$D_{1}$}
\put(5.3,0.5){$D_{2}$}
\put(6.2,2){$D_{\text{out}}$}
\end{picture}
\end{center}
and of dual polygon 

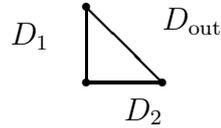
\begin{center}
\setlength{\unitlength}{1cm}
\begin{picture}(10,2)
\thicklines
\put(4,1){\circle*{0.1}}
\put(5,1){\circle*{0.1}}
\put(4,1){\line(1,0){1}}
\put(4,2){\line(1,-1){1}}
\put(4,2){\line(0,-1){1}}
\put(4,2){\circle*{0.1}}
\put(3,1.5){$D_{1}$}
\put(4.5,0.5){$D_{2}$}
\put(5,1.7){$D_{\text{out}}$}
\end{picture}
\end{center}

Let $D_1$, $D_2$ and 
$D_{\text{out}}$ be the 
toric boundary divisors of $\PP^2$.
The class $\beta_{\Delta}$ is simply the class of a curve of degree one, i.e.\ of a line, on 
$\PP^2$.
Let $\overline{M}_{g,\Delta}$ be the moduli space of genus $g$ stable log maps
of type $\Delta$.
We have 
evaluation maps
\[ (\text{ev}_1, 
\text{ev}_2) \colon \overline{M}_{g,\Delta} \rightarrow  D_1
\times
D_2
 \,,\]
and in Section \ref{section_statement_gluing}, we defined 
\[ N_{g,\Delta}^{1, 2} = \int_{[\overline{M}_{g,\Delta_V}]^{\virt}}(-1)^g \lambda_g 
\text{ev}_1^* (\mathrm{pt}_1) \text{ev}_2^* (\mathrm{pt}_2) \,,\]
where $\mathrm{pt}_1 \in A^* (D_1)$
and $\mathrm{pt}_2 \in A^*(D_2)$
are classes of a point on $D_1$ and $D_2$
respectively.

By definition
(see Section 
\ref{section_multiplicity}),
we have 
\[F_1(u)=
\sum_{g \geqslant 0} N_{g,\Delta}^{1, 2} u^{2g+1}\,.\]

\begin{prop} \label{prop_mult1}
The contribution of a vertex
of multiplicity one is given by 
\[F_1(u) = 2\sin \left(\frac{u}{2} \right) = -i(q^{\frac{1}{2}}-q^{-\frac{1}{2}})\]
where $q=e^{iu}$.
\end{prop}

\begin{proof}

Let $P_1$ and $P_2$ be points on 
$D_1$ and $D_2$
respectively, away 
from the torus fixed 
points.
Let $S$ be the surface obtained by blowing-up
$\PP^2$ at $P_1$ and $P_2$. 
Denote by $D$ the strict transform of the class of a
line in $\PP^2$ and by $E_1$, $E_2$ the exceptional divisors.
Denote $\partial S$ the strict transform of the toric boundary 
$\partial \PP^2$
of 
$\PP^2$. We endow $S$ with the divisorial log structure with 
respect to 
$\partial S$.
Let $\overline{M}_g(S)$ be the moduli space of genus $g$
stable log maps to $S$ of class 
$D-E_1-E_2$ with tangency condition to intersect
$\partial S$ in one point with multiplicity one.
It has virtual dimension $g$ and we define
\[ N_g^S \coloneqq
\int_{[\overline{M}_g(S)]^{\virt}} 
(-1)^g \lambda_g \,.\]
 
The strict transform $C$
of the line $L$ in $\PP^2$
passing through $P_1$ and $P_2$ is the unique genus zero curve satisfying these conditions and has normal bundle 
$N_{C|S}=\cO_{\PP^1}(-1)$
in $S$. All the higher genus maps factor
through $C$, and as $C$ is away from the preimage of the torus fixed points of $\PP^2$,
log invariants coincide with 
relative invariants \cite{MR3329675}. 
More precisely, we can consider the moduli space 
$\overline{M}_g(\PP^1/\infty, 
1,1)$ genus $g$
stable maps to $\PP^1$, of degree one, and relative to 
a point $\infty \in \PP^1$. 
If $\pi \colon \cC \rightarrow
\overline{M}_g(\PP^1/\infty, 
1,1) $ is the universal curve and 
$f \colon \cC \rightarrow \PP^1 \simeq C$ is the 
universal map, the difference in obstruction theories 
between stable maps to $S$ and stable maps to $\PP^1$ comes from 
$R^1 \pi_* f^* N_{C|S}=R^1 \pi_* f^* \cO_{\PP^1}(-1)$.
So we obtain 
\[N_g^S= \int_{[\overline{M}_g(\PP^1/\infty, 
1,1)]^{\virt}}(-1)^g \lambda_g
\, e \left(R^1 \pi_* f^{*} \cO_{\PP^1} (-1) \right) \,,\]
where $e(-)$ is the Euler class.
Rewriting 
\[(-1)^g \lambda_g = e(R^1 \pi_* \cO_{\cC})
=e(R^1 \pi_* f^* \cO_{\PP^1})\,,\] we get 
\[N_g^S= \int_{[\overline{M}_g(\PP^1/\infty, 
1,1)]^{\virt}}
e \left(R^1 \pi_* f^*
(\cO_{\PP^1} \oplus \cO_{\PP^1} (-1)) \right) \,.\]
These integrals have been computed by Bryan and Pandharipande\cite{MR2115262}, (see the proof of the
Theorem 5.1), and the result is
\[\sum_{g \geqslant 0} N_g^S u^{2g-1} = \frac{1}{2 \sin \left( \frac{u}{2} \right) } \,.\]

As in \cite{MR2667135}, we will work with the non-compact varieties $(\PP^2)^\circ$, 
$D_1^\circ$, $D_2^\circ$, $S^\circ$ 
obtained by removing the torus fixed points of $\PP^2$ and their preimages in $S$.

Denote 
$\PP_1^\circ$ the projectivized 
normal bundle to $D_1^\circ$ in 
$(\PP^2)^\circ$, coming with two natural sections $(D_1^\circ)_0$ and 
$(D_1^\circ)_\infty$. Denote 
$\tilde{\PP}_1^\circ$ the blow-up of
$\PP_1^\circ$ at the point $P_1 \in 
(D_1^\circ)_\infty$, $\tilde{E}_1$
the corresponding exceptional divisor and 
$C_1$ the strict transform of the fiber 
of $\PP_1^\circ$ passing through $P_1$. In particular, $\tilde{E}_1$ and 
$C_1$ are both projective lines with degree
$-1$ normal bundle in $(\tilde{\PP}_1)^\circ$. Furthermore, $\tilde{E}_1$  
and $C_1$ intersect in one point.
Similarly,
denote 
$\PP_2^\circ$ the projectivized 
normal bundle to $D_2^\circ$ in 
$(\PP^2)^\circ$, coming with two natural sections $(D_2^\circ)_0$ and 
$(D_2^\circ)_\infty$. Denote 
$\tilde{\PP}_2^\circ$ the blow-up of
$\PP_2^\circ$ at the point $P_2 \in 
(D_2^\circ)_\infty$, $\tilde{E}_2$
the corresponding exceptional divisor and 
$C_2$ the strict transform of the fiber 
of $\PP_2^\circ$ passing through $P_2$. In particular, $\tilde{E}_2$ and 
$C_2$ are both projective lines with degree
$-1$ normal bundle in $(\tilde{\PP}_2)^\circ$. Furthermore, $\tilde{E}_2$  
and $C_2$ intersect in one point.

We degenerate
$S^\circ$ as in Section
5.3 of \cite{MR2667135}.
We first degenerate $(\PP^2)^\circ$ to the normal cone of $D_1^\circ \cup D_2^\circ$, i.e.\ we blow-up $(D_1^\circ \cup D_2^\circ) \times \{0\}$ in 
$(\PP^2)^{\circ} \times \C$.
The fiber over $0 \in \C$ has three
irreducible components: 
$(\PP^2)^\circ$, $\PP_1^\circ$, $\PP_2^\circ$, with $\PP_1^\circ$ and $\PP_2^\circ$ glued along 
$(D_1^\circ)_0$ and 
$(D_2^\circ)_0$ 
to $D_1^\circ$ and 
$D_2^\circ$ in $(\PP^2)^\circ$.
We then blow-up the strict transforms of the sections 
$P_1 \times \C$ and $P_2 \times \C$.
The fiber of the resulting family away from
$0 \in \C$ is isomorphic to $S^\circ$.
The fiber over zero has three
irreducible components: 
$(\PP^2)^{\circ}$, $\tilde{\PP}_1^{\circ}$,
$\tilde{\PP}_2^{\circ}$.

We would like to apply a degeneration formula to this family in order to compute $N_g^S$.
As discussed above, all the maps in 
$\overline{M}_g(S)$ factor through $C$ and so $N_g^S$ can be seen as a relative
Gromov-Witten invariant of the non-compact surface $S^\circ$, 
relatively to the strict transforms of 
$D_1^\circ$ and $D_2^\circ$.

The key point is that for homological degree reasons, 
the degenerating relative stable maps do not leave the non-compact geometries we are considering. More precisely, any limiting relative stable map has to factor through 
$C_1 \cup L \cup C_2$, with degree one over
each of the components $C_1$, $L$ and $C_2$.
So, even if the target geometry is non-compact, all the relevant moduli spaces of relative stable maps are compact.
It follows that
we can apply  
the ordinary
degeneration formula 
in relative Gromov-Witten theory
\cite{MR1938113}. 

We obtain 
\[\sum_{g \geqslant 0} N_g^S u^{2g-1} = \left( \sum_{g \geqslant 0} N_{g, \Delta}^{1, 2} u^{2g+1} \right)
\left( \sum_{g \geqslant 0} N_g^{C_1} u^{2g-1} \right)
\left( \sum_{g \geqslant 0} N_g^{C_2} u^{2g-1} \right) \,.\]
The invariants $N_g^{C_1}$ and 
$N_g^{C_2}$, coming from curves factoring through $C_1$ and $C_2$, which are 
$(-1)$-curves in $\tilde{\PP}_1^\circ$ and 
$\tilde{\PP}_2^\circ$ respectively,
can be written as relative invariants of $\PP^1$:  
\[N_g^{C_1}=N_g^{C_2}= \int_{[\overline{M}_g(\PP^1/\infty, 
1,1)]^{\virt}}
e \left( R^1 \pi_* f^* (\cO_{\PP^1} \oplus \cO_{\PP^1} (-1)) \right) \,,\]
which is exactly the formula giving $N_g^S$, and so 
\[\sum_{g \geqslant 0} N_g^{C_1} u^{2g-1} 
=\sum_{g \geqslant 0} N_g^{C_2} u^{2g-1}
=\frac{1}{2 \sin \left( \frac{u}{2} \right)} \,.\]
Remark that this equality is a higher genus version of Proposition 5.2 of \cite{MR2667135}. 
Combining the previous equalities, we obtain 
\[\frac{1}{2 \sin \left( \frac{u}{2} \right)} = \left( \sum_{g \geqslant 0} N_{g, \Delta}^{1, 2} u^{2g+1} \right)
\left( \frac{1}{2 \sin \left( \frac{u}{2} \right) } \right)^2 \,,\]
and so 
\[\sum_{g \geqslant 0} N_{g, \Delta}^{1, 2} u^{2g+1} = 2 \sin \left( \frac{u}{2} \right) \,.\]
\end{proof}

\subsubsection{Vertex of multiplicity $2$}
\label{subsection_multiplicity_two}
We now evaluate the contribution 
$F_2(u)$ of a vertex of multiplicity $2$
by \mbox{direct computation.} 

We consider
$\Delta 
=\{(-1,0), (0,-2), (1,2)\}$.
The corresponding toric surface
$X_\Delta$ is simply the weighted 
projective plane $\PP^{1,1,2}$,
of fan

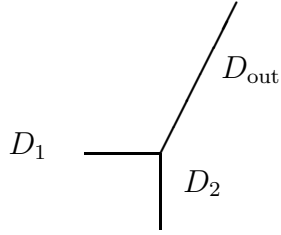
\begin{center}
\setlength{\unitlength}{1cm}
\begin{picture}(10,3)
\thicklines
\put(4,1){\line(1,0){1}}
\put(5,1){\line(0,-1){1}}
\put(5,1){\line(1,2){1}}
\put(3,1){$D_{1}$}
\put(5.3,0.5){$D_{2}$}
\put(5.8,2){$D_{\text{out}}$}
\end{picture}
\end{center}
and of dual polygon

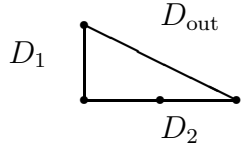
\begin{center}
\setlength{\unitlength}{1cm}
\begin{picture}(10,2)
\thicklines
\put(4,1){\circle*{0.1}}
\put(5,1){\circle*{0.1}}
\put(6,1){\circle*{0.1}}
\put(4,1){\line(1,0){2}}
\put(4,2){\line(2,-1){2}}
\put(4,2){\line(0,-1){1}}
\put(4,2){\circle*{0.1}}
\put(3,1.5){$D_{1}$}
\put(5,0.5){$D_{2}$}
\put(5,2){$D_{\text{out}}$}
\end{picture}
\end{center}

Let $D_1$, $D_2$ and $D_{\text{out}}$ be the toric 
boundary divisors of 
$\PP^{1,1,2}$.
We have  the following 
numerical properties: 
\[2D_1=D_2=2D_{\text{out}}\,,\]
\[D_1.D_2=1,\, D_1.D_{\text{out}}=\frac{1}{2},\,
D_2.D_{\text{out}}=1\,,\]
\[ D_1^2=\frac{1}{2}, D_2^2=2, 
D_{\text{out}}^2=\frac{1}{2}\,.\]
The class $\beta_\Delta$ satisfies 
$\beta_\Delta .D_1=1$,
$\beta_\Delta .D_2=2$,
$\beta_\Delta .D_{\text{out}}=1$
and so 
\[ \beta_\Delta = 2D_1=D_2=2D_{\text{out}}\,.\]
Let $\overline{M}_{g,\Delta}$ be the moduli space of genus $g$ stable log maps
of type $\Delta$.
We have 
evaluation maps
\[ (\text{ev}_1, 
\text{ev}_2) \colon \overline{M}_{g,\Delta} \rightarrow  D_1
\times
D_2
 \,,\]
and in Section \ref{section_statement_gluing}, we defined 
\[ N_{g,\Delta}^{1, 2} = \int_{[\overline{M}_{g,\Delta_V}]^{\virt}}(-1)^g \lambda_g 
\text{ev}_1^* (\mathrm{pt}_1) \text{ev}_2^* (\mathrm{pt}_2) \,,\]
where $\mathrm{pt}_1 \in A^* (D_1)$
and $\mathrm{pt}_2 \in A^*(D_2)$
are classes of a point on $D_1$ and $D_2$
respectively.

By definition
(see Section 
\ref{section_multiplicity}),
we have 
\[F_1(u)=
2 \left( \sum_{g \geqslant 0} N_{g,\Delta}^{1, 2} u^{2g+1} \right) \,.\]

\begin{prop} \label{prop_mult2}
The contribution of a vertex of 
multiplicity two is given by
\[ F_2(u)= 2 \sin(u) = (-i) (q-q^{-1})\]
where $q=e^{iu}$.
\end{prop}

\begin{proof}
We have to prove that
\[\sum_{g \geqslant 0} N_{g,\Delta}^{1, 2} u^{2g+1} = \sin (u) \,. \]

Let $P_2$ be a point on 
$D_2$ away 
from the torus fixed 
points.
Let $S$ be the surface obtained by blowing-up $\PP^{1,1,2}$ at $P_2$. 
Still denote $\beta_\Delta$ the strict transform of the class $\beta_\Delta$ 
and by $E_2$ the exceptional divisor. 
Denote $\partial S$ the strict 
transform of the toric boundary 
$\partial \PP^{1,1,2}$ of 
$\PP^{1,1,2}$. 
We endow $S$ with the divisorial 
log structure with respect to 
$\partial S$.
Let $\overline{M}_g(S)$
be the moduli space of genus $g$
stable log maps to $S$ of class
$\beta_\Delta-2E_2$
with tangency condition to intersect 
$D_1$ in one point 
with multiplicity one
and $D_{\text{out}}$ 
in one point with multiplicity one.
It has virtual dimension $g$ and we
have an evaluation map 
\[ \text{ev}_1 \colon 
\overline{M}_{g,S}
\rightarrow D_1 \,\] 
We define
\[ N_g^S \coloneqq
\int_{[\overline{M}_g(S)]^{\virt}} 
(-1)^g \lambda_g \text{ev}_1^*(\mathrm{pt}_1)\,,\]
where 
$\mathrm{pt}_1 \in A^1(D_1)$ is the class of a point on 
$D_1$.
  
In fact, because a curve in the linear system
$\beta_\Delta-2E_2$ is of arithmetic genus $g_a$ given by 
\begin{align*}
2g_a-2&=(\beta_\Delta-2E_2) \cdot(\beta_\Delta -2E_2+K_S) \\
&=
(2D_1-2E_2) \cdot (2D_1-4D_1-E_2) \\
&=-4D_1^2+2E_2^2\\
&=-4 \,,
\end{align*}
i.e.\ $g_a=-1<0$, all the moduli spaces 
$\overline{M}_g(S)$ are empty and so 
\[ \sum_{g \geqslant 0} N_g^S u^{2g-1}=0 \,.\]

We write 
$\tilde{\Delta}
=\{(-1,0), (0,-1), (0,-1), (1,2)\}$
and $\overline{M}_{g,\tilde{\Delta}}$
the moduli space of genus 
$g$ stable log maps of type 
$\tilde{\Delta}$. 
We have evaluation maps 
\[(\text{ev}_1, \text{ev}_2, \text{ev}_{2'}) 
\colon \overline{M}_{g,\tilde{\Delta}} \rightarrow D_1 \times D_2 \times D_2 \,,\]
and we define 
\[N_{g,\tilde{\Delta}}^{1,2,2'}
\coloneqq \int_{[\overline{M}_{g,\tilde{\Delta}}]^{\virt}} 
(-1)^g \lambda_g 
\text{ev}_1^*(\mathrm{pt}_1) \text{ev}_2^*(\mathrm{pt}_2)
\text{ev}_{2'}^*(\mathrm{pt}_2) \,,
\]
where $\mathrm{pt}_1 \in A^* (D_1)$
and $\mathrm{pt}_2 \in A^*(D_2)$
are classes of a point on $D_1$ and $D_2$
respectively.

As in \cite{MR2667135}, we will work with the non-compact varieties $(\PP^{1,1,2})^\circ$, $D_1^\circ$, 
$D_2^\circ$, $S^\circ$ 
obtained by removing the torus fixed 
points of $\PP^{1,1,2}$ and their preimages in $S$. 
Denote 
$\PP_2^\circ$ the projectivized 
normal bundle to $D_2^\circ$ in 
$(\PP^2)^\circ$, coming with two natural sections $(D_2^\circ)_0$ and 
$(D_2^\circ)_\infty$. Denote 
$\tilde{\PP}_2^\circ$ the blow-up of
$\PP_2^\circ$ at the point $P_2 \in 
(D_2^\circ)_\infty$, $\tilde{E}_2$
the corresponding exceptional divisor and 
$C_2$ the strict transform of the fiber 
of $\PP_2^\circ$ passing through $P_2$. 
In particular, $\tilde{E}_2$ and 
$C_2$ are both projective lines with degree
$-1$ normal bundle in $(\tilde{\PP}_2)^\circ$. Furthermore, $\tilde{E}_2$  
and $C_2$ intersect in one point.

We degenerate
$S^\circ$ as in Section
5.3 of \cite{MR2667135}.
We first degenerate $(\PP^{1,1,2})^\circ$ to the normal cone of $D_2^\circ$, i.e.\ we blow-up $D_2^\circ \times \{0\}$ in 
$(\PP^{1,1,2})^{\circ} \times \C$.
The fiber over $0 \in \C$ has two components: 
$(\PP^{1,1,2})^\circ$
and $\PP_2^\circ$, with $\PP_2^\circ$ glued along 
$(D_2^\circ)_0$ 
to  
$D_2^\circ$ in $(\PP^{1,1,2})^\circ$.
We then blow-up the strict transform of the section $P_2 \times \C$.
The fiber of the resulting family away from
$0 \in \C$ is isomorphic to $S^\circ$.
The fiber over zero has two components: 
$(\PP^{1,1,2})^{\circ}$
and
$\tilde{\PP}_2^{\circ}$.

We would like to apply a degeneration formula to this family in order to compute $N_g^S$.
The key point is that for homological degree reasons, 
the relevant degenerating relative stable maps do not leave the non-compact geometries we are considering. More precisely, after fixing a point 
$P_1 \in D_1^\circ$, 
realizing the insertion 
$\ev_1^{*}(\mathrm{pt}_1)$, any limiting relative stable map has to factor through 
$ L \cup C_2$, with degree one over $L$ and degree two over $C_2$, 
where $L$ is the unique curve in 
$\PP^{1,1,2}$, of class $\beta_{\Delta}$, passing through $P_1$ and through 
$P_2$ with tangency order two along $D_2^{\circ}$.
So, even if the target geometry is non-compact, all the relevant moduli spaces of relative stable maps are compact.
It follows that
we can apply  
the ordinary
degeneration formula 
in relative Gromov-Witten theory
\cite{MR1938113}.

The application of the degeneration formula gives two terms, corresponding to the two partitions $2=1+1$ and $2=2$ of the 
intersection number
\[(\beta_\Delta -2E_2).E_2=2\,.\]
For the first term, the invariants on the side of 
$\PP^{1,1,2}$ are $N_{g,\tilde{\Delta}}^{1,2,2'}$, whereas on the side of $\tilde{\PP}_2$, we have disconnected invariants, corresponding to two degree one
maps to $C_2$.
As in the proof of Proposition
\ref{prop_mult1}, the relevant 
connected degree one invariants of $C_2$ are given by 
\[N_g^{C_2}= \int_{[\overline{M}_g(\PP^1/\infty, 
1,1)]^{\virt}}
e \left( R^1 \pi_* f^* (\cO_{\PP^1} \oplus \cO_{\PP^1} (-1)) \right) \,,\]
satisfying
\[ \sum_{g \geqslant 0} N_g^{C_2} u^{2g-1}
=\frac{1}{2 \sin \left( \frac{u}{2} \right)} \,.\]

For the second term, the invariants on the side of $\PP^{1,1,2}$ are $N_{g,\Delta}^{1,2}$, whereas on the side of $\tilde{\PP}_2$, we have connected invariants, corresponding to one degree two map to $C_2$.
More precisely, the relevant connected degree two invariants of 
$C_2$ are given by
\[N_g^{2C_2}= \int_{[\overline{M}_g(\PP^1/\infty, 
2,2)]^{\virt}}
e \left( R^1 \pi_* f^* (\cO_{\PP^1} \oplus \cO_{\PP^1} (-1)) \right) \,,\]
where $\overline{M}_g(\PP^1/\infty, 
2,2)$ is the moduli space of genus $g$ stable maps to $\PP^1$, of degree two, and relative to a point $\infty \in \PP^1$ with maximal tangency order two.
According to \cite{MR2115262}
(see the proof of Theorem 5.1), we have
\[\sum_{g \geqslant 0} N_g^{2C_2} u^{2g-1}
=-\frac{1}{2} \frac{1}{2 \sin (u)}\,.\]

It follows that the degeneration formula takes the form
\[\sum_{g \geqslant 0} N_g^S u^{2g-1} \]
\[= \frac{1}{2}
\left( \sum_{g \geqslant 0} N_{g,\tilde{\Delta}}^{1,2,2'}
u^{2g+2} \right)
\left( \frac{1}{2 \sin \left(
\frac{u}{2} \right)}
\right)^2
+2 \left(  \sum_{g \geqslant 0} 
N_{g, \Delta}^{1,2} 
u^{2g+1} \right)
\frac{(-1)}{2} \frac{1}{2 \sin(u)} \,.\]
The factor
$\frac{1}{2}$ in front of the fist term 
is a symmetry factor and the factor 
$2$ in front of the second term 
is a multiplicity.

There exists a unique tropical curve of 
type $\tilde{\Delta}$, which looks like

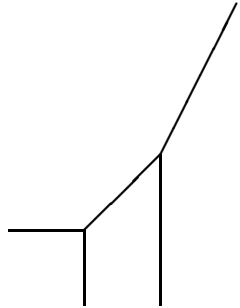
\begin{center}
\setlength{\unitlength}{1cm}
\begin{picture}(10,4)
\thicklines
\put(3,1){\line(1,0){1}}
\put(4,0){\line(0,1){1}}
\put(5,0){\line(0,1){2}}
\put(4,1){\line(1,1){1}}
\put(5,2){\line(1,2){1}}
\end{picture}
\end{center}

This tropical curve has two vertices of multiplicity one, so
using the gluing formula of
Proposition \ref{prop_gluing2}
and Proposition 
\ref{prop_mult1}, we find
\[\sum_{g \geqslant 0} N_{g,\tilde{\Delta}}^{1,2,2'}
u^{2g+2}=(F_1(u))^2 
= \left( 2 \sin \left( \frac{u}{2} \right)
\right)^2 \,.\]
Combining the previous results, we obtain 
\[ 0 = \frac{1}{2} - \frac{1}{2 \sin (u)}
\left( \sum_{g \geqslant 0} 
N_{g, \Delta}^{1,2} 
u^{2g+1} \right)\,, \]
and so the desired formula. 
\end{proof}

\textbf{Remark:} The proofs of Propositions
\ref{prop_mult1} and \ref{prop_mult2}
rely on the fact that the involved curves have low degree. More precisely, in each case, the key point is that the dual polygon does not contain any interior
integral point, i.e.\ a generic curve in the corresponding linear system on the surface has genus zero. This implies that, after imposing tangency constraints, all the higher genus stable maps factor through some rigid genus zero curve in the surface.
This guarantees the compactness 
result needed to work as we did with relative Gromov-Witten theory of non-compact geometries. The higher genus generalization of the most general case of the degeneration argument of Section
5.3 of \cite{MR2667135} cannot be dealt with in the same way. This generalization will be treated and applied in 
\cite{bousseau2018quantum_tropical},
using techniques similar to those used to prove the gluing formula in \mbox{Section
\ref{proof_gluing}}.

\subsection{Contribution of a general vertex}
\label{section_general_vertex}

\begin{prop} \label{prop_vertex}
The contribution of a vertex of 
multiplicity $m$ is given by 
\[F_m(u)=(-i) (q^{\frac{m}{2}}-q^{-\frac{m}{2}}).\]
\end{prop}

\begin{proof}
By Proposition 
\ref{prop_mult1}, the result is true for $m=1$ and 
by \mbox{Proposition \ref{prop_mult2}}, the result is true for $m=2$.
By consistency of the gluing formula
of Proposition 
\ref{prop_gluing2}, the
function $F(m) \coloneqq F_m$
valued in the ring 
$R \coloneqq \Q[\![u]\!]$ satisfies the 
hypotheses of 
Proposition \ref{prop_red_mult}.
The result follows by induction on $m$ using 
Proposition \ref{prop_red_mult}.
\end{proof}

The proof of Theorem
\ref{main_thm1}
(Theorem 
\ref{main_thm}
in Section \ref{section_refined_thm}) 
follows from the 
combination of Proposition
\ref{prop_gluing2},
Proposition 
\ref{prop_mult}
and Proposition 
\ref{prop_vertex}.

To prove Theorem \ref{thm_fixing_points},
generalizing Theorem \ref{main_thm1} by allowing to fix the positions of some of the 
intersection points with the toric boundary, we only have to organize the 
gluing procedure slightly differently.
The connected components of the complement
of the bivalent vertices of 
$\Gamma$, as at the beginning of
Section \ref{section_statement_gluing}, are trees with one unfixed unbounded edge and possibly several fixed unbounded edges.  
We fix an orientation of the edges such that edges adjacent to bivalent 
pointed vertices go out of the bivalent
pointed vertices, such that the fixed unbounded edges are ingoing and such that the unfixed unbounded edge is outgoing. With respect to this orientation,
every trivalent vertex has two ingoing and one outgoing edges, and so, without any modification, we obtain the analogue of 
the gluing formula of 
Corollary \ref{cor_gluing}: 
\[N_{g, h}^{\Delta, n}
= 
\left(
\prod_{V
\in V^{(3)}(\Gamma)} N_{g(V), V}^{1, 2} 
\right)
\left( 
 \prod_{E \in E_f(\Gamma)} w(E)
 \right) \,.\]
In Lemma \ref{lem_vertex}, we defined
$N_{g,V} \coloneqq N_{g(V), V}^{1, 2} w(E_V^{\text{in},1}) w(E_V^{\text{in},2})$, where $E_V^{\text{in},1}$ and $E_V^{\text{in},1}$ are the ingoing edges adjacent to $V$. Every bounded edge is an ingoing edge to some vertex but some ingoing edges are fixed unbounded edges and so
\[N_{g, h}^{\Delta, n}
=\left( 
\prod_{E_\infty^F \in E_\infty^F(\Gamma)} \frac{1}{w(E_\infty^F)} \right)
\left( \prod_{V
\in V^{(3)}(\Gamma)} N_{g(V), V} 
\right)\,,\]
where the first product is over the fixed
unbounded edges of $\Gamma$.
Theorem \ref{thm_fixing_points} then follows from Proposition \ref{prop_vertex}.

\appendix
\section{An explicit example} \label{section_ example}
In this appendix, 
we check by a direct computation one of the 
consequences of Theorem \ref{main_thm1}.
Let us consider the problem
of counting rational cubic curves 
in $\PP^2$ passing through $8$
points in general position. To match the notations of the Introduction, we choose 
$\Delta$ containing  three times the vector $(1,0)$, three times the vector 
$(0,1)$ and three times the vector $(-1,-1)$. 

The toric surface $X_\Delta$
is then $\PP^2$ and the curve class $\beta_\Delta$ is the class of a cubic curve 
in $\PP^2$. We have 
$|\Delta|=9$, $n=8$, $g_{\Delta,n}=0$. 
Let us write $N_g$ for $N_g^{\Delta ,n}$.
We have $N_0=12$ and the 
corresponding 
Block-G\"ottsche invariant is 
$q+10+q^{-1}$ 
(see Example 1.3 of \cite{nicaise2016tropical} for pictures of tropical curves).
From the point of view of G\"ottsche-Shende
\cite{MR3268777}, the relevant 
relative Hilbert scheme to consider happens to be the pencil 
of cubics passing through the 8 given points, i.e.\ $\PP^2$ blown-up in $9$ points, 
whose Hirzebruch genus is indeed $1+10q+q^2$.

According to Theorem \ref{main_thm},
we have 
\[\sum_{g \geqslant 0}
N_g u^{2g-2+9} = i(q+10+q^{-1})
(q^{\frac{1}{2}}-q^{-\frac{1}{2}})^7\]
\[=i(q^{\frac{9}{2}}
+3q^{\frac{7}{2}} 
-48q^{\frac{5}{2}}
+168q^{\frac{3}{2}}
-294q^{\frac{1}{2}}
+294q^{-\frac{1}{2}}
-168q^{-\frac{3}{2}}
+48q^{-\frac{5}{2}}
-3q^{-\frac{7}{2}}
-q^{-\frac{9}{2}})\]
\[=12 u^7 - \frac{9}{2}u^9+\frac{137}{160}u^{11}
-\frac{1253}{11520}u^{13}+\dots\]
We will check directly that 
$N_1=-\frac{9}{2}$. Remark that a Block-G\"ottsche 
invariant equal to $12$ rather than to 
$q+10+q^{-1}$ would lead to $N_1=-\frac{7}{2}$.
In particular, the value of $N_1$ is already sensitive to the choice of the correct refinement.

We have\footnote{A general choice of representative for 
$\lambda_1$ cuts out a locus in the moduli space made entirely of 
torically transverse stable maps. 
In particular, we do not have to worry about the difference between log and usual stable maps. 
A general form
of this argument is used in the proof of the
gluing formula in Section
\ref{proof_gluing}.}
\[ N_1=\int_{[\overline{M}_{1,8}(\PP^2,3)]^{\text{virt}}}
(-1)^1 \lambda_1 \prod_{j=1}^8 
\text{ev}_j^{*}(\mathrm{pt})\,,\]
where $\text{pt} \in A^2(\PP^2)$ is the class
of a point.
Introducing an extra marked point and using the divisor equation, one can write 
\[ N_1=\frac{1}{3}\int_{[\overline{M}_{1,8+1}(\PP^2,3)]^{\text{virt}}}
(-1)^1 \lambda_1 \left( 
\prod_{j=1}^8 
 \text{ev}_j^{*}(\mathrm{pt}) \right) \text{ev}_9^{*}(h) \,, \]
where $h \in A^1(\PP^2)$ is the class of a line.
On $\overline{M}_{1,1}$, we have 
\[\lambda_1 = \frac{1}{12} \delta_0 \,,\]
where $\delta_0$ is the class of a point. Taking for representative of $\delta_0$ the 
point corresponding to the nodal 
genus one curve, with $j$-invariant $i\infty$,
and resolving the node, we can write 
\[N_1=-\frac{1}{12}\cdot\frac{1}{2}\cdot\frac{1}{3} 
\int_{[\overline{M}_{0, 8+1+2}(\PP^2,3)]^{\virt}} \left( \prod_{j=1}^8 \text{ev}_j^{*}(\mathrm{pt})
\right) \text{ev}_9^{*}(h) (\text{ev}_{10}^{*} \times \text{ev}_{11}^{*})(D)\,,\]
where the factor $\frac{1}{2}$ comes from the two ways of labeling the two 
points resolving the node, and $D$ is the class of the diagonal in 
$\PP^2 \times \PP^2$. We have 
\[D=1 \times \mathrm{pt} + \mathrm{pt} \times 1 + h \times h\,.\]
The first two terms do not contribute to $N_1$ for dimension reasons so
\[N_1=-\frac{1}{12}\cdot\frac{1}{2}\cdot\frac{1}{3} 
\int_{[\overline{M}_{0, 8+1+2}(\PP^2,3)]^{\virt}} \left( \prod_{j=1}^8 \text{ev}_j^{*}(\mathrm{pt})
\right) \text{ev}_9^{*}(h) \text{ev}_{10}^{*}(h)\text{ev}_{11}^{*}(h)\,.\]
Using the divisor equation, we obtain 
\[N_1=-\frac{1}{12}\cdot\frac{1}{2}\cdot 3 \cdot 3 
\int_{[\overline{M}_{0, 8}(\PP^2,3)]^{\virt}} \left( \prod_{j=1}^8 \text{ev}_j^{*}(\mathrm{pt})
\right) = -\frac{9}{24} N_0 = 
- \frac{9}{2} \,,\]
as expected.

\vspace{+8 pt}
\noindent
Department of Mathematics \\
Imperial College London \\
pierrick.bousseau12@imperial.ac.uk

\end{document}